\documentclass[amsthm]{elsart}

\usepackage{amssymb,amsfonts,amsmath,times}

\theoremstyle{plain}
\newtheorem{theorem}{Theorem}[section]
\newtheorem{proposition}[theorem]{Proposition}
\newtheorem{lemma}[theorem]{Lemma}
\newtheorem{corollary}[theorem]{Corollary}

\newtheorem{remark}{Remark}[section]
\newtheorem{example}{Example}[section]
\newtheorem{definition}{Definition}[section]

\newcommand{\Div}{\mathop{\mathrm{div}}\nolimits}
\newcommand{\Ker}{\mathop{\mathrm{Ker}}\nolimits}
\newcommand{\Cross}{\mathop{\raisebox{-.17pc}{\large$\mathsf{X}$}}}
\newcommand{\ssCross}{\mathop{\raisebox{-.16pc}{\footnotesize $\mathsf{X}$}}}
\newcommand{\supp}{\mathop{\mathrm{supp}}\nolimits}
\newcommand{\Diff}{\mathrm{Diff}}
\newcommand{\Vect}{\mathrm{Vect}}
\newcommand{\const}{\mathrm{const}}
\newcommand{\rd}{\mathrm{d}}
\newcommand{\re}{\mathrm{e\myp}}

\newcommand{\CD}{\mathop{\raisebox{-.11pc}{\Large$\cdot$}}}

\newcommand{\RR}{\mathbb{R}}
\newcommand{\NN}{\mathbb{N}}
\newcommand{\ZZ}{\mathbb{Z}}

\newcommand{\calB}{\mathcal{B}}

\newcommand{\calN}{\mathcal{N}}

\newcommand{\frakX}{\mathfrak{X}}
\newcommand{\mucl}{\mu_{\mathrm{cl}}}

\newcommand{\myp}{\mbox{$\:\!$}}
\newcommand{\mypp}{\mbox{$\;\!$}}

\newcommand{\myn}{\mbox{$\;\!\!$}}
\newcommand{\mynn}{\mbox{$\:\!\!$}}

\makeatletter \@addtoreset {equation}{section}
\renewcommand{\theequation}{\arabic{section}.\arabic{equation}}
\makeatother

\begin{document}

\begin{frontmatter}

\title{\normalfont\LARGE
Poisson cluster measures: quasi-invariance, integration by parts and
equilibrium stochastic dynamics}
%\thanksref{t1}}

\runtitle{Poisson cluster measures on configuration spaces}

%\thanks[t1]{Research supported in part by DFG Grant 436 RUS
%113/722.}

\author[Leeds]{Leonid Bogachev}, \,%\corauthref{cor1}}
\ead{L.V.Bogachev@leeds.ac.uk}
\author[York]{Alexei Daletskii}
\ead{ad557@york.ac.uk}

\address[Leeds]{Department of Statistics, University of Leeds, Leeds LS2 9JT, UK}
\address[York]{Department of Mathematics, University of York, York YO10 5DD, UK}

\runauthor{L.\ Bogachev, A.\ Daletskii/Journal of Functional
Analysis}

%\corauth[cor1]{Corresponding author.}

\begin{abstract}
The distribution $\mucl$ of a Poisson cluster process in
$X=\mathbb{R}^{d}$ (with i.i.d.\ clusters) is studied via an
auxiliary Poisson measure on the space of configurations in
$\frakX=\bigsqcup_{n}\myn X^n$, with intensity measure defined as a
convolution of the background intensity of cluster centres and the
probability distribution of a generic cluster. We show that the
measure $\mucl$ is quasi-invariant with respect to the group of
compactly supported diffeomorphisms of $X$ and prove an
integration-by-parts formula for $\mucl$. The corresponding
equilibrium stochastic dynamics is then constructed using the method
of Dirichlet forms.
\end{abstract}

\begin{keyword}
% Redefine \MSC
\def\MSC{\par\leavevmode\hbox {\it 2000 MSC:\ }}%
cluster point process; Poisson measure; configuration space;
quasi-invariance; integration by parts; Dirichlet form; stochastic
dynamics

\medskip \MSC Primary 58J65; Secondary 31C25, 46G12, 60G55, 70F45
\end{keyword}

\end{frontmatter}

\section{Introduction}

In the mathematical modelling of multi-component stochastic systems,
it is conventional to describe their behaviour in terms of random
configurations of ``particles'' whose spatio-temporal dynamics is
driven by interaction of particles with each other and the
environment. Examples are ubiquitous and include various models in
statistical mechanics, quantum physics, astrophysics, chemical
physics, biology, computer science, economics, finance, etc. (see
\cite{DVJ1} and the extensive bibliography therein).

Initiated in statistical physics and theory of point processes, the
development of a general mathematical framework for suitable classes
of configurations was over decades a recurrent research theme
fostered by widespread applications. More recently, there has been a
boost of more specific interest in the \emph{analysis} and
\emph{geometry} of configuration spaces. In the seminal papers
\cite{AKR1,AKR2}, an approach was proposed to configuration spaces
as \emph{infinite-dimensional manifolds}. This is far from
straightforward, since configuration spaces are not vector spaces
and do not possess any natural structure of Hilbert or Banach
manifolds. However, many ``manifold-like'' structures can be
introduced, which appear to be nontrivial even in the Euclidean
case. We refer the reader to papers
\cite{ADKal,AKR2,AKR3,Liebscher,Ro} and references therein for
further discussion of various aspects of analysis on configuration
spaces and applications.

Historically, the approach in \cite{AKR1,AKR2} was motivated by the
theory of representations of diffeomorphism groups (see
\cite{GGPS,Ism,VGG}). To introduce some notation, let $\varGamma_X$
be the space of countable subsets (\emph{configurations}) without
accumulation points in a topological space $X$ (e.g., Euclidean
space $\mathbb{R}^d$). Any probability measure $\mu$ on
$\varGamma_{X}$, quasi-invariant with respect to the action of the
group $\Diff_{0}(X)$ of compactly supported diffeomorphisms of $X$
(lifted pointwise to transformations of $\varGamma_{X}$), generates
a canonical unitary representation of $\Diff_{0}(X)$ in
$L^{2}(\varGamma_{X},\mu )$. It has been proved in \cite{VGG} that
this representation is irreducible if and only if $\mu $ is
$\Diff_{0}(X)$-ergodic. Representations of such type are
instrumental in the general theory of representations of
diffeomorphism groups \cite{VGG} and in quantum field theory
\cite{GGPS,Goldin}.

According to a general paradigm described in \cite{AKR1,AKR2},
configuration space analysis is determined by the choice of a
suitable probability measure $\mu$ on $\varGamma_{X}$
(quasi-invariant with respect to $\Diff_{0}(X)$). It can be shown
that such a measure $\mu $ satisfies a certain integration-by-parts
formula, which enables one to construct, via the theory of Dirichlet
forms, the associated equilibrium dynamics (stochastic process) on
$\varGamma_{X}$ such that $\mu$ is its invariant measure
\myp\cite{AKR1,AKR2,MR}. In turn, the equilibrium process plays an
important role in the asymptotic analysis of statistical-mechanical
systems whose spatial distribution is controlled by the measure
$\mu$; for instance, this process is a natural candidate for being
an asymptotic ``attractor'' for motions started from a perturbed
(non-equilibrium) configuration.

This programme has been successfully implemented in \cite{AKR1} for
the Poisson measure, which is the simplest and most well-studied
example of a $\Diff_{0}(X)$-quasi-invariant measure on
$\varGamma_{X}$, and in \cite{AKR2} for a wider class of Gibbs
measures, which appear in statistical mechanics of classical
continuous gases. In particular, it has been shown that in the
Poisson case, the equilibrium dynamics amounts to the well-known
independent particle process, that is, an infinite family of
independent (distorted) Brownian motions started at the points of a
random Poisson configuration. In the Gibbsian case, the dynamics is
much more complex due to interaction between the particles.

The Gibbsian class (containing the Poisson measure as a simple
``interaction-free'' case) is essentially the sole example so far
that has been fully amenable to such analysis. In the present paper,
our aim is to develop a similar framework for a different class of
random spatial structures, namely the well-known \emph{cluster point
processes} (see, e.g., \cite{CI,DVJ1}). Cluster process is a simple
model to describe effects of grouping (``clustering'') in a sample
configuration. The intuitive idea is to assume that the random
configuration has a hierarchical structure, whereby independent
clusters of points are distributed around a certain (random)
configuration of invisible ``centres''. The simplest model of such a
kind is the \emph{Poisson cluster process}, obtained by choosing a
Poisson point process as the background configuration of the cluster
centres.

Cluster models have been very popular in numerous practical
applications ranging from neurophysiology (nerve impulses) and
ecology (spatial distribution of offspring around the parents) to
seismology (statistics of earthquakes) and cosmology (formation of
constellations and galaxies). More recent examples include
applications to trapping models of diffusion-limited reactions in
chemical kinetics \cite{AB,BMBM,BMMB}, where clusterization may
arise due to binding of traps to a substrate (e.g., a polymer chain)
or trap generation (e.g., by radiation damage). An exciting range of
new applications in physics and biology is related to the dynamics
of clusters consisting of a few to hundreds of atoms or molecules.
Investigation of such ``mesoscopic'' structures, intermediate
between bulk matter and individual atoms or molecules, is of
paramount importance in the modern nanoscience and nanotechnology
(for an authoritative account of the state of the art in this area,
see a recent review \cite{CBZ} and further references therein).

In the present work, we consider Poisson cluster processes in
$X=\mathbb{R}^{d}$. We prove the $\Diff_{0}(X)$-quasi-invariance of
the Poisson cluster measure $\mucl$ and establish the
integration-by-parts formula. We then construct an associated
Dirichlet form, which implies in a standard way the existence of
equilibrium stochastic dynamics on the configuration space
$\varGamma_{X}$. Our technique is based on the representation of
$\mucl$ as a natural ``projection'' image of a certain Poisson
measure on an auxiliary configuration space $\varGamma_{\frakX}$
over a disjoint union $\frakX=\bigsqcup_{n}\myn X^{n}$, comprising
configurations of ``droplets'' representing individual clusters of
variable size. A suitable intensity measure on $\frakX$ is obtained
as a convolution of the background intensity $\lambda(\rd x)$ (of
cluster centres) with the probability distribution
$\eta(\rd\bar{y})$ of a generic cluster. This approach enables one
to apply the well-developed apparatus of Poisson measures to the
study of the Poisson cluster measure $\mucl$.

Let us point out that the projection construction of the Poisson
cluster measure is very general, and in particular it works even in
the case when ``generalized'' configurations (with possible
accumulation or multiple points) are allowed. However, to be able to
construct a well-defined differentiable structure on cluster
configurations, we need to restrict ourselves to the space
$\varGamma_X$ of ``proper'' (i.e., locally finite and simple)
configurations. Using the technique of Laplace functionals, we
obtain necessary and sufficient conditions of almost sure (a.s.)
properness for Poisson cluster configurations, set out in terms of
the background intensity $\lambda(\rd x)$ of cluster centres and the
in-cluster distribution $\eta(\rd\bar{y})$. To the best of our
knowledge, these conditions appear to be new (cf., e.g.,
\cite[\S\,6.3]{DVJ1})
%\cite[Lemma~6.3.II \& Propositopn~6.3.III]{DVJ1})
and may be of interest for the general theory of cluster point
processes.

Some of the results of this paper have been sketched in \cite{BD1}
(in the case of clusters of fixed size). We anticipate that the
projection approach developed in the present paper can be applied to
the study of more general cluster measures on configurations spaces,
especially Gibbs cluster measure (see \cite{BD2} for the case of
fixed-size clusters). Such models, and related functional-analytic
issues, will be addressed in our future work.

The paper is organized as follows. In Section \ref{sec:General}, we
set out a general framework of probability measures in the space of
generalized configurations $\varGamma^\sharp_X$. In Section
\ref{sec:Poisson}, we recall the definition and discuss the
construction and some basic properties of the Poisson measure on the
space $\varGamma^\sharp_X$, while Section \ref{sec:Poisson-cluster}
goes on to describe the Poisson cluster measure. In Section
\ref{sec:proper}, we discuss criteria for Poisson cluster
configurations to be a.s.\ locally finite and simple
(Theorem~\ref{th:properClusterPoisson}, the proof of which is
deferred to the Appendix). An auxiliary intensity measure
$\lambda^\star$ on the space $\mathfrak{X}=\bigsqcup_n\mynn X^n$ is
introduced and discussed in Section~\ref{n-clusters}, which allows
us to define the corresponding Poisson measure $\pi_{\lambda^\star}$
on the configuration space $\varGamma_\frakX^\sharp$
(Section~\ref{sec:3.2}). Theorem~\ref{th:mucl} of Section
\ref{sec:PCP} shows that the Poisson cluster measure $\mucl$ can be
obtained as a push-forward of the Poisson measure
$\pi_{\lambda^\star}$ on $\varGamma_{\mathfrak{X}}^\sharp$ under the
``unpacking'' map $\frakX\ni \bar{x}\mapsto
\mathfrak{p}(\bar{x}):=\bigsqcup_{x_i\in\bar{x}}\{x_i\}\in\varGamma_\frakX^\sharp$.
In Section \ref{sec:3.3}, we describe a more general construction of
$\mucl$ using another Poisson measure defined on the space
$\varGamma^\sharp_{X\times \frakX}$ of configurations of pairs
$(x,\bar{y})$ ($x=\text{cluster centre}$, $\bar{y}=\text{in-cluster
configuration}$), with the product intensity measure $\lambda(\rd
x)\otimes\eta(\rd \bar{y})$. Following a brief compendium on
differentiable functions in configuration spaces
(Section~\ref{app2}), Section \ref{sec:QI-mu} deals with the
property of quasi-invariance of the measure $\mucl$ with respect to
the diffeomorphism group $\Diff_0(X)$ (Theorem~\ref{q-inv}). Further
on, an integration-by-parts formula for $\mucl$ is established in
Section~\ref{sec:4.2} (Theorem~\ref{IBP-}). The Dirichlet form
$\mathcal{E}_{\mucl}$ associated with $\mucl$ is defined and studied
in Section \ref{sec:Dir-mu}, which enables us to construct in
Section~\ref{sec:equil} the canonical equilibrium dynamics (i.e.,
diffusion on the configuration space with invariant measure
$\mucl$). In addition, we show that the form $\mathcal{E}_{\mucl}$
is irreducible (Theorem~\ref{th:irr}, Section~\ref{sec:irreduc}).
Finally, the Appendix includes the proof of
Theorem~\ref{th:properClusterPoisson} (Section~\ref{app1}) and the
proof of a well-known general result on quasi-invariance of Poisson
measures, adapted to our purposes (Section~\ref{app3}).

\section{Poisson and Poisson cluster measures in configuration
spaces}\label{sec:PP}

In this section, we fix some notations and describe the setting of
configuration spaces that we shall use. As compared to a standard
exposition (see, e.g., \cite{CI,DVJ1}), we adopt a more general
standpoint by allowing configurations with multiple points and/or
accumulation points. With this modification in mind, we recall the
definition and some properties of Poisson point process (as a
probability measure in the generalized configuration space
$\varGamma_{X}^\sharp$). We then proceed to introduce the main
object of the paper, the cluster Poisson point process and the
corresponding measure $\mucl$ in $\varGamma_{X}^\sharp$. The central
result of this section is the projection constriction showing that
$\mucl$ can be obtained as a push-forward of a suitable Poisson
measure in the auxiliary ``vector'' configuration space
$\varGamma_{\frakX}^\sharp$, where $\frakX=\bigsqcup_n X^n$.

\subsection{Generalized configurations}\label{sec:General}

Let $X$ be a Polish space (i.e., separable completely metrizable
topological space), equipped with the Borel $\sigma$-algebra
$\mathcal{B}(X)$ generated by the open sets. Denote
$\overline{\ZZ}_+:=\ZZ_+\cup \{\infty\}$, where
$\ZZ_+=\{0,1,2,\dots\}$, and consider the space $\frakX$ built from
Cartesian powers of $X$, that is, a disjoint union
$\frakX:=\bigsqcup_{n\in \overline{\ZZ}_+}\myn X^{n}$ including
$X^0=\{\emptyset\}$ and the space $X^\infty$ of infinite sequences
$(x_1,x_2,\dots)$. That is to say, $\bar x=(x_1,x_2,\dots)\in\frakX$
if and only if $\bar{x}\in X^n$ for some $n\in\overline{\ZZ}_+$. For
simplicity of notation, we take the liberty to write $x_i\in\bar x$
if $x_i$ is a coordinate of the vector $\bar x$.

Each space $X^n$ is equipped with the product topology induced by
$X$, that is, the coarsest topology in which all coordinate
projections $(x_1,\dots,x_n)\mapsto x_i$ are continuous
($i=1,\dots,n$). Hence, the space $\frakX$ is endowed with the
natural disjoint union topology, that is, the finest topology in
which the canonical injections \,$\mathbf{j}_{\myp n}:X^{n}\to
\frakX$ are continuous ($n\in\overline{\ZZ}_+$). In other words, a
set ${U}\subset \frakX$ is open in this topology whenever
${U}=\bigsqcup_{n\in\overline{\ZZ}_+}\myn U_n$, where each $U_n$ is
an open subset in $X^n$ ($n\in\overline{\ZZ}_+$). Hence, the Borel
$\sigma$-algebra on $\frakX$ is given by
$\calB(\frakX)=\bigoplus_{n\in\overline{\ZZ}_+} \calB(X^n)$, that
is, consists of sets of the form
${B}=\bigsqcup_{n\in\overline{\ZZ}_+}\myn B_n$, where
\,$B_n\in\calB(X^n)$, \,$n\in\overline{\ZZ}_+$\myp.

\begin{remark}\label{rm:compact}
\normalfont Note that a set $K\subset\frakX$ is compact if and only
if $K=\bigsqcup_{n=0}^N\myn K_n$, where $N<\infty$ and $K_n$ are
compact subsets of $X^n$, respectively. This becomes clear by
considering an open cover of $K$ by the sets ${U}_n=X^n$,
\,$n\in\overline{\ZZ}_+$.
\end{remark}

Denote by ${\calN}(X)$ the space of $\overline{\ZZ}_+$-valued
measures $N(\cdot)$ on $\mathcal{B}(X)$ with countable (i.e., finite
or countably infinite) support $\supp N:=\{x\in X: N\{x\}>0\}$ (here
and below, we use $N\{x\}$ as a shorthand for a more accurate
$N(\{x\})$; the same convention applies to other measures). Consider
the natural projection
\begin{equation}\label{eq:pr0}
\frakX\ni \bar{x} \mapsto
\mathfrak{p}(\bar{x}):=\sum_{x_i\in\bar{x}}\delta_{x_i}\in
\mathcal{N}(X),
\end{equation}
where $\delta_x$ is Dirac measure at point $x\in X$. Gathering any
coinciding points $x_i\in\bar{x}$, the measure
$N=\sum_{x_i\in\bar{x}}\delta_{x_i}$ in (\ref{eq:pr0}) can be
written down as $N=\sum_{x_i^*\in\supp N} k_i\myp\delta_{x_i^*}$,
where $k_i=N\{x_i^*\}>0$ is the ``multiplicity'' (possibly infinite)
of the point $x_i^*\in\supp N$. Any such measure $N$ can be
conveniently associated with a \emph{generalized configuration}
$\gamma$ of points in $X$,
\begin{equation*}
N\leftrightarrow \gamma:= \bigsqcup_{x_i^*\in\supp N}
\underbrace{\{x_i^*\}\sqcup\cdots\sqcup \{x_i^*\}}_{k_i}\,,
\end{equation*}
where the disjoint union $\{x^*\}\sqcup\cdots\sqcup \{x^*\}$
signifies the inclusion of several distinct copies of point
$x^*\in\supp N$. Thus, the mapping (\ref{eq:pr0}) can be
symbolically rewritten as
\begin{equation}\label{eq:pr}
{\mathfrak p}(\bar{x})=\gamma:=\bigsqcup_{x_i\in \bar{x}}
\{x_i\},\qquad \bar{x}=(x_1,x_2,\dots)\in\frakX.
\end{equation}
That is to say, under the projection mapping $\mathfrak{p}$ each
vector from $\frakX$ is ``unpacked'' into distinct components,
resulting in a countable aggregate of points in $X$ (with possible
multiple points), which we interpret as a generalized configuration
$\gamma$. Note that, formally, $\bar{x}$ may be from the ``trivial''
component $X^0=\{\emptyset\}$, in which case the union in
(\ref{eq:pr}) (as well as the sum in~(\ref{eq:pr0})) is vacuous and
hence corresponds to the empty configuration, $\gamma=\emptyset$.

Even though generalized configurations are not, strictly speaking,
subsets of $X$ (due to possible multiple points), it is convenient
to keep using set-theoretic notations, which should not cause any
confusion. For instance, we write $\gamma_B:=\gamma\cap B$ for the
restriction of configuration $\gamma$ to a subset
$B\in\mathcal{B}(X)$. Similarly, for a function $f:X\to\RR$ we
denote
\begin{equation}\label{eq:f-gamma}
\langle f,\gamma \rangle :=\sum_{x_i\in \gamma }f(x_i)\equiv
\sum_{x_i^*\in\supp N} N\{x_i^*\}\,f(x_i^*) =\int_{X}f(x)\, N(\rd
x).
\end{equation}
This formula motivates the following convention that will be used
throughout: if $\gamma=\emptyset$ then $\sum_{x\in\gamma} f(x):= 0$.

In what follows, we shall identify generalized configurations
$\gamma$ with the corresponding measures
$N=\sum_{x_i\in\gamma}\myn\delta_{x_i}$, and we shall opt to
interpret the notation $\gamma$ either as an aggregate of (multiple)
points in $X$ or as a $\overline{\ZZ}_+$-valued measure or both,
depending on the context. For example, if ${\bf 1}_B(x)$ is the
indicator function of a set $B\in\mathcal{B}(X)$ then $\langle{\bf
1}_B,\gamma\rangle=\gamma(B)$ is the total number of points (counted
with their multiplicities) in the restriction $\gamma_B$ of the
configuration $\gamma$ to~$B$.

\begin{definition}\label{def:gen}
\normalfont \emph{Configuration space} $\varGamma_{X}^\sharp$ is the
set of generalized configurations $\gamma$ in $X$, endowed with the
\emph{cylinder $\sigma$-algebra} $\calB(\varGamma_{X}^\sharp)$
generated by the class of cylinder sets $C_{B}^{\myp
n}:=\{\gamma\in\varGamma_{X}^\sharp: \gamma(B)=n\}$,
\,$B\in\calB(X)$, \,$n\in\ZZ_+$\myp.
\end{definition}
\begin{remark}\normalfont
Note that the set $C_B^{\infty}=\{\gamma\in\varGamma_X^\sharp:
\gamma(B)=\infty\}$ is measurable:
$$
C_B^\infty=\bigcap_{n=0}^\infty \{\gamma\in \varGamma_{X}^\sharp:
\gamma(B)\ge n\}=\bigcap_{n=0}^\infty\bigcup_{k=n}^\infty C_B^{\myp
k}\in\mathcal{B}(\varGamma_{X}^\sharp).
$$
\end{remark}

The mapping $\mathfrak{p}:\frakX\to\varGamma_X^\sharp$ defined by
formula (\ref{eq:pr}) is measurable, since for any cylinder set
$C_{B}^{\myp n}\in\calB(\varGamma_{X}^\sharp)$ we have
\begin{equation}\label{eq:DD}
\mathfrak{p}^{-1}(C_{B}^{\myp n})= D_{B}^{\myp
n}:=\left\{\bar{x}\in\frakX: \,\sum_{x_i\in\bar{x}}\myn{\bf
1}_B(x_i)=n\right\}\in\mathcal{B}(\frakX).
\end{equation}

As already mentioned, conventional theory of point processes (and
their distributions as probability measures on configuration spaces)
usually rules out the possibility of accumulation points or multiple
points (see, e.g., \cite{DVJ1}).
\begin{definition}\label{def:proper}
\normalfont Configuration $\gamma\in\varGamma _{X}^\sharp$ is said
to be \emph{locally finite} if $\gamma(K)<\infty$ for any compact
set $K\subset X$. Configuration $\gamma\in\varGamma _{X}^\sharp$ is
called \emph{simple} if $\gamma\{x\}\le 1$ for each $x\in X$.
Configuration $\gamma\in\varGamma _{X}^\sharp$ is called
\emph{proper} if it is both locally finite and simple. The set of
proper configurations will be denoted by $\varGamma_{X}$ and called
the \emph{proper configuration space} over $X$. The corresponding
$\sigma$-algebra $\calB(\varGamma_{X})$ is generated by the cylinder
sets $\{\gamma\in\varGamma_{X}: \gamma(B)=n\}$ \,($B\in\calB(X)$,
\,$n\in\ZZ_+$).
\end{definition}

Like in the standard theory for proper configuration spaces (see,
e.g., \cite[\S\,6.1]{DVJ1}), every measure $\mu$ on the generalized
configuration space $\varGamma_{X}^\sharp$ can be characterized by
its Laplace functional
\begin{equation}\label{eq:LAPLACE}
L_{\mu}[f]:=\int_{\varGamma_{X}^\sharp}\re^{-\langle f,\gamma\rangle
}\,\mu(\rd\gamma ),\qquad f\in\mathrm{M}_+(X),
\end{equation}
where $\mathrm{M}_+(X)$ is the set of measurable non-negative
functions on $X$ (so that the integral in (\ref{eq:LAPLACE}) is well
defined since $0\le \re^{-\langle f,\gamma\rangle }\le 1$). To see
why $L_{\mu}[\cdot]$ completely determines the measure $\mu$ on
$\calB(\varGamma_X^\sharp)$, note that if $B\in\mathcal{B}(X)$ then
$L_\mu[s{\bf 1}_{B}]$ as a function of $s>0$ gives the
Laplace--Stieltjes transform of the distribution of the random
variable $\gamma(B)$ and as such determines the values of the
measure $\mu$ on the cylinder sets $C_B^{\myp
n}\in\calB(\varGamma_X^\sharp)$ ($n\in\ZZ_+$). In particular,
$L_\mu[s{\bf 1}_{B}]=0$ if and only if $\gamma(B)=\infty$
($\mu$-a.s.). Similarly, using linear combinations $\sum_{i=1}^k
s_i{\bf 1}_{B_i}$ we can recover the values of $\mu$ on the cylinder
sets
$$
C_{B_1,\dots,B_k}^{\myp n_1,\dots,n_k}:=\bigcap_{i=1}^k
C_{B_i}^{\myp n_i}=\{\gamma\in \varGamma_{X}^\sharp:
\gamma(B_i)=n_i,\ i=1,\dots,k\}
$$
and hence on the ring $\mathcal{C}(X)$ of finite disjoint unions of
such sets. Since the ring $\mathcal{C}(X)$ generates the cylinder
$\sigma$-algebra $\calB(\varGamma_X^\sharp)$, the extension theorem
(see, e.g., \cite[\S\,13, Theorem~A]{Halmos} or
\cite[Theorem~A1.3.III]{DVJ1}) ensures that the measure $\mu$ on
$\calB(\varGamma_X^\sharp)$ is determined uniquely.

\subsection{Poisson measure}\label{sec:Poisson}
We recall here some basic facts about Poisson measures in
configuration spaces. As compared to the customary treatment,
another difference, apart from working in the space of generalized
configurations $\varGamma_X^\sharp$, is that we use a
$\sigma$-\emph{finite} intensity measure rather than a \emph{locally
finite} one.

Poisson measure on the configuration space $\varGamma_X^\sharp$ is
defined descriptively as follows (cf.\ \cite[\S\,2.4]{DVJ1}).
\begin{definition}\label{def:Poisson}
\normalfont Let $\lambda$ be a $\sigma$-finite measure in
$(X,\calB(X))$ (not necessarily infinite, i.e.,
$\lambda(X)\le\infty$). The \emph{Poisson measure $\pi_\lambda$ with
intensity $\lambda$} is a probability measure on
$\calB(\varGamma_{X}^\sharp)$ satisfying the following condition:
for any disjoint sets $B_1,\dots,B_k\in\calB(X)$ (i.e., $B_i\cap
B_j=\emptyset$ for $i\ne j$), such that $\lambda(B_i)<\infty$
($i=1,\dots,k$), and any $n_1,\dots,n_k\in\ZZ_+$, the value of
$\pi_\lambda$ on the cylinder set $ C_{B_1,\dots,B_k}^{\myp
n_1,\dots,n_k}$ is given by
\begin{equation}\label{eq:PoissonFD}
\pi_\lambda\bigl(C_{B_1,\dots,B_k}^{\myp
n_1,\dots,n_k}\bigr)=\prod_{i=1}^k \frac{\lambda(B_i)^{n_i}
\,\re^{-\lambda(B_i)}}{n_i!}
\end{equation}
(with the convention $0^0:=1$). That is, for disjoint sets $B_i$ the
values $\gamma(B_i)$ are mutually independent Poisson random
variables with parameters $\lambda(B_i)$, respectively.
\end{definition}

A well-known ``explicit'' construction of the Poisson measure
$\pi_\lambda$ is as follows (cf.\ \cite{AKR1,Sh}). For a fixed set
$\varLambda\in\calB(X)$ such that $\lambda(\varLambda)<\infty$,
consider the restriction mapping $p_\varLambda$\myp,
$$
\varGamma_X^\sharp\ni \gamma\mapsto p_\varLambda
\gamma=\gamma\cap\varLambda\equiv
\gamma_\varLambda\in\varGamma_\varLambda^\sharp.
$$
Clearly, $p_\varLambda(C_\varLambda^{\myp
n})=\{\tilde\gamma\in\varGamma_\varLambda^\sharp:
\tilde\gamma(\varLambda)=n\}$. For
$A\in\calB(\varGamma_{\varLambda}^\sharp)$ and $n\in\ZZ_+$, let
$A_{\varLambda,\myp n}:=A\cap p_\varLambda(C_{\varLambda}^{\myp
n})\in\calB(\varGamma_{\varLambda}^\sharp)$ and define the measure
\begin{equation}\label{eq:P-Lambda}
\pi_\lambda^\varLambda(A):=\re^{-\lambda(\varLambda)}\sum_{n=0}^\infty
\frac{1}{n!}\,\lambda^{\otimes\myp n}\mynn\circ \mathfrak{p}^{-1}
(A_{\varLambda,\myp n}), \qquad A\in
\calB(\varGamma_{\varLambda}^\sharp),
\end{equation}
where $\lambda^{\otimes\myp
n}=\underbrace{\,\lambda\otimes\dots\otimes\lambda}_{n}$ \, is the
product measure in $(X^n,\calB(X^n))$ (we formally set
$\lambda^{\otimes\myp 0}\!:=\delta_{\{\emptyset\}}$) and ${\mathfrak
p}$ is the projection operator defined in (\ref{eq:pr}). In
particular, (\ref{eq:P-Lambda}) implies that
$\pi_\lambda^\varLambda$ is a probability measure on
$\varGamma_{\varLambda}^\sharp$. It is easy to check that the
``cylindrical'' measure $\pi_\lambda^\varLambda\circ p_\varLambda$
in $\varGamma_X^\sharp$ (in fact, supported on
$\bigcup_{n=0}^{\infty}\myn C_\varLambda^{\myp n}$) satisfies
equation (\ref{eq:PoissonFD}) for any disjoint Borel sets
$B_i\subset \varLambda$. It is also clear that the family
$\{\pi_\lambda^{\varLambda},\,\varLambda\subset X\}$ is consistent,
that is, the restriction of the measure $\pi_\lambda^{\varLambda}$
to a smaller configuration space $\varGamma_{\varLambda'}^\sharp$
(with $\varLambda'\subset\varLambda$) coincides with
$\pi_\lambda^{\varLambda'}$, that is, $\pi_\lambda^{\varLambda}\circ
(p_\varLambda p_{\varLambda'}^{-1})=\pi_\lambda^{\varLambda'}$.

Existence (and uniqueness) of a measure $\pi_\lambda$ in
$(\varGamma_X^\sharp,\calB(\varGamma_X^\sharp))$ such that, for any
$\varLambda\in\calB(X)$, the push-forward measure
$p_\varLambda^*\myp\pi_\lambda\equiv\pi_\lambda\circ
p_\varLambda^{-1}$ coincides with $\pi_\lambda^\varLambda$ (which
implies that $\pi_\lambda$ satisfies Definition \ref{def:Poisson}
and is therefore a Poisson measure on the configuration space
$\varGamma_X^\sharp$), now follows by a projective version of the
fundamental Kolmogorov extension theorem (see, e.g.,
\cite[\S\,A1.5]{DVJ1} or
%Theorem A1.5.IV]{DVJ1}, \cite[Ch.~5]{Kal}, Theorem 5.16]{Kal}
\cite[Ch.~5]{Par}). More precisely, recall that the measure
$\lambda$ on $X$ is $\sigma$-finite, hence there is a countable
family of sets $B_k\in\mathcal{B}(X)$ such that
$\lambda(B_k)<\infty$ and $\bigcup_{k=1}^{\infty}\myn B_k=X$. Then
$\varLambda_m:=\bigcup_{k=1}^{m}\myn B_k\in\mathcal{B}(X)$
($m\in\NN$) is a monotone increasing sequence of sets such that
$\lambda(\varLambda_m)<\infty$ and
$\bigcup_{m=1}^{\infty}\myn\varLambda_m=X$. By the construction
(\ref{eq:P-Lambda}), we obtain a consistent family of probability
measures $\pi_\lambda^{\varLambda_m}$ on the configuration spaces
$\varGamma_{\varLambda_m}^\sharp$, respectively. Using the metric in
$X$ (which is assumed to be a Polish space, see
Section~\ref{sec:General}), one can define a suitable distance
between finite configurations in each space
$\varGamma_{\varLambda_m}^\sharp$ and thus convert
$\varGamma_{\varLambda_m}^\sharp$ into a Polish space (see
\cite{Sh}), which ensures that the Kolmogorov extension theorem is
applicable.

\begin{remark}\normalfont
Even though the paper \cite{Sh} deals with simple configurations
only, its methods may be easily extended to a more general case of
configurations with multiple points. However, finiteness of
configurations in each $\varLambda_m$ is essential.
\end{remark}

\begin{remark}\normalfont
The requirement that $X$ is a Polish space (see
Section~\ref{sec:General}) is only needed in order to equip the
spaces of finite configurations in the sets $\varLambda_m$ with the
structure of a Polish space and thus to be able to apply the
Kolmogorov extension theorem as explained above (see \cite{Sh}).
This assumption may be replaced by a more general condition that
$(X,\mathcal{B}(X))$ is a standard Borel space (i.e., Borel
isomorphic to a Borel subset of a Polish space, see \cite{Kal,Par}).
\end{remark}

\begin{remark}\label{rm:Lambda}
\normalfont Formula (\ref{eq:P-Lambda}), rewritten in the form
\begin{equation*}
\pi_\lambda^\varLambda(A)=\sum_{n=0}^\infty
\frac{\lambda(\varLambda)^n
\,\re^{-\lambda(\varLambda)}}{n!}\cdot\frac{\lambda^{\otimes\myp n}
\mynn\circ \mathfrak{p}^{-1}(A_{\varLambda,\myp
n})}{\lambda(\varLambda)^n}\,,
\end{equation*}
gives an explicit way of sampling a Poisson configuration
$\gamma_\varLambda$ in the set $\varLambda$: first, a random value
of $\gamma(\varLambda)$ is sampled as a Poisson random variable with
parameter $\lambda(\varLambda)<\infty$, and then, conditioned on the
event $\{\gamma(\varLambda)=n\}$ \,($n\in\ZZ_+$), the $n$ points are
distributed over $\varLambda$ independently of each other, with
probability distribution $\lambda(\rd x)/\lambda(\varLambda)$ each
(cf.\ \cite[\S\.2.4]{Kingman}).
\end{remark}

Decomposition (\ref{eq:P-Lambda}) implies that if $F(\gamma )\equiv
F(\gamma_\varLambda)$ for some set $\varLambda\in\calB(X)$ such that
$\lambda(\varLambda)<\infty$, then
\begin{align}
\notag \int_{\varGamma^\sharp_{X}}F(\gamma
)\,\pi_{\lambda}(\rd\gamma )&
=\int_{\varGamma_{X}^\sharp}F(p_\varLambda\gamma)\,\pi_{\lambda}(\rd\gamma)
=\int_{\varGamma_{\varLambda}^\sharp}F(\gamma
)\,\pi_{\lambda}^\varLambda(\rd\gamma)\\ \label{3.1} &=\re^{-\lambda
(\varLambda)}\sum_{n=0}^{\infty }\frac{1}{n!}\int_{
\varLambda^{n}}F(\{x_{1},\dots ,x_{n}\})\,\lambda (\rd x_{1})\cdots
\lambda (\rd x_{n}).
\end{align}

A well-known formula for the Laplace functional of a Poisson point
process without accumulation points (see, e.g., \cite{AKR1,DVJ1}])
is easily verified in the case of generalized configurations.
\begin{proposition}\label{pr:PoissonLT}
The Laplace functional
$L_{\pi_\lambda}[f]:=\int_{\varGamma_{X}^\sharp}\re^{-\langle
f,\gamma\rangle }\,\pi_{\lambda}(\rd\gamma )$ of the Poisson measure
$\pi_\lambda$ on the configuration space $\varGamma_X^\sharp$ is
given by
\begin{equation}\label{eq:PoissonLT}
L_{\pi_\lambda}[f]=\exp\left\{-\int_{X}\left(1-\re^{-f(x)}\right)\lambda
(\rd x)\right\},\qquad f\in\mathrm{M}_+(X).
\end{equation}
\end{proposition}
\proof Repeating a standard derivation, suppose that
$\lambda(\varLambda)<\infty$ and set $f_\varLambda:=f\cdot{\bf
1}_\varLambda$. Applying formula (\ref{3.1}) we have
\begin{align}
\notag\int_{\varGamma_{X}^\sharp}\re^{-\langle
f_\varLambda,\gamma\rangle}\,\pi_{\lambda}(\rd\gamma)
&=\re^{-\lambda (\varLambda)}\sum_{n=0}^{\infty }\frac{1}{n!}\int_{
\varLambda^{n}} \exp\left\{-\sum_{i=1}^n f_\varLambda(x_{i})\right\}
\lambda (\rd
x_{1})\cdots \lambda (\rd x_{n})\\
\notag
&=\re^{-\lambda (\varLambda)}\sum_{n=0}^{\infty
}\frac{1}{n!}\left(\int_{\varLambda} \re^{-f_\varLambda(x)}
\,\lambda (\rd x)\right)^n\\
 \label{eq:Poisson-Lambda}
&=\exp\left\{-\int_{X}\left(1-\re^{-f_\varLambda(x)}\right)\lambda
(\rd x)\right\}.
\end{align}
Since $f_\varLambda(x)\uparrow f(x)$ as $\varLambda \uparrow X$
(more precisely, setting $\varLambda=\varLambda_m$ as in the above
construction of $\pi_\lambda$
 and passing to the limit as $m\to\infty$), by applying the
monotone convergence theorem to both sides of
(\ref{eq:Poisson-Lambda}) we obtain (\ref{eq:PoissonLT}).
\endproof

Formula (\ref{eq:PoissonFD}) implies that if $B_1\cap B_2=\emptyset$
then the restricted configurations $\gamma_{B_1}$ and $\gamma_{B_2}$
are independent under the Poisson measure $\pi_\lambda$. That is, if
$B:=B_1\cup B_2$ then the distribution
$\pi_\lambda^B=p_B^*\myp\pi_\lambda$ of composite configurations
$\gamma_B=\gamma_{B_1}\sqcup \gamma_{B_2}$ coincides with the
product measure $\pi_\lambda^{B_1}\!\otimes \pi_\lambda^{B_2}$
($\pi_\lambda^{B_i}=p_{B_i}^*\myp\pi_\lambda$). Building on this
observation, we obtain the following useful result.
\begin{proposition}\label{pr:product}
Suppose that $(X_n,\calB(X_n))$ $(n\in\NN)$ is a family of disjoint
measurable spaces \textup{(}i.e., $X_i\cap X_j=\emptyset$, \,$i\ne
j$\textup{)}, with measures $\lambda_n$, respectively, and let\/
$\pi_{\lambda_n}$ be the corresponding Poisson measures on the
configuration spaces $\varGamma_{X_n}^\sharp$ $(n\in\NN)$. Consider
the disjoint-union space $X=\bigsqcup_{n=1}^{\infty}\myn X_n$
endowed with the $\sigma$-algebra
$\calB(X)=\bigoplus_{n=1}^{\infty}\calB(X_n)$ and measure
$\lambda=\bigoplus_{n=1}^{\infty}\myn\lambda_n$\myp. Then the
product measure $\pi_\lambda=\bigotimes_{n=1}^\infty
\pi_{\lambda_n}$ exists and is a Poisson measure on the
configuration space $\varGamma_X^\sharp$ with intensity measure
$\lambda$\mypp.
\end{proposition}
\proof Note that $\varGamma_X^\sharp$ is a Cartesian product space,
$\varGamma_X^\sharp=\Cross_{n=1}^{\infty}
\myn\varGamma_{X_n}^\sharp$, endowed with the product
$\sigma$-algebra
$\calB(\varGamma_X^\sharp)=\bigotimes_{n=1}^{\infty}
\calB(\varGamma_{X_n}^\sharp)$. The existence of the product measure
$\pi_\lambda:=\bigotimes_{n=1}^{\infty} \pi_{\lambda_n\!}$ on
$(\varGamma_X^\sharp,\calB(\varGamma_X^\sharp))$ now follows by a
standard result for infinite products of probability measures (see,
e.g., \cite[\S\,38, Theorem~B]{Halmos} or
\cite[Corollary~5.17]{Kal}). Let us point out that this theorem is
valid without any regularity conditions on the spaces $X_n$.
% (see also a more general
%Ionescu Tulcea's extension-by-conditioning theorem \cite[Theorem
%5.17]{Kal}).

To show that $\pi_\lambda$ is a Poisson measure, one could check the
cylinder condition (\ref{eq:P-Lambda}), but it is easier to compute
its Laplace functional. Note that each function
$f\in\mathrm{M}_+(X)$ is decomposed as $f=\sum_{n=1}^\infty
f_{X_n}\!\cdot{\bf 1}_{X_n}$, where $f_{X_n}\!\in\mathrm{M}_+(X_n)$
is the restriction of $f$ to $X_n$; similarly, each configuration
$\gamma\in\varGamma_X^\sharp$ may be represented as
$\gamma=\bigsqcup_{n=1}^{\infty}\gamma_{X_n}$, where
$\gamma_{X_n}=p_{X_n}\myn\gamma\in\varGamma_{X_n}^\sharp$. Hence,
$\langle f,\gamma\rangle=\sum_{n=1}^\infty \langle
f_{X_n},\gamma_{X_n}\rangle$ and, using Proposition
\ref{pr:PoissonLT} for each $\pi_{\lambda_n}$, we obtain
\begin{align*}
\int_{\varGamma_X^\sharp}\re^{-\langle f,\gamma\rangle
}\,\pi_{\lambda}(\rd\gamma )&=\int_{\ssCross_{n=1}^\infty
\myn\varGamma_{X_n}^\sharp}\!\exp\left\{-\sum_{n=1}^\infty \langle
f_{X_n},\gamma_{n}\rangle \right\}\,\bigotimes_{n=1}^\infty\pi_{\lambda_n}(\rd\gamma_n)\\
&=\prod_{n=1}^\infty\int_{\varGamma_{X_n}^\sharp}\re^{-\langle
f_{X_n},\gamma_n\rangle }\,\pi_{\lambda_n}(\rd\gamma_n)\\
&=\exp\left\{-\sum_{n=1}^\infty\int_{X_n}\!\!\left(1-\re^{-f_{X_n}(x_n)}\right)\lambda_n
(\rd x_n)\right\}\\
&=\exp\left\{-\int_{X}\left(1-\re^{-f(x)}\right)\lambda(\rd
x)\right\},
\end{align*}
and it follows, according to formula (\ref{eq:PoissonLT}), that
$\pi_\lambda$ is a Poisson measure.
\endproof

\begin{remark}\label{rm:product}
\normalfont Using Proposition \ref{pr:product}, one can give a
construction of a Poisson measure $\pi_\lambda$ on the configuration
space $\varGamma_X^\sharp$ avoiding any additional topological
conditions upon the space $X$ (e.g., that $X$ is a Polish space)
that are needed for the sake of the Kolmogorov extension theorem
(similar ideas are developed in \cite{Kingman,Kingman2} in the
context of proper configuration spaces). To do so, recall that the
measure $\lambda$ is $\sigma$-finite and define
$X_n:=\varLambda_{n}\setminus\varLambda_{n-1}$ ($n\in\NN$), where
the sets $\emptyset=\varLambda_0\subset\varLambda_1
\subset\dots\subset\varLambda_n\subset\cdots\subset X$, such that
$\lambda(\varLambda_n)<\infty$ and
$\bigcup_{n=1}^\infty\varLambda_{n}=X$, were considered above. Then
the family of sets $(X_n)$ is a disjoint partition of $X$ (i.e.,
$X_i\cap X_j=\emptyset$ for $i\ne j$ and $\bigcup_{n=1}^\infty
X_{n}=X$), such that $\lambda(X_n)<\infty$ for all $n\in\NN$. Using
formula (\ref{eq:PoissonFD}), we construct the Poisson measures
$\pi_{\lambda_n}\equiv p_{X_n}\pi_\lambda$ on each
$\varGamma_{X_n}^\sharp$, where $\lambda_n=\lambda_{X_n}$ is the
restriction of the measure $\lambda$ to the set $X_n$. Now, it
follows by Proposition \ref{pr:product} that the product measure
$\pi_\lambda=\bigotimes_{n=1}^\infty \pi_{\lambda_n}$ is the
required Poisson measure on $\varGamma_X^\sharp$.
\end{remark}
\begin{remark}\normalfont
Although not necessary for the \emph{existence} of the Poisson
measure, in order to develop a sensible theory one needs to ensure
that there are enough measurable sets and in particular any
singleton set $\{x\}$ is measurable. To this end, it is suitable to
assume (see \cite[\S\,2.1]{Kingman}) that the diagonal set $\{x=y\}$
is measurable in the product space $X^2=X\times X$, that is,
\begin{equation}\label{eq:diag}
D:=\{(x,y)\in X^2: x=y\}\in\mathcal{B}(X^2).
\end{equation}
This condition readily implies that $\{x\}\in\mathcal{B}(X)$ for
each $x\in X$. Note that if $X$ is a Polish space, condition
(\ref{eq:diag}) is automatically satisfied because then the diagonal
$D$ is a closed set in $X^2$.
\end{remark}

Let us also record one useful general result known as the Mapping
Theorem (see \cite[\S\,2.3]{Kingman}, where configurations are
assumed proper and the mapping is one-to-one). Let $\varphi:
X\rightarrow Y$ be a measurable mapping (not necessarily one-to-one)
of $X$ to another (or the same) measurable space $Y$ endowed with
Borel $\sigma$-algebra $\mathcal{B}(Y)$. The mapping $\varphi $ can
be lifted to a measurable ``diagonal'' mapping (denoted by the same
letter) between the configuration spaces $\varGamma_{ X}^\sharp$ and
$\varGamma_{Y}^\sharp$:
\begin{equation}\label{eq:di*}
\varGamma_{X}^\sharp\ni \gamma\mapsto \varphi(\gamma):
=\bigsqcup_{x\in\gamma} \{\varphi(x)\}\in\varGamma_{Y}^\sharp.
\end{equation}

\begin{proposition}[Mapping Theorem]\label{pr:mapping}
If\/ $\pi_\lambda$ is a Poisson measure on $\varGamma_{X}^\sharp$
with intensity measure $\lambda$, then under the mapping
\textup{(\ref{eq:di*})} the push-forward measure
$\varphi^*\pi_\lambda\equiv\pi_\lambda\circ \varphi^{-1}$ is a
Poisson measure on $\varGamma_{Y}^\sharp$ with intensity measure
$\varphi^*\lambda\equiv \lambda\circ \varphi^{-1}$.
\end{proposition}

\proof It suffices to compute the Laplace functional of $\varphi
^{*}\pi _{\lambda}$. Using Proposition \ref{pr:PoissonLT}, for any
$f\in \mathrm{M}_+(Y)$ we have
\begin{align*}
L_{\varphi ^{*}\pi _{\lambda}}[f]&=
\int_{\varGamma_{Y}^\sharp}\re^{-\left\langle f,\gamma
\right\rangle}\,(\varphi ^{*}\pi_{\lambda})(\rd\gamma
)=\int_{\varGamma _{X}^\sharp}\re^{-\left\langle
f,\mypp\varphi(\gamma
)\right\rangle }\,\pi _{\lambda}(\rd\gamma )\\
&= \exp\left\{-\int_{X}\left(1-\re^{-f(\varphi(x))}\right)\lambda
(\rd x)\right\} \\[.2pc]
&=\exp\left\{-\int_{Y}\left(1-\re^{-f(y)}\right)(\varphi^{*}\myn\lambda)(\rd
y)\right\} = L_{\pi_{\varphi ^{*}\myn\lambda}}[f],
\end{align*}
and the proof is complete.
\endproof

We conclude this section with necessary and sufficient conditions in
order that $\pi_\lambda$-almost all (a.a.) configurations $\gamma\in
\varGamma _{X}^\sharp$ be proper (see Definition~\ref{def:proper}).
Although being apparently well-known folklore, these criteria are
not always proved or even stated explicitly in the literature, most
often being mixed up with various sufficient conditions, e.g., using
the property of orderliness etc.\ (see, e.g.,
\cite{CI,DVJ1,Kingman}). We do not include the proof here, as the
result follows from a more general statement for the Poisson cluster
measure (see Theorem~\ref{th:properClusterPoisson} below).

\begin{proposition}\label{pr:properPoisson}
\textup{(a)} If\/ $B\in\mathcal{B}(X)$ then $\gamma(B)<\infty$
\,\textup{(}$\pi_\lambda$-a.s.\textup{)} if and only if
$\lambda(B)<\infty$. In particular, in order that
$\pi_\lambda$-a.a.\ configurations $\gamma\in \varGamma _{X}^\sharp$
be locally finite, it is necessary and sufficient that
$\lambda(K)<\infty$ for any compact set $K\in{\mathcal B}(X)$.

\textup{(b)} In order that $\pi_\lambda$-a.a.\ configurations
$\gamma\in \varGamma _{X}^\sharp$ be simple, it is necessary and
sufficient that the measure $\lambda$ be non-atomic, that is,
$\lambda\{x\}=0$ for each $x\in X$.
\end{proposition}

\subsection{Poisson cluster measure}\label{sec:Poisson-cluster}

Let us first recall the notion of a general cluster point process
(CPP). The intuitive idea is to construct its realizations in two
steps: (i) take a background random configuration of (invisible)
``centres'' obtained as a realization of some point process
$\gamma_{\rm c}$ governed by a probability measure $\mu_{\rm c}$ on
$\varGamma_{X}^\sharp$, and (ii) relative to each centre
$x\in\gamma_{\rm c}$, generate a set of observable secondary points
(referred to as a \emph{cluster centred at~$x$}) according to a
point process $\gamma_x^{\myp\prime}$ with probability measure
$\mu_x$ on $\varGamma_{X}^\sharp$ ($x\in X$).

The resulting (countable) assembly of random points, called the
\emph{cluster point process}, can be symbolically expressed as
\begin{equation*}
\gamma=\bigsqcup_{x\in\gamma_{\rm c}}\mynn
\gamma^{\myp\prime}_{x}\in\varGamma_{X}^\sharp,
\end{equation*}
where the disjoint union signifies that multiplicities of points
should be taken into account. More precisely, assuming that the
family of secondary processes $\gamma^{\myp\prime}_x(\cdot)$ is
measurable as a function of $x\in X$, the integer-valued measure
corresponding to a CPP realization $\gamma$ is given by
\begin{equation}\label{eq:cluster-gamma}
\gamma(B)=\int_{X} \gamma_x^{\myp\prime}(B)\,\gamma_{\rm c}(\rd
x)=\sum_{x\in\gamma_{\rm c}}
\gamma_x^{\myp\prime}(B)=\sum_{x\in\gamma_{\rm c}}
\sum_{y\in\gamma_x^{\myp\prime}}\delta_y(B),\qquad B\in\calB(X).
\end{equation}

A tractable model of such a kind is obtained when (i) $X$ is a
linear space so that translations $X\ni y\mapsto y+x\in X$ are
defined, and (ii) random clusters are independent and identically
distributed (i.i.d.), being governed by the same probability law
translated to the cluster centres,
\begin{equation}\label{eq:theta-}
\mu_{x}(A)= \mu_{0}(A-x), \qquad A\in\calB(\varGamma^\sharp_X).
\end{equation}
From now on, we make both of these assumptions.

\begin{remark}\normalfont
Unlike the standard theory of CPPs whose sample configurations are
\emph{presumed} to be a.s.\ locally finite (see, e.g.,
\cite[Definition 6.3.I]{DVJ1}), the description of the CPP given
above only implies that its configurations $\gamma$ are countable
aggregates in $ X$, but possibly with multiple and/or accumulation
points, even if the background point process $\gamma_{\rm c}$ is
proper. Therefore, the distribution $\mu$ of the CPP
(\ref{eq:cluster-gamma}) is a probability measure defined on the
space $\varGamma_{ X}^\sharp$ of \emph{generalized} configurations.
It is a matter of interest to obtain conditions in order that $\mu$
be actually supported on the proper configuration space $\varGamma_{
X}$, and we shall address this issue in Section \ref{sec:proper}
below in the case of Poisson CPPs.
\end{remark}

Let $\nu_x:=\gamma_x^{\myp\prime}( X)$ be the total (random) number
of points in a cluster $\gamma_x^{\myp\prime}$ centred at point
$x\in X$ (referred to as the \emph{cluster size}). According to our
assumptions, the random variables $\nu_x$ are i.i.d.\ for different
$x$, with common distribution
\begin{equation}\label{pn}
p_n:=\mu_{0} \{\nu_0=n\} \qquad (n\in \overline{\ZZ}_+)
\end{equation}
(so in principle the event $\{\nu_0=\infty\}$ may have a positive
probability, $p_\infty\ge0$).

\begin{remark}\normalfont
One might argue that allowing for vacuous clusters (i.e., with
$\nu_x=0$) is superfluous since these are not visible in a sample
configuration, and in particular the probability $p_0$ cannot be
estimated statistically \cite[Corollary~6.3.VI]{DVJ1}. In fact, the
possibility of vacuous cluster may be ruled out without loss of
generality, at the expense of rescaling the background intensity
measure, $\lambda\mapsto (1-p_0)\mypp\lambda$. However, we keep this
possibility in our model in order to provide a suitable framework
for evolutionary cluster point processes with annihilation and
creation of particles, which we intend to study elsewhere.
\end{remark}

The following fact is well known in the case of CPPs without
accumulation points (see, e.g., \cite[\S\,6.3]{DVJ1}).
\begin{proposition}\label{pr:cluster}
The Laplace functional $L_{\mu}[\cdot]$ of the probability measure
$\mu$ on $\varGamma_ X^\sharp$ corresponding to the CPP\/
\textup{(\ref{eq:cluster-gamma})} is given, for all functions
$f\in\mathrm{M}_+(X)$, by
\begin{equation}\label{laplace}
\begin{aligned}
L_{\mu}[f] =L_{\mu_{\rm c}}\myn\bigl(-\ln L_{\mu_{x}}[f]\bigr)=
L_{\mu_{\rm c}}\myn\bigl(-\ln L_{\mu_{0}}[f(\;\!\cdot +x)]\bigr),
\end{aligned}
\end{equation}
where $L_{\mu_{\rm c}}$ acts in variable $x$.
\end{proposition}

\proof The representation (\ref{eq:cluster-gamma}) of cluster
configurations $\gamma$ implies that
$$
\langle f,\gamma\rangle=\sum_{z\in \gamma}f(z)= \sum_{x\in
\gamma_{\rm c}}\sum_{y\in \gamma_x^{\myp\prime}} f(y).
$$
Conditioning on the background configuration $\gamma_{\rm c}$ and
using the independence of the clusters $\gamma_x^{\myp\prime}$ for
different $x$, we obtain
\begin{align*}
\int_{\varGamma_ X^\sharp}\re^{-\langle f,\gamma\rangle
}\,\mu(\rd\gamma)&=\int_{\varGamma_ X^\sharp} \prod_{x\in\gamma_{\rm
c}}\left( \int_{\varGamma_
X^\sharp}\re^{-\sum_{y\in\gamma_x^{\myp\prime}}
f(y)}\,\mu_{x}(\rd\gamma_x^{\myp\prime})\right)\mu_{\rm c}(\rd\gamma_{\rm c})\\
&=\int_{\varGamma_ X^\sharp} \exp\biggl\{\myn\sum_{x\in\gamma_{\rm
c}} \ln \left(L_{\mu_{x}}[f]\right)\biggr\}\:\mu_{\rm
c}(\rd\gamma_{\rm c})=L_{\mu_{\rm c}}\bigl(-\ln
L_{\mu_{x}}[f]\bigr),
\end{align*}
which proves the first formula in (\ref{laplace}). The second one
easily follows by shifting the measure $\mu_{x}$ to the origin using
(\ref{eq:theta-}).
\endproof

In this paper, we are mostly concerned with the \emph{Poisson
CPPs}\/, which are specified by assuming that $\mu_{\rm c}$ is a
Poisson measure on configurations, with some intensity measure
$\lambda$. The corresponding probability measure  on the
configuration space $\varGamma^\sharp_{X}$ will be denoted by
$\mucl$ and called the \emph{Poisson cluster measure}.

The combination of (\ref{eq:PoissonLT}) and (\ref{laplace}) yields a
formula for the Laplace functional of the measure $\mucl$.

\begin{proposition}\label{pr:clusterP}
The Laplace functional $L_{\mucl}[f]$ of the Poisson cluster measure
$\mucl$ on $\varGamma^\sharp_{ X}$ is given, for all
$f\in\mathrm{M}_+(X)$, by
\begin{equation}\label{eq:ClusterPoissonLT}
L_{\mucl}[f]=\exp \left\{-\int_{X} \left(\int_{\varGamma^\sharp_{
X}} \left(1-\re^{-\sum_{y\in\gamma_0^{\myp\prime}}
f(y+x)}\right)\mu_0(\rd \gamma_0^{\myp\prime})\right)\lambda(\rd
x)\right\}.
\end{equation}
\end{proposition}
According to the convention made in Section~\ref{sec:General} (see
after equation~(\ref{eq:f-gamma})), if
$\gamma_0^{\myp\prime}=\emptyset$ then the function under the
internal integral in (\ref{eq:ClusterPoissonLT}) vanishes, so the
integral over $\varGamma_X^\sharp$ is reduced to that over the
subset $\{\gamma_0^{\myp\prime}\in\varGamma_X^\sharp:
\gamma_0^{\myp\prime}\ne\emptyset\}$.

\subsection{Criteria of local finiteness and
simplicity}\label{sec:proper}

In this section, we give criteria for the Poisson CPP to be locally
finite and simple. As mentioned in the Introduction, these results
appear to be new (e.g., a general criterion of local finiteness in
\cite[Lemma 6.3.II and Proposition~6.3.III]{DVJ1} is merely a more
formal rewording of the finiteness condition).

%We assume here that the measure $\lambda$ on $X$ is locally
%finite, that is, $\lambda(K)<\infty$ for any compact set $K\subset X$.
For a given set $B\in\calB( X)$ and each in-cluster configuration
$\gamma_0^{\myp\prime}$ centred at the origin, consider the set
(referred to as \emph{droplet cluster})
\begin{equation}\label{eq:D}
D_B(\gamma_0^{\myp\prime}):=\bigcup _{y\in\gamma_0^{\myp\prime}}
(B-y),
\end{equation}
which is a set-theoretic union of ``droplets'' of shape $B$ shifted
to the centrally reflected points of $\gamma_0^{\myp\prime}$.

\begin{theorem}\label{th:properClusterPoisson} Let
$\mucl$ be a Poisson cluster measure on the generalized
configuration space $\varGamma_ X^\sharp$.

\textup{(a)} In order that $\mucl$-a.a.\ configurations $\gamma\in
\varGamma _{ X}^\sharp$ be locally finite, it is necessary and
sufficient that the following two conditions hold:

{\rm (a-i)} in-cluster configurations\/ $\gamma_0^{\myp\prime}$ are
a.s.\ locally finite, that is, for any compact set
$K\in{\mathcal{B}}(X)$,
\begin{equation}\label{eq:condA1}
\gamma_0^{\myp\prime}(K)<\infty\qquad (\mu_0\text{-a.s.})
\end{equation}

{\rm (a-ii)}  for any compact set $K\in{\mathcal{B}}(X)$, the mean
$\lambda$-measure of the droplet cluster
$D_K(\gamma_0^{\myp\prime})$ is finite,
\begin{equation}\label{eq:condA2}
\int_{\varGamma^\sharp_ X}
\lambda\bigl(D_K(\gamma_0^{\myp\prime})\bigr)\,
\mu_0(\rd\gamma_0^{\myp\prime})<\infty\myp.
\end{equation}

\textup{(b)} In order that $\mucl$-a.a.\ configurations $\gamma\in
\varGamma_{X}^\sharp$ be simple, it is necessary and sufficient that
the following two conditions hold:

{\rm (b-i)} in-cluster configurations $\gamma_0^{\myp\prime}$ are
a.s.\ simple,
\begin{equation}\label{eq:condB1}
\sup_{x\in X}\gamma_0^{\myp\prime}\{x\}\le 1\qquad
(\mu_0\text{-a.s.})
\end{equation}

{\rm (b-ii)} for any $x\in X$, the ``point'' droplet cluster
$D_{\{x\}}(\gamma_0^{\myp\prime})$ has a.s.\ zero $\lambda$-measure,
\begin{equation}\label{eq:condB2}
\lambda\bigl(D_{\{x\}}(\gamma_0^{\myp\prime})\bigr)=0\qquad
(\mu_0\text{-a.s.})
\end{equation}
\end{theorem}

The proof of Theorem \ref{th:properClusterPoisson} is deferred to
the Appendix (Section~\ref{app1}).

Let us discuss the conditions of properness. First of all, the
interesting question is whether the local finiteness of the Poisson
CPP is compatible with the possibility that the number of points  in
a  cluster, $\nu_0=\gamma_0^{\myp\prime}( X)$, is infinite
(see~(\ref{pn})). The next proposition describes a simple situation
where this is not the case.
\begin{proposition} Let both conditions \textup{(a-i)} and
\textup{(a-ii)} be satisfied, and suppose that for any compact set
$K\in\calB(X)$, the $\lambda$-measure of its translations is
uniformly bounded from below,
\begin{equation}\label{sigma-cond-below}
c_K:=\inf_{x\in X} \lambda(K+x)>0.
\end{equation}
Then $\nu_0<\infty$ \textup{(}$\mu_0$-a.s.\textup{)}.
\end{proposition}
\proof Suppose that $\gamma_0^{\myp\prime}$ is an infinite
configuration. Due to (a-i), $\gamma_0^{\myp\prime}$ must be locally
finite ($\mu_0$-a.s.), which implies that there is an infinite
subset of points $y_j\in\gamma_0^{\myp\prime}$ such that the sets
$K-y_j$ are disjoint ($j\in\NN$). Hence, using
(\ref{sigma-cond-below}) we get
$$
\lambda\bigl(D_K(\gamma_0^{\myp\prime})\bigr)\ge \sum_{j=1}^\infty
\lambda(K-y_j)=\infty,
$$
which, according to condition (a-ii), may occur only with zero
probability.
\endproof

On the other hand, it is easy to construct examples of locally
finite Poisson CPPs with a.s.-infinite clusters.

\begin{example}\normalfont
Let $X=\RR^d$ and choose a measure $\lambda$ such that, for any
compact set $K\subset\RR^d$,  $\lambda(K-x)\sim C_d\,\lambda(K)\mypp
|x|^{-\alpha}$ as $x\to \infty$, where $\alpha>0$ (e.g., take
$\lambda(\rd x)=(1+|x|)^{-\alpha-d+1}\,\rd x$). Suppose now that the
in-cluster configurations $\gamma_0^{\myp\prime}=\{x_n\}$ are such
that $n^{2/\alpha}<|x_n|\le (n+1)^{2/\alpha}$, $n\in\NN$
\,($\mu_0$-a.s.). Then for any compact set~$K$
$$
\lambda\bigl(D_K(\gamma_0^{\myp\prime})\bigr)\le \sum_{x_n\in
\gamma_0^{\myp\prime}} \lambda(K-x_n)<\infty,
$$
because $\lambda(K-x_n)\sim  C_d\,\lambda(K)\myp
|x_n|^{-\alpha}=O(n^{-2})$ as $n\to\infty$.
\end{example}

It is easy to give conditions sufficient for (a-ii). The first set
of conditions below is expressed in terms of the intensity measure
$\lambda$ and the mean number of points in a cluster, while the
second condition focuses on the location of in-cluster points.

\begin{proposition}\label{pr:a1} Suppose that $\nu_0<\infty$
\textup{(}$\mu_0$-a.s.\textup{)}. Then either of the following
conditions is sufficient for condition \textup{(a-ii)} in Theorem
\textup{\ref{th:properClusterPoisson}}.

\textup{(a-ii$'$)} For any compact set $K\in\calB(X)$, the
$\lambda$-measure of its translations is uniformly bounded from
above,
\begin{equation}\label{sigma-cond}
C_K:=\sup_{x\in X} \lambda(K+x)<\infty,
\end{equation}
and, moreover, the mean number of in-cluster points is finite,
\begin{equation}\label{sigma-cond*}
\int_{\varGamma_{ X}^\sharp}
\gamma_0^{\myp\prime}(X)\,\mu_0(\rd\gamma_0^{\myp\prime}) =\sum
_{n\in\overline{\ZZ}_+} n\myp p_n<\infty
\end{equation}
\textup{(}\/this necessarily implies that
$p_\infty=0$\myp\textup{)}.

\textup{(a-ii$''$)} In-cluster configuration $\gamma_0^{\myp\prime}$
as a set in $X$ is $\mu_0$-a.s.\ bounded, that is, there exists a
compact set $K_0\in\calB(X)$ such that $\gamma_0^{\myp\prime}\subset
K_0$ \textup{(}$\mu_0$-a.s.\textup{)}.
\end{proposition}

\proof From (\ref{eq:D}) and (\ref{sigma-cond}) we obtain
$$
\lambda\bigl(D_K(\gamma_0^{\myp\prime})\bigr)\le
\sum_{y\in\gamma_0^{\myp\prime}} \lambda(K-y)\le C_K\myp
\gamma_0^{\myp\prime}( X)=C_K\mypp \nu_0,
$$
and condition (a-ii) follows by (\ref{sigma-cond*}),
$$
\int_{\varGamma_
X^\sharp}\lambda\bigl(D_K(\gamma_0^{\myp\prime})\bigr)
\,\mu_0(\rd\gamma_0^{\myp\prime}) \le C_K \int_{\varGamma_{
X}^\sharp} \gamma_0^{\myp\prime}(
X)\,\mu_0(\rd\gamma_0^{\myp\prime})<\infty.
$$

If condition (a-ii$''$) holds then
$$
D_K(\gamma_0^{\myp\prime})\subset \bigcup_{y\in K_0}(K-y)=:K-K_0,
$$
where the set $K-K_0$ is compact. Therefore,
$$
\int_{\varGamma_
X^\sharp}\lambda\bigl(D_K(\gamma_0^{\myp\prime})\bigr)
\,\mu_0(\rd\gamma_0^{\myp\prime}) \le\lambda(K-K_0)\int_{\varGamma_
X^\sharp}\mu_0(\rd\gamma_0^{\myp\prime})= \lambda(K-K_0)<\infty,
$$
and condition (a-ii) follows.
\endproof

The impact of conditions (a-ii$'$) and (a-ii$''$) is clear:
(a-ii$'$) imposes a bound on the \emph{number} of points which can
be contributed from remote clusters, while (a-ii$''$) restricts the
\emph{range} of such contribution.

Similarly, one can work out simple sufficient conditions for (b-ii).
The first condition below is set in terms of the measure $\lambda$,
whereas the second one exploits the in-cluster distribution $\mu_0$.

\begin{proposition}\label{pr:a2}
Suppose that $\nu_0<\infty$ \textup{(}$\mu_0$-a.s.\textup{)}. Then
either of the following conditions is sufficient for condition
\textup{(b-ii)} of Theorem \textup{\ref{th:properClusterPoisson}}.

\textup{(b-ii$'$)} The measure $\lambda$ is non-atomic, that is,
$\lambda\{x\}=0$\myp{} for each $x\in X$.

\textup{(b-ii$''$)} In-cluster configurations
$\gamma_0^{\myp\prime}$ have no fixed points, that is,
$\mu_0\{\gamma_0^{\myp\prime}\in\varGamma_
X^\sharp:x\in\gamma_0^{\myp\prime}\}=0$\myp{} for each $x\in X$.
\end{proposition}

\proof Condition (b-ii$'$) readily implies (b-ii):
$$
0\le \lambda\bigl(D_{\{x\}}(\gamma_0^{\myp\prime})\bigr)\le
\sum_{y\in\gamma_0^{\myp\prime}} \lambda\{x-y\}=0.
$$
Further, if condition (b-ii$''$) holds then
\begin{align}
\notag \int_{\varGamma_ X^\sharp}
\lambda\bigl(D_{\{x\}}(\gamma_0^{\myp\prime})\bigr)\,\mu_0(\rd
\gamma_0^{\myp\prime}) &=\int_{ X}\biggl(\mynn\int_{\varGamma_
X^\sharp}{\bf
1}_{\cup_{y\in\gamma_0^{\myp\prime}}\{x-y\}}(z)\,\mu_0(\rd
\gamma_0^{\myp\prime}) \biggr)\,\lambda(\rd z)\\
\notag
&=\int_{ X}\biggl(\mynn\int_{\varGamma_ X^\sharp}{\bf
1}_{\gamma_0^{\myp\prime}}(z-x)\,\mu_0(\rd \gamma_0^{\myp\prime})
\biggr)\,\lambda(\rd z)\\
\label{eq:b-ii}
&=\int_{ X} \mu_0\{\gamma_0^{\myp\prime}\in\varGamma_
X^\sharp: z-x\in\gamma_0^{\myp\prime}\} \,\lambda(\rd z)=0,
\end{align}
and condition (b-ii) follows.
\endproof

\section{Poisson cluster processes via Poisson measures}\label{sec:3}

In this section, we construct an auxiliary Poisson measure
$\pi_{\lambda^\star}$\ on the ``vector'' configuration space
$\frakX$ and prove that the Poisson cluster measure $\mucl$
coincides with the projection of $\pi_{\lambda^\star}$ onto the
configuration space $\varGamma^\sharp_X$ (Theorem~\ref{th:mucl}).
This furnishes a useful description of Poisson cluster measures that
will enable us to apply to their study the well-developed calculus
on Poisson configuration spaces.

\subsection{An auxiliary intensity measure $\lambda^\star$}\label{n-clusters}

Recall that the space $\frakX=\bigsqcup_{n\in\overline{\ZZ}_+}\mynn
X^n$ of finite or infinite vectors $\bar{x}=(x_1,x_2,\dots)$ was
introduced in Section \ref{sec:General} The probability distribution
$\mu_0$ of a generic cluster $\gamma_0^{\myp\prime}$ centred at the
origin (see Section~\ref{sec:Poisson-cluster}) determines a
probability measure $\eta$ in $\mathfrak{X}$ which is symmetric with
respect to permutations of coordinates. Conversely, $\mu_0$ is a
push-forward of the measure $\eta$ under the projection mapping
$\mathfrak{p}:\mathfrak{X}\to\varGamma^\sharp_X$ defined by
(\ref{eq:pr}), that is,
\begin{equation}\label{eq:p*eta}
\mu_0=\mathfrak{p}^*\eta\equiv\eta\circ\mathfrak{p}^{-1}.
\end{equation}

Conditional measure induced by $\eta$ on the space $X^n$ via the
condition $\gamma_0^{\myp\prime}(X)=n$ will be denoted $\eta_n$
($n\in\overline{\ZZ}_+$); in particular,
$\eta_0=\delta_{\{\emptyset\}}$. Hence (recall~(\ref{pn})),
\begin{equation}\label{eq:eta-eta}
\eta({B})=\sum_{n\in\overline{\ZZ}_+}\mynn p_n\eta_n({B}\cap X^n),
\qquad {B} \in\calB(\frakX).
\end{equation}
Note that if $p_n=\eta\{\gamma_0^{\myp\prime}(X)=n\}=0$ then
$\eta_n$ is not well defined; however, this is immaterial since the
corresponding term vanishes from the sum (\ref{eq:eta-eta}) (cf.\
also the decomposition (\ref{eq:star}) below).

The following definition is fundamental for our construction.
\begin{definition}\label{def:lambda*}
\normalfont We introduce the measure $\lambda^{\star}$ on $\frakX$
as a special ``convolution'' of the measures $\eta$ and
$\lambda$\myp:
\begin{equation}\label{eq:sigma*}
\lambda^{\star}({B}):=\int_{ X} \eta({B}-x)\,\lambda(\rd x),\qquad
{B}\in\calB(\frakX);
\end{equation}
equivalently, if $\mathrm{M}_+(\frakX)$ is the set of all
non-negative measurable functions on $\frakX$ then, for any
$f\in\mathrm{M}_+(\frakX)$,
\begin{equation}\label{eq:int-sigma-n}
\int_{\frakX}f(\bar{y})\,\lambda^{\star}(\rd\bar{y})= \int_{
X}\left(\int_{\frakX}f(\bar{y}+x)\,\eta(\rd\bar{y})
\right)\lambda(\rd x).
\end{equation}
%(this equality implies that both sides are finite or infinite
%simultaneously).
Here and below, we use the shift notation
\begin{align*}
&\bar{y}+x:=(y_1+x,y_2+x,\dots),\qquad \bar{y}=(y_1,y_2,\dots)\in
\frakX, \quad x\in X.
\end{align*}
\end{definition}

Using the decomposition (\ref{eq:eta-eta}), the measure
$\lambda^{\star}$ on $\frakX$ can be represented as a weighted sum
of contributions from the constituent spaces $X^n$:
\begin{equation}\label{eq:star}
\lambda^{\star}({B})=\sum _{n\in\overline{\ZZ}_+} \mynn p_n\myp
\lambda^{\star}_{n}({B}\cap X^n),\qquad {B} \in\calB(\frakX),
\end{equation}
where, for each $n\in\overline{\ZZ}_+$,
\begin{equation}\label{mu-measure-n}
\lambda^{\star}_{n}(B_n):=\int_{X} \eta_n(B_n-x)\,\lambda(\rd
x),\qquad B_n\in\calB(X^{n}).
\end{equation}

\begin{remark}[Case $n=0$]\normalfont
Recall that $X^0=\{\emptyset\}$ and $\mathcal{B}(X^0)=\{\emptyset,
X^0\}=\{\emptyset, \{\emptyset\}\}$. Since $\emptyset-x=\emptyset$,
$\{\emptyset\}-x=\{\emptyset\}$ ($x\in X$) and
$\eta_0=\delta_{\{\emptyset\}}$, formula (\ref{mu-measure-n}) for
$n=0$ must be interpreted as follows:
\begin{equation}\label{eq:sigma0}
\begin{aligned} \lambda^{\star}_{0}(\emptyset)&=\int_{X}
\eta_0(\emptyset)\,\lambda(\rd x) = 0,\\
\lambda^{\star}_{0}(\{\emptyset\})&=\int_{X}
\eta_0(\{\emptyset\})\,\lambda(\rd x) =\int_{X} \lambda(\rd x) =
\lambda(X)=\infty.
\end{aligned}
\end{equation}
\end{remark}

If $p_\infty=0$ (i.e., clusters are a.s.\ finite) and $X=\RR^d$,
then in order that the measure $\eta$ be absolutely continuous
(a.c.) with respect to the ``Lebesgue measure''
$\rd\bar{y}=\delta_{\{\emptyset\}}(\rd\bar{y})\oplus
\bigoplus_{n=1}^\infty \rd{y}_1\mynn\otimes\cdots\otimes\rd{y}_n$ on
$\frakX=\bigsqcup_{n=0}^{\infty}\myn X^n$, with some density $h$,
\begin{equation}\label{eq:m}
\eta(\rd\bar{y})=h(\bar{y})\,\rd\bar{y}, \qquad \bar{y}\in \frakX,
\end{equation}
it is necessary and sufficient that each measure $\eta_n$ is a.c.\
with respect to Lebesgue measure on $X^n$, that is,
$\eta_n(\rd\bar{y})=h_n(\bar{y})\,\rd\bar{y}$, \,$\bar{y}\in X^n$
($n\in\ZZ_+$); in this case, the density $h$ is decomposed as
\begin{equation}\label{eq:h-h}
h(\bar{y})=\sum_{n=0}^\infty p_n \myp h_n(\bar{y})\,{\bf
1}_{X^n}(\bar{y}), \qquad \bar{y}\in\frakX.
\end{equation}
Moreover, it follows that the measures $\lambda^{\star}$ and
$\lambda^{\star}_n$ ($n\in\ZZ_+$) are also a.c., with the
corresponding densities
\begin{equation}\label{density}
\begin{aligned}
s(\bar{y})&=\frac{\lambda^{\star}(\rd\bar{y})}{\rd\bar{y}}=\int_{
X}h(\bar{y}-x)\,\lambda(\rd x),&\hspace{1.5pc}\bar{y}&\in \frakX,\\
s_n(\bar{y})&=\frac{\lambda^{\star}_n(\rd\bar{y})}{\rd\bar{y}}=\int_{
X}h_n(\bar{y}-x)\,\lambda(\rd x),&\hspace{1.5pc}\bar{y}&\in X^n,
\end{aligned}
\end{equation}
related by the equation (cf.\ (\ref{eq:star}), (\ref{eq:h-h}))
\begin{equation}\label{eq:s-sn}
s(\bar{y})=\sum_{n=0}^\infty p_n \myp s_n(\bar{y})\,{\bf
1}_{X^n}(\bar{y}), \qquad \bar{y}\in\frakX.
\end{equation}

\begin{remark}\label{rm:n=1}
\normalfont In the case $n=1$, the definition (\ref{mu-measure-n})
is reduced to
\begin{equation}\label{eq:sigma1}
\lambda^{\star}_{1}(B_1)=\int_ X \eta_1(B_1-x)\,\lambda(\rd x)=
\int_ X \lambda(B_1-x)\,\eta_1(\rd x),\qquad B_1\in\calB(X).
\end{equation}
In particular, if $\lambda$ is translation invariant (i.e.,
$\lambda(B_1-x)=\lambda(B_1)$ for each $B_1\in\mathcal{B}(X)$ and
any $x\in X$), then $\lambda^{\star}_{1}$ coincides with $\lambda$.
\end{remark}

\begin{remark}\label{rm:blowup}
\normalfont There is a possibility that the measure
$\lambda^{\star}_{n}$ defined by (\ref{mu-measure-n}) is not
$\sigma$-finite (even if $\lambda$ is), and moreover,
$\lambda^{\star}_{n}$ may appear to be locally infinite, in that
$\lambda^{\star}_{n}(B)=\infty$ for any compact set $B\subset \RR^n$
with non-empty interior, as in the following example.
\end{remark}

\begin{example}\label{ex:blowup}
\normalfont  Let $X=\RR$, and for $n\ge1$ set
$$
\lambda(\rd x):= \re^{|x|}\,\rd x,\qquad\eta_1(\rd
x):=\frac{|x|\,\rd x}{(x^2+1)^2} \qquad (x\in\RR),
$$
and $ \eta_n(\rd \bar{x}):=\eta_1(\rd
x_1)\otimes\cdots\otimes\eta_1(\rd x_n)$,
\,$\bar{x}=(x_1,\dots,x_n)\in\RR^n$. Note that for $a<b$ and any
$x\notin [a,b]$,
$$
\eta_1[a-x,b-x]= \frac{(b-a)\mypp|a+b-2x|}{2\mypp((a-x)^2+1)\mypp
((b-x)^2+1)}\sim\frac{b-a}{|x|^3}\qquad(x\to\infty),
$$
so, for any rectangle $B=\Cross_{i=1}^{n} [a_i,b_i]\subset\RR^n$
($a_i<b_i$), by (\ref{eq:sigma1}) we obtain
$$
\lambda^{\star}_{1}(B)=\int_{-\infty}^\infty \prod_{i=1}^n
\eta_1[a_i-x,b_i-x]\;\re^{|x|}\,\rd x =\infty.
$$
\end{example}

The next example illustrates a non-pathological situation.

\begin{example}\label{ex:1} \normalfont Let $
X=\mathbb{R}$, and for $n\ge 1$ set
\begin{equation*}
h_n(\bar{y})=\frac{1}{(2\pi)^{n/2}}\,\re^{-\|\bar y\|^2/2},\qquad
\bar{y}=(y_1,\dots,y_n)\in\RR^n,
\end{equation*}
where $\|\cdot\|$ is the usual Euclidean norm in $\RR^n$. Thus,
$\eta_n$ is a standard Gaussian measure on $\RR^{n}$. Assume that
$\lambda$ is the Lebesgue measure on $\RR$, $\lambda(\rd x)=\rd x$.
For $n=1$, from equation (\ref{density}) we obtain
\begin{align*} s_1(y) &=\frac{1}{\sqrt{2\pi}}
\int_{-\infty}^\infty\re^{-(y-x)^{2}/2}\,\rd x =1,
\end{align*}
hence $\lambda^{\star}_{1}=\lambda$, in accord with Remark
\ref{rm:n=1}. If $n=2$ then from (\ref{density}) we get
\begin{align*}
s_2(y_1,y_2) &=\frac{1}{2\pi}
\int_{-\infty}^\infty\re^{-((y_{1}-x)^{2}+(y_{2}-x)^{2})/2}\,\rd x
=\frac{1}{2\sqrt{\pi}}\,\re^{-(y_{1}-y_{2})^{2}/4}.
\end{align*}
Via the orthogonal transformation $z_1=(y_1+y_2)/\sqrt{2}$,
\,$z_2=(y_1-y_2)/\sqrt{2}$, the measure $\lambda^{\star}_2$ is
reduced to
$$
\lambda^{\star}_2(\rd{z}_1,\rd{z}_2)=\frac{1}{2\sqrt{\pi}}\,\re^{-z_2^2/2}\,\rd
z_1\,\rd z_2,
$$
which is a product of the standard Gaussian measure (along the
coordinate axis $z_1$) and the scaled Lebesgue measure $\rd
z_2/\sqrt{2}$. Note that $\lambda^{\star}_2(\RR^2)=\infty$, but any
vertical or horizontal strip of finite width (in coordinates $\bar
y$) has finite $\lambda^{\star}_2$-measure.

In general ($n\ge 2$), integration in (\ref{density}) yields
\begin{align*}
s_n(\bar{y})&=\frac{1}{(\sqrt{2\pi})^{n-1}\sqrt{n}}\,
\exp\left\{-\frac{1}{2}\left(\|\bar{y}\|^2-n^{-1}|y_1+\cdots
+y_n|^2\right)\right\}, \qquad {\bar y}\in\RR^n,
\end{align*}
It is easy to check that after an orthogonal transformation
$\bar{z}=\bar{y}\mypp U$ such that $z_1=n^{-1/2}(y_1+\cdots+y_n)$,
the measure $\lambda^{\star}_n$ takes the form
$$
\lambda^{\star}_n(\rd{\bar z})=\frac{\rd
z_1}{\sqrt{n}}\cdot\frac{1}{(\sqrt{2\pi})^{n-1}}\,\re^{-(z_2^2+\cdots+z_{n}^2)/2}\,
\rd z_2\cdots \rd z_{n}\,,\qquad \bar{z}=(z_1,\dots,z_n).
$$
That is, $\lambda^{\star}_n(\rd\bar{z})$ is a product of the scaled
Lebesgue measure $\rd z_1/\sqrt{n}$ and the standard Gaussian
measure in coordinates $z_2,\dots, z_{n}$. Hence
$\lambda^{\star}_n(\RR^n)\allowbreak=\infty$, but for any coordinate
strip $C_i=\{\bar{y}\in\RR^n: |y_i|\le c\}$ we have
$\lambda^{\star}_n(C_i)<\infty$.
\end{example}

Example \ref{ex:1} can be generalized as follows.

\begin{proposition}\label{pr:projection}
Suppose that $p_\infty=0$ and $X=\RR^d$. For each $n\ge1$, consider
an orthogonal linear transformation $\bar{z}=\bar{y}\mypp U_n$ of
the space $ X^n$ such that
\begin{equation}\label{eq:z1'}
z_1=\frac{y_1+\cdots+y_n}{\sqrt{n}},\qquad
\bar{z}=(z_1,\dots,z_n),\quad \bar{y}=(y_1,\dots,y_n).
\end{equation}
Set\/ $\bar z^{\myp\prime}\!:=(z_2,\dots,z_n)$ and consider the
measures
\begin{align}
\label{eq:eta'} \eta_n^{\myp\prime}(B^{\myp\prime}):={}&\int_ X
\eta_n(\rd z_1,B^{\myp\prime})=\eta_n( X\times
B^{\myp\prime}),&&B^{\myp\prime}\in\calB(X^{n-1}),\\[.3pc]
\label{eq:tilde-sigma} \tilde\lambda_n(B_1|\mypp\bar
z^{\myp\prime}):={}&\int_ X
\lambda\left(\frac{B_1-z_1}{\sqrt{n}}\right)\eta_n(\rd z_1|\mypp
\bar z^{\myp\prime}),&&B_1\in\calB(X),
\end{align}
where $\eta_n(\rd z_1|\mypp \bar z^{\myp\prime})$ is the measure on
$X$ obtained from $\eta_n$ via conditioning on
$\bar{z}^{\myp\prime}$. Then the measure\/ $\lambda^{\star\!}$ can
be decomposed as
\begin{equation}\label{eq:sigmacl}
\lambda^{\star}(\rd\bar z)=p_0\myp\lambda^{\star}_0(\rd \bar
z)+\sum_{n=1}^\infty p_n\myp \tilde\lambda_n(\rd z_1|\mypp\bar
z^{\myp\prime})\, \eta_n^{\myp\prime}(\rd\bar z^{\myp\prime}),
\end{equation}
where $\lambda^{\star}_0$ is defined in \textup{(\ref{eq:sigma0})}.
In particular, if\/ the measure $\lambda$ on $X=\mathbb{R}^d$ is
translation invariant then
\begin{equation}\label{eq:tilde-sigma0}
\lambda^{\star}(\rd\bar z)=p_0\myp\lambda^{\star}_0(\rd \bar
z)+\sum_{n=1}^\infty p_n \frac{\lambda(\rd
z_1)}{n^{d/2}}\,\eta_n^{\myp\prime}(\rd \bar z^{\myp\prime}).
\end{equation}
\end{proposition}

\proof For a fixed $n\ge1$, let $\bar{z}=\bar{y}\mypp U_n$ and
consider a Borel set in $X^n$ of the form $B_n=\{\bar{y}\in X^n:
z_1\in B_1, \,\bar{z}^{\myp\prime}\!\in B^{\myp\prime}_n\}$. By
equation (\ref{eq:z1'}) and orthogonality of $U_n$, we have
$B_n-x=\{\bar{z}\in X^n: z_1\in B_1-x\sqrt{n},
\,\bar{z}^{\myp\prime}\!\in B^{\myp\prime}_n\}$. Therefore, from
(\ref{mu-measure-n}) we obtain
\begin{align*} \lambda^{\star}_{n}(B_n)
&=\int_ X\left(\int_{ X^n}  {\bf 1}_{(B_1-x\sqrt{n}\myp)\times
B^{\myp\prime}_n}(\bar z)\,\eta_n(\rd\bar z)\right)\lambda(\rd x)\\
&=\int_{ X^n}\left(\int_ X {\bf
1}_{B_1-x\sqrt{n}}\myp(z_1)\,\lambda(\rd x)\right)
{\bf 1}_{B^{\myp\prime}_n}(\bar z^{\myp\prime})\,\eta_n(\rd\bar z)\\
&=\int_{ X\times  X^{n-1}}\left( \int_ X {\bf
1}_{(B_1-z_1)/\sqrt{n}\,}(x)\,\lambda(\rd x)\right){\bf
1}_{B^{\myp\prime}_n}(\bar z^{\myp\prime})\,\eta_n(\rd z_1\myp|\mypp
\bar z^{\myp\prime})\,
\eta_n^{\myp\prime}(\rd\bar z^{\myp\prime})\\
&=\int_{B^{\myp\prime}_n}\! \left(\int_ X
\lambda\left((B_1-z_1)/\sqrt{n}\,\right) \eta_n(\rd z_1\myp|\mypp
\bar z^{\myp\prime})\right)
\eta_n^{\myp\prime}(\rd\bar z^{\myp\prime})\\
&=\int_{B^{\myp\prime}_n}\tilde\lambda_n(B_1|\mypp\bar
z^{\myp\prime})\, \eta_n^{\myp\prime}(\rd\bar z^{\myp\prime}),
\end{align*}
and by inserting this into equation (\ref{eq:star}) we get
(\ref{eq:sigmacl}). Finally, the translation invariance of $\lambda$
implies that
$\lambda((B_1-z_1)/\sqrt{n}\,)=n^{-d/2}\myp\lambda(B_1)$. Formula
(\ref{eq:tilde-sigma}) then gives $\tilde\lambda_n(B_1|\mypp\bar
z^{\myp\prime})=n^{-d/2}\myp\lambda(B_1)$, and
(\ref{eq:tilde-sigma0}) readily follows from (\ref{eq:sigmacl}).
\endproof

Using decomposition (\ref{eq:sigmacl}), it is easy to obtain the
following criterion of absolute continuity of the measure
$\lambda^\star$.
\begin{corollary}\label{cor:sigma*a.c.}
Suppose that $p_\infty=0$ and $X=\RR^d$. Then the measure
$\lambda^\star(\rd\bar{x})$ on $\frakX$ is a.c.\ with respect to the
Lebesgue measure
$\rd\bar{x}=\delta_{\{\emptyset\}}(\rd\bar{x})\oplus\bigoplus_{n=1}^\infty
\rd{x}_1\mynn\otimes\cdots\otimes\rd{x}_n$ if and only if the
following two conditions hold:
\begin{enumerate}
\item[\rm (i)] for each $n\ge1$, the measure
$\eta_n^{\myp\prime}(\rd\bar{z}^{\myp\prime})$ is a.c.\ with respect
to the Lebesgue measure $\rd\bar{z}^{\myp\prime}$ on $
X^{n-1}$\/\textup{;}
\item[\rm (ii)] for a.a.\ $\bar{z}^{\myp\prime}$, the measure $\tilde\lambda_n(\rd
z_1|\mypp\bar z^{\myp\prime})$ is a.c.\ with respect to the Lebesgue
measure $\rd z_1$ on $ X$.
\end{enumerate}
In particular, if\/ $\lambda$ is translation invariant then
condition \textup{(ii)} is automatically fulfilled and hence
condition \textup{(i)} alone is necessary and sufficient for the
absolute continuity of\/ $\lambda^\star$.
\end{corollary}

\begin{remark}\normalfont
The absolute continuity of $\eta$ is sufficient (cf.\ (\ref{eq:m}),
(\ref{density})), but not necessary, for condition (i). This is
illustrated by the following example:
$$
\eta(\rd y_1,\rd y_2)=\frac12\,\delta_{\{1\}}(\rd y_1) f(y_2)\,\rd
y_2 +\frac12\,\delta_{\{1\}}(\rd y_2) f(y_1)\,\rd y_1,\qquad
(y_1,y_2)\in\mathbb{R}^2,
$$
where $f(y)$ ($y\in\mathbb{R}$) is some probability density
function. Then the projection measure $\eta^{\myp\prime}$ on
$\mathbb{R}$ (see~(\ref{eq:eta'})) is given by
$$
\eta^{\myp\prime}(\rd
z^{\myp\prime})=\frac{\sqrt{2}}{2}\left(f(1-\sqrt{2}\,z^{\myp\prime})+f(1+\sqrt{2}\,z^{\myp\prime})\right)\rd
z^{\myp\prime},\qquad z^{\myp\prime}=\frac{y_1-y_2}{\sqrt{2}},
$$
and so $\eta^{\myp\prime}(\rd z^{\myp\prime})$ is absolutely
continuous.
\end{remark}

The next result shows that the absolute continuity of
$\lambda^\star$ implies that the Poisson cluster process a.s.\ has
no multiple points (see Definition~\ref{def:proper}).

\begin{proposition}\label{pr:sigma*a.c.}
Suppose that $p_\infty=0$, $X=\RR^d$, and the measure
$\lambda^\star(\rd\bar{x})$ on\/ $\mathfrak{X}$ is a.c.\ with
respect to the Lebesgue measure $\rd\bar{x}$. Then $\mucl$-a.a.\
configurations $\gamma\in\varGamma_{X}^\sharp$ are simple.
\end{proposition}
\proof By Theorem \ref{th:properClusterPoisson}, it suffices to
check conditions (b-i) and (b-ii). First, note that if condition
(b-i) is not satisfied (i.e., if the set of points
$\bar{y}\in\mathfrak{X}$ with two or more coinciding coordinates has
positive $\eta$-measure), than the projected measure
$\eta^{\myp\prime}(\rd\bar{z}^{\myp\prime})$ charges a hyperplane
(of codimension $1$) in the space $\mathfrak{X}^{\myp\prime}$
spanned over the coordinates $\bar{z}^{\myp\prime}$. But this
contradicts the absolute continuity of $\lambda^\star$, since such
hyperplanes have zero Lebesgue measure.

Furthermore, similarly to (\ref{eq:b-ii}) and using the definition
(\ref{eq:sigma*}), for each $x\in X$ we obtain
\begin{align*}
\int_{\mathfrak{X}}\textstyle \lambda\bigl(\myn\bigcup
_{y_i\in\bar{y}}{}\{x-y_i\}\bigr) \,\eta(\rd\bar{y})&=\int_{X}
\eta\{\bar{y}\in\mathfrak{X}: z-x\in\bar{y}\} \,\lambda(\rd z)\\[.2pc]
&=\lambda^\star\{\bar{y}\in\mathfrak{X}:
-x\in\mathfrak{p}(\bar{y})\}=0,
\end{align*}
by the absolute continuity of $\lambda^\star$. Hence,
$\lambda\bigl(\myn\bigcup_{y_i\in\bar{y}}\{x-y_i\}\bigr)=0$
\,($\eta$-a.s.) and condition (b-ii) follows.
\endproof

\subsection{An auxiliary
Poisson measure $\pi_{\lambda^\star}$}\label{sec:3.2}

Recall  that the ``unpacking'' map
$\mathfrak{p}:\frakX\to\varGamma_X^\sharp$ is defined in
(\ref{eq:pr}). For any Borel subset $B\in\mathcal{B}(X)$, denote
\begin{equation}\label{eq:K}
\frakX_{B}:=\{\bar{x}\in \frakX:\ \mathfrak{p}(\bar{x})\cap B \neq
\emptyset\}\in\mathcal{B}(\varGamma_X^\sharp).
\end{equation}
The following result is crucial for our purposes (cf.\ Example
\ref{ex:1}).

\begin{proposition}\label{prop1}
Let $B\in\mathcal{B}(X)$ be a set such that $\lambda(B)<\infty$.
Then condition \textup{(\ref{eq:condA2})} of Theorem
\textup{\ref{th:properClusterPoisson}\myp(a)} \textup{(}i.e., that
the mean $\lambda$-measure of the droplet cluster $D_B$ is
finite\textup{)} is necessary and sufficient in order that
$\lambda^{\star}(\frakX_{B})<\infty$, or equivalently,
$\bar{\gamma}(\frakX_{B}) <\infty$ for $\pi_{\lambda^{\star}}$-a.a.
$\bar{\gamma}\in\varGamma_{\frakX}$\myp.
\end{proposition}

\proof Using (\ref{eq:sigma*}) we obtain
\begin{equation}\label{eq:p3}
\lambda^{\star}(\frakX_{B})=\int_ X \eta(\frakX_{B}-x) \,\lambda(\rd
x)=\int_{\frakX}\!\left(\int_ X{\bf 1}_{\frakX_{B}}(\bar
y+x)\,\lambda(\rd x)\right)\eta(\rd\bar{y}).
\end{equation}
By definition (\ref{eq:K}), $\bar y+x\in \frakX_B$ if and only if
$x\in \bigcup_{y_i\in\bar y} (B-y_i)\equiv D_B(\bar y)$
(see~(\ref{eq:D})). Hence, (\ref{eq:p3}) can be rewritten as
\begin{align*}
\lambda^{\star}(\frakX_{B})&=\int_{\frakX}\!\left(\int_ X{\bf
1}_{D_{B}(\bar y)}(x)\,\lambda(\rd x)\right)\eta(\rd\bar{y})\\[.2pc]
&=\int_{\frakX}\lambda\bigl(D_B(\bar y)\bigr)\,\eta(\rd\bar{y})
=\int_{\varGamma_X^\sharp}\lambda\bigl(D_B(\gamma_0^{\myp\prime})\bigr)\,
\mu_0(\rd\gamma_0^{\myp\prime}),
\end{align*}
by the change of measure (\ref{eq:p*eta}). Thus, the bound
$\lambda^{\star}(\frakX_{B})<\infty$ is nothing else but condition
(\ref{eq:condA2}) applied to $B$. The second part follows by
Proposition~\ref{pr:properPoisson}\myp(a).
\endproof

Let us consider the cluster configuration space
$\varGamma^\sharp_{\frakX}$ over the space $\frakX$ with generic
elements $\bar{\gamma}\in\varGamma^\sharp_{\frakX}$. Our next goal
is to define a Poisson measure $\pi_{\lambda^\star}$ on
$\varGamma_{\frakX}^\sharp$ with intensity $\lambda^\star$. However,
as Remark \ref{rm:blowup} and Example \ref{ex:blowup} indicate, the
measure $\lambda^\star$ may not be $\sigma$-finite, in which case a
general construction of the Poisson measure as developed in Section
\ref{sec:Poisson} would not be applicable. It turns out that
Proposition \ref{prop1} provides a suitable basis for a good theory.

\begin{proposition}\label{pr:pi*}
Suppose that condition \textup{(\ref{eq:condA2})} of Theorem
\textup{\ref{th:properClusterPoisson}\myp(a)} is fulfilled for any
set $B\in\mathcal{B}(X)$ such that $\lambda(B)<\infty$. Then the
measure $\lambda^\star$ on $\frakX$ is $\sigma$-finite.
\end{proposition}

\proof Since the measure $\lambda$ on $X$ is $\sigma$-finite, there
is a sequence of sets $B_k\in\mathcal{B}(X)$ $(k\in\NN$) such that
$\lambda(B_k)<\infty$ and $\bigcup_{k=1}^{\infty}\myn B_k=X$. Hence,
by Proposition \ref{prop1}, $\lambda^\star(\frakX_{B_k})<\infty$ for
each $B_k$, and from the definition (\ref{eq:K}) it is clear that
$\bigcup_{k=1}^\infty \frakX_{B_k}=\frakX$.
\endproof

By virtue of Proposition \ref{pr:pi*} and according to the
discussion in Section \ref{sec:Poisson}, the Poisson measure
$\pi_{\lambda^\star}$ on the configuration space
$\varGamma^\sharp_{\frakX}$ does exist. Moreover, due to Remark
\ref{rm:product}, this is true even without any extra topological
assumptions, except that of $\sigma$-finiteness of the basic
intensity measure $\lambda$. The construction of
$\pi_{\lambda^\star}$ may be elaborated further by applying
Proposition \ref{pr:product} to
$\frakX=\bigsqcup_{n\in\overline{\ZZ}_+}\myn X^n$ and
$\lambda^\star=\bigoplus_{n\in\overline{\ZZ}_+}\myn
p_n\lambda^\star_n$; namely, one first defines the Poisson measures
$\pi_{p_n\lambda_n^\star}$ on the constituent configuration spaces
$\varGamma_{X^n}^\sharp$ (of course, the measures $\lambda_n^\star$
are $\sigma$-finite together with $\lambda^\star$) and then
constructs the Poisson measure $\pi_{\lambda^\star}$ on
$\varGamma_\frakX^\sharp=\Cross_{n\in\overline{\ZZ}_+}\myn
\varGamma_{X^n}^\sharp$ as a product measure,
$\pi_{\lambda^\star}=\bigotimes_{n\in\overline{\ZZ}_+}
\pi_{p_n\lambda_n^\star}$.

\begin{remark}\normalfont
A degenerate Poisson measure $\pi_{p_0\lambda_0^\star}$ on
$\varGamma_{X^0}^\sharp$ is defined as
$\pi_{p_0\lambda_0^\star}:=\delta_{\{\bar{\gamma}_\infty\}}$, where
$\bar{\gamma}_\infty=(\{\emptyset\},\{\emptyset\},\dots)$, i.e.,
$\bar{\gamma}_\infty(X^0)=\infty$. The component
$\pi_{p_0\lambda_0^\star}$ is actually irrelevant in the projection
construction described in the next section.
\end{remark}

\subsection{Poisson cluster measure via the Poisson measure
$\pi_{\lambda^\star}$}\label{sec:PCP}

We can lift the projection mapping (\ref{eq:pr}) to the
configuration space $\varGamma^\sharp_{\frakX}$ by setting
\begin{equation}\label{eq:proj}
\varGamma_{\frakX}^\sharp\ni\bar{\gamma}\mapsto
\mathfrak{p}(\bar{\gamma}):=\bigsqcup _{\bar{x}\in
\bar{\gamma}}\mathfrak{p}(\bar{x})\in\varGamma_{X}^\sharp.
\end{equation}
Disjoint union in (\ref{eq:proj}) highlights the fact that
$\mathfrak{p}(\bar{\gamma})$ may have multiple points, even if
$\bar\gamma$ is proper. It is not difficult to see that
(\ref{eq:proj}) is a measurable mapping. Indeed, using the sets
$D_B^{\myp n}$ introduced in (\ref{eq:DD}), for any cylinder set
$C_B^{\myp n}\subset \varGamma^\sharp_{ X}$ ($B\in\mathcal{B}(X)$,
$n\in\ZZ_+$) we have $\mathfrak{p}^{-1}(C_B^{\myp n})=A_B^{\myp
n}\in\mathcal{B}(\varGamma^\sharp_\frakX)$, where, for instance,
\begin{align*}
A_B^{\myp 0}&=\{\bar{\gamma}\in\varGamma^\sharp_{\frakX}:
\,\bar{\gamma}(\frakX\setminus D_B^{\myp 0})=0\},\\[.2pc]
A_B^{\myp 1}&=\{\bar{\gamma}\in\varGamma^\sharp_{\frakX}:
\,\bar{\gamma}(D_B^{\myp 1})=1 \},\\[.2pc]
A_B^{\myp 2}&=\{\bar{\gamma}\in\varGamma^\sharp_{\frakX}:
\,\bar{\gamma}(D_B^{\myp 2})=1 \ \ \text{or} \ \
\bar{\gamma}(D_B^{\myp 1})=2\},
\end{align*}
and, more generally, $ A_B^{\myp
n}=\bigcup_{(n_k)}\myn\!\bigcap_{k=1}^{\infty}\{\bar{\gamma}\in\varGamma^\sharp_{\frakX}:
\,\bar{\gamma}(D_B^{\myp k})=n_k\}$, where the union is taken over
integer arrays $(n_k)=(n_1,n_2,\dots)$ such that $n_k>0$ and $\sum_k
k\myp n_k=n$.

Finally, we introduce the measure $\mu$ on $\varGamma^\sharp_ X$ as
a push-forward of the Poisson measure $\pi_{\lambda^{\star}}$ under
the mapping $\mathfrak{p}$,
\begin{equation}\label{eq:mu*}
\mu(A):= (\mathfrak{p}^*\pi_{\lambda^{\star}})(A)
\equiv\pi_{\lambda^{\star}}(\mathfrak{p}^{-1}(A)),\qquad
A\in\calB(\varGamma^\sharp_ X).
\end{equation}

The next theorem is the main result of this section.

\begin{theorem}\label{th:mucl}
The measure $\mu=\mathfrak{p}^*\pi_{\lambda^{\star}}$ on
$\varGamma^\sharp_ X$ defined by\/ \textup{(\ref{eq:mu*})} coincides
with the Poisson cluster measure $\mucl$.
\end{theorem}

\proof According to Section \ref{sec:General}, it is sufficient to
compute the Laplace functional of the measure $\mu$. For any
$f\in\mathrm{M}_+(X)$, by the change of measure (\ref{eq:mu*}) we
have
\begin{equation}\label{eq:LTn}
\int_{\varGamma _{ X}^\sharp}\re^{-\langle f,\myp\gamma \rangle
}\,\mu(\rd\gamma )  =\int_{\varGamma _{\frakX}}\re^{-\langle
f,\mypp\mathfrak{p}(\bar{\gamma})\rangle }\,\pi
_{\lambda^{\star}}(\rd\bar{\gamma}) =\int_{\varGamma
_{\frakX}}\re^{-\langle \tilde{f},\myp\bar{\gamma}\rangle }\,\pi
_{\lambda^{\star}}(\rd\bar{\gamma}),
\end{equation}
where $\tilde{f}(\bar{y}):=\sum_{y_i\in\bar{y}}
f(y_i)\in\mathrm{M}_+(\frakX)$. According to (\ref{eq:PoissonLT})
and (\ref{eq:int-sigma-n}), the right-hand side of (\ref{eq:LTn})
takes the form
\begin{gather*}
\exp\left\{ -\int_{\frakX}\left(1-\re^{-\tilde{f}(\bar{y})}\right)
\lambda^{\star}(\rd\bar{y})\right\} =\exp \left\{- \int_{
X}\int_{\frakX} \left(1-\re^{-\tilde{f}(\bar{y}+x)}\right)\eta
(\rd\bar{y})\,\lambda(\rd x)\right\}\\[.2pc]
=\exp \left\{-\int_{ X}\left(\int_{\frakX}
\left(1-\re^{-\sum_{y_i\in\bar y} f(y_i+x)}\right)\eta
(\rd\bar{y})\right)\lambda(\rd x)\right\},
\end{gather*}
which, after the change of measure (\ref{eq:p*eta}), coincides with
the expression (\ref{eq:ClusterPoissonLT}) for the Laplace
functional of the Poisson cluster measure $\mucl$.
\endproof

\begin{remark}\normalfont
As an elegant application of the technique developed here, let us
give a transparent proof of Theorem
\ref{th:properClusterPoisson}\myp(a) (cf.\ the Appendix,
Section~\ref{app1}). Indeed, in order that a given compact set
$K\subset X$ contain finitely many points of configuration
$\gamma=\mathfrak{p}(\bar\gamma)$, it is necessary and sufficient
that (i) each cluster ``point'' $\bar{x}\in\bar\gamma$ is locally
finite, which is equivalent to the condition (a-i), and (ii) there
are finitely many points $\bar{x}\in\bar\gamma$ which contribute to
the set $K$ under the mapping $\mathfrak{p}$, the latter being
equivalent to condition (a-ii) by Proposition~\ref{prop1}.
\end{remark}

\subsection{An alternative construction of the measures\/
$\pi_{\lambda^{\star}}\!$ and\/ $\mucl$}\label{sec:3.3}

The measure $\pi_{\lambda^{\star}}$ was introduced in the previous
section as a Poisson measure on the configuration space
$\varGamma_{\mathfrak{X}}$ with a certain intensity measure
$\lambda^{\star}$ prescribed \emph{ad hoc} by equation
(\ref{eq:sigma*}). In this section, we show that
$\pi_{\lambda^{\star}}$ can be obtained in a more natural way as a
suitable skew projection of a canonical Poisson measure
\,$\widehat\pi$ \,defined on a bigger configuration space
$\varGamma^\sharp_{X\times \frakX}$, with the product intensity
measure $\lambda\otimes\eta$.

More specifically, given a Poisson measure $\pi_{\lambda}$ in
$\varGamma^\sharp_{ X}$, let us construct a new measure
\,$\widehat\mu$ \,in $\varGamma^\sharp_{ X\times\frakX}$ as the
probability distribution of random configurations
$\widehat{\gamma}\in \varGamma^\sharp_{ X\times\frakX}$ obtained
from Poisson configurations $\gamma\in\varGamma^\sharp_{ X}$ by the
rule
\begin{equation}\label{eq:gamma-hat}
\gamma \mapsto \widehat{\gamma}:=\{(x,\bar{y}_x): x\in\gamma,
\,\bar{y}_x\in\frakX\},
\end{equation}
where the random vectors $\{\bar{y}_x\}$ are i.i.d., with common
distribution $\eta(\rd \bar{y})$. Geometrically, such a construction
may be viewed as pointwise i.i.d.\ translations of the Poisson
configuration $\gamma\in X$ into the space $ X\times\frakX$,
$$
X\ni x\leftrightarrow (x,0)\mapsto (x,\bar{y}_x)\in X\times\frakX.
$$

\begin{remark}\normalfont
Vector $\bar{y}_x$ in each pair $(x,\bar{y}_x)\in X\times\frakX$ can
be interpreted as a \emph{mark} attached to the point $x\in X$, so
that $\widehat\gamma$ becomes a marked configuration, with the mark
space $\frakX$ (see \cite{DVJ1,KunaPhD}).
\end{remark}

\begin{theorem}\label{th:product-Poisson}
The probability distribution\/ $\widehat{\mu}$ of random
configurations\/ $\widehat{\gamma}\in \varGamma^\sharp_{
X\times\frakX}$ constructed in\/ \textup{(\ref{eq:gamma-hat})} is
given by the Poisson measure\/ $\pi_{\widehat{\lambda}}$ on the
configuration space $\varGamma^\sharp_{ X\times\frakX}$, with the
product intensity measure\/ $\widehat{\lambda}:=\lambda\otimes\eta$.
\end{theorem}
\proof Let us check that, for any non-negative measurable function
$f(x,\bar{y})$ on $ X\times\frakX$, the Laplace functional of the
measure $\widehat{\mu}$ is given by formula (\ref{eq:PoissonLT}).
Using independence of the vectors $\bar{y}_x$ corresponding to
different $x$, we obtain
\begin{align*}
\int_{\varGamma^\sharp_{ X\times\frakX}}\re^{-\langle
f,\mypp\widehat\gamma\mypp\rangle}\,\widehat{\mu}(\rd\widehat{\gamma})
&=\int_{\varGamma^\sharp_{ X}}\prod_{x\in\gamma}\left(\int_{\frakX}
\re^{-f(x,\bar{y})}\,\eta(\rd\bar{y})\right)\pi_{\lambda}(\rd\gamma)\\
&=\exp\left\{-\int_{ X}\left(1-\int_{\frakX}
\re^{-f(x,\bar{y})}\,\eta(\rd\bar{y})\right)\lambda(\rd x)
\right\}\\[.3pc]
&=\exp\left\{-\int_{ X}\int_{\frakX}\left(1-
\re^{-f(x,\bar{y})}\right)\lambda(\rd x)\,\eta(\rd\bar{y}) \right\}\\[.3pc]
&=\exp\left\{-\int_{ X\times\frakX}\left(1-
\re^{-f(x,\bar{y})}\right)\widehat\lambda(\rd
x,\rd\bar{y})\right\}\\
&=\int_{\varGamma^\sharp_{ X\times\frakX}}\re^{-\langle
f,\mypp\widehat\gamma\mypp\rangle}\,\pi_{\widehat{\lambda}}(\rd\widehat{\gamma}),
\end{align*}
where we have applied formula (\ref{eq:PoissonLT}) for the Laplace
functional of the Poisson measure $\pi_\lambda$ with the function
$\tilde f(x)=-\ln\left( \int_{\frakX}
\re^{-f(x,\bar{y})}\,\eta(\rd\bar{y})\right)\in\mathrm{M}_+(X)$.
\endproof

\begin{remark}\label{rem0} \normalfont
The measure $\widehat{\mu}$, originally defined on configurations
$\widehat{\gamma}$ of the form (\ref{eq:gamma-hat}), naturally
extends to a probability measure on the entire space
$\varGamma^\sharp_{ X\times\frakX}$.
\end{remark}

\begin{remark}\normalfont
Theorem \ref{th:product-Poisson} can be regarded as a generalization
of the well-known invariance property of Poisson measures under
random i.i.d.\ translations (see, e.g., \cite{CI,DVJ1,Kingman}). A
novel element here is that starting from a Poisson point process in
$X$, random translations create a new (Poisson) point process in a
bigger space, $ X\times\frakX$, with the product intensity measure.
On the other hand, note that the pointwise coordinate projection
$X\times\frakX\ni (x,\bar{y}_x)\mapsto x\in X$ recovers the original
Poisson measure $\pi_{\lambda}$, in accord with the Mapping Theorem
(see Proposition~\ref{pr:mapping}).
% and \cite[\S\,2.3]{Kingman}).
Therefore, Theorem \ref{th:product-Poisson} provides a converse
counterpart to the Mapping Theorem. To the best of our knowledge,
these interesting properties of Poisson measures have not been
pointed out in the literature so far.
\end{remark}

Theorem \ref{th:product-Poisson} can be easily extended to more
general (skew) translations.

\begin{theorem}\label{th:product-Poisson''} Suppose that random
configurations\/ $\widehat{\gamma}_+\in
\varGamma^\sharp_{X\times\frakX}$ are obtained from Poisson
configurations $\gamma\in\varGamma^\sharp_{ X}$ by pointwise
translations $x\mapsto(x,\bar{y}_x+x)$, where $\bar{y}_x\in
\mathfrak{X}$ $(x\in X)$ are i.i.d.\ with common distribution\/
$\eta(\rd\bar{y})$. Then the corresponding probability measure\/
$\widehat{\mu}_+$ on $\varGamma^\sharp_{X\times\frakX}$ coincides
with the Poisson measure of intensity
\begin{equation}\label{eq:hat-sigma}
\widehat{\lambda}_+(\rd x,\rd\bar{y}):=\lambda(\rd
x)\,\eta(\rd\bar{y}-x).
\end{equation}
\end{theorem}

\begin{corollary}\label{cor:sigma*}
Under the pointwise projection $(x,\bar{y})\mapsto\bar{y}$ applied
to configurations\/ $\widehat{\gamma}_+\in
\varGamma^\sharp_{X\times\frakX}$, the Poisson measure\/
$\widehat{\mu}_+$ of Theorem \textup{\ref{th:product-Poisson''}} is
pushed forward to the Poisson measure $\pi_{\lambda^{\star}}$ on
$\varGamma^\sharp_{\frakX}$ with intensity measure $\lambda^{\star}$
defined in \textup{(\ref{eq:sigma*})}.
\end{corollary}
\proof By the Mapping Theorem (see Proposition \ref{pr:mapping}),
the image of the measure $\widehat{\mu}_+$ under the projection
$(x,\bar{y}+x)\mapsto \bar{y}+x$ is a Poisson measure with intensity
given by the push-forward of the measure (\ref{eq:hat-sigma}), that
is,
$$
\int_{ X}\widehat{\lambda}_+(\rd x,B)=\int_{
X}\eta(B-x)\,\lambda(\rd x)=\lambda^{\star}(B),\qquad
B\in\calB(\frakX),
$$
according to the definition (\ref{eq:sigma*}).
\endproof

\begin{remark}\normalfont
According to Corollary \ref{cor:sigma*}, $\sigma$-finiteness of the
intensity measure $\lambda^{\star}$ (see Proposition~\ref{pr:pi*})
is not necessary for the existence of the Poisson measure
$\pi_{\lambda^{\star}}$.
\end{remark}

Finally, combining Theorems \ref{th:product-Poisson},
\ref{th:product-Poisson''} and Corollary \ref{cor:sigma*} with
Theorem \ref{th:mucl}, we arrive at the following result.
\begin{theorem}\label{th:product-Poisson'''}
Suppose that all the conditions of Theorems
\textup{\ref{th:product-Poisson}} and
\textup{\ref{th:product-Poisson''}} are fulfilled. Then, under the
composition mapping
$$
\tilde{\mathfrak{p}}: (x,\bar{y})\mapsto (x,\bar{y}+x)\mapsto
\bar{y}+x \mapsto \mathfrak{p}(\bar{y}+x),
$$
the Poisson measure $\pi_{\widehat\lambda}$ constructed in Theorem
\textup{\ref{th:product-Poisson}} is pushed forward from the space
$\varGamma^\sharp_{X\times \frakX}$ directly to the space
$\varGamma^\sharp_{X}$ where it coincides with the prescribed
Poisson cluster measure $\mucl$\/,
\begin{equation*}
(\tilde{\mathfrak{p}}^*\pi_{\widehat\lambda})(A)
\equiv\pi_{\widehat\lambda}(\tilde{\mathfrak{p}}^{-1}(A))=\mucl(A),\qquad
A\in\calB(\varGamma^\sharp_ X).
\end{equation*}
\end{theorem}
\begin{remark}\normalfont
The construction used in Theorem \ref{th:product-Poisson'''} may
prove instrumental for more complex (e.g., Gibbs) cluster processes,
as it enables one to avoid the intermediate space
$\varGamma^\sharp_{\frakX}$ where the push-forward measure
(analogous to $\pi_{\lambda^{\star}}$) may have no explicit
description.
\end{remark}

\section{Quasi-invariance and integration by
parts}\label{sec:QI-IBP}

From now on, we restrict ourselves to the case where
$X=\mathbb{R}^d$. We shall assume throughout that conditions (a-i)
and (a-ii) of Theorem \ref{th:properClusterPoisson} are fulfilled,
so that $\mucl$-a.a.\ configurations $\gamma\in\varGamma^\sharp_X$
are locally finite. Furthermore, all clusters are assumed to be
a.s.\ finite, hence $p_\infty\equiv\mu_0\{\nu_0=\infty\}=0$ and the
component $X^\infty$ may be dropped from the disjoint union
$\frakX=\bigsqcup_{n}\myn X^n$. We shall also require the absolute
continuity of the measure $\lambda^\star$ (see the corresponding
necessary and sufficient conditions in
Corollary~\ref{cor:sigma*a.c.}). By Proposition~\ref{pr:sigma*a.c.},
this implies that configurations $\gamma$ are $\mucl$-a.s.\ simple
(i.e., have no multiple points). In particular, these assumptions
ensure that $\mucl$-a.a.\ configurations $\gamma$ belong to the
proper configuration space $\varGamma_X$.

Under these conditions, in this section we prove the
quasi-invariance of the measure $\mucl$ with respect to the action
of compactly supported diffeomorphisms of $X$ and establish an
integration-by-parts formula. We begin with a brief description of
some convenient ``manifold-like'' concepts and notations first
introduced in \cite{AKR1}, which provide the suitable framework for
analysis on configuration spaces.

\subsection{Differentiable functions on configuration spaces}\label{app2}

Let $T_{x} X$ be the tangent space of $X=\RR^d$ at point $x\in X$.
It can be identified in the natural way with $\mathbb{R}^{d}$, with
the corresponding (canonical) inner product denoted by a ``fat''
dot~$\CD$\,. The gradient on $X$ is denoted by $\nabla$. Following
\cite{AKR1}, we define the ``tangent space'' of the configuration
space $\varGamma_{X}$ at $\gamma\in\varGamma_{X}$ as the Hilbert
space $T_{\gamma }\varGamma _{ X}:=L^{2}( X\rightarrow T
X;\,\rd\gamma )$, or equivalently
$T_{\gamma}\varGamma_{X\myn}=\bigoplus_{x\in\gamma} T_{x} X$. The
scalar product in $T_{\gamma }\varGamma _{ X}$ is denoted by
$\langle \cdot,\cdot \rangle _{\gamma }$. A vector field $V$ over
$\varGamma _{ X}$ is a mapping $\varGamma _{ X}\ni \gamma \mapsto
V(\gamma )=(V(\gamma )_{x})_{x\in \gamma }\in T_{\gamma }\varGamma
_{X}$. Thus, for vector fields $V_1,V_2$ over $\varGamma_{ X}$ we
have
\begin{equation*}
\left\langle V_1(\gamma ),V_2(\gamma )\right\rangle _{\gamma
}=\sum_{x\in \gamma } V_1(\gamma )_{x}\CD V_2(\gamma
)_{x}\mypp,\qquad \gamma\in \varGamma_X.
\end{equation*}

For $\gamma \in \varGamma _{ X}$ and $x\in \gamma $, denote by
$\mathcal{O}_{\gamma ,\myp x}$ an arbitrary open neighbourhood of
$x$ in $ X$ such that $\mathcal{O}_{\gamma ,\myp x}\cap \gamma
=\{x\}$. For any measurable function $F:\varGamma _{ X}\rightarrow
{\mathbb{R}}$, define the function $F_{x}(\gamma ,\cdot
):\mathcal{O}_{\gamma,\myp x} \rightarrow \allowbreak \mathbb{R}$ by
$F_{x}(\gamma ,y):=F((\gamma \setminus \{x\})\cup \{y\})$, and set
\begin{equation*}
\nabla _{\mynn x}F(\gamma ):=\left. \nabla F_{x}(\gamma
,y)\right|_{y=x},\qquad x\in X,
\end{equation*}
provided $F_{x}(\gamma,\cdot )$ is differentiable at $x$.

Denote by $\mathcal{FC}(\varGamma _{X})$ the class of functions on
$\varGamma _{ X}$ of the form
\begin{equation}\label{local-funct}
F(\gamma )=f(\langle \phi_{1},\gamma \rangle ,\dots ,\langle
\phi_{k},\gamma \rangle ),\qquad \gamma \in \varGamma _{ X},
\end{equation}
where $k\in \mathbb{N}$, \,$f\in C_{b}^{\infty }(\mathbb{R}^{k})$
($:=$ the set of $C^{\infty }$-functions on $\mathbb{R}^{k}$ bounded
together with all their derivatives), and
$\phi_{1},\dots,\phi_{k}\in C_{0}^{\infty }( X)$ ($:=$ the set of
$C^{\infty }$-functions on $ X$ with compact support). Each $F\in
\mathcal{FC}(\varGamma _{ X})$ is local, that is, there is a compact
set $K\subset X$ (which may depend on $F$) such that
$F(\gamma)=F(\gamma_K)$ for all $\gamma \in \varGamma _{ X}$. Thus,
for a fixed $\gamma $ there are only finitely many non-zero
derivatives $\nabla _{\myn x}F(\gamma)$.

For a function $F\in \mathcal{FC}(\varGamma _{X})$, its
$\varGamma$-gradient $\nabla^{\varGamma\myn}F$ is defined as
follows:
\begin{equation}\label{eq:G-gradient}
\nabla^{\varGamma\myn} F(\gamma ):=(\nabla_{\myn x}F(\gamma ))_{x\in
\gamma }\in T_{\gamma }\varGamma _{ X},\qquad\gamma \in \varGamma
_{X},
\end{equation}
so the directional derivative of $F$ along a vector field $V$ is
given by
\begin{equation*}
\nabla_{\myn V}^{\varGamma\myn}F(\gamma ):=\langle \nabla^{\varGamma
}\mynn F(\gamma ),V(\gamma )\rangle _{\gamma }=\sum_{x\in \gamma}
\nabla _{\mynn x}F(\gamma )\myn\CD V(\gamma )_{x},\qquad \gamma \in
\varGamma _{X}.
\end{equation*}
Note that the sum on the right-hand side contains only finitely many
non-zero terms. Further, let $\mathcal{FV}(\varGamma_{X})$ be the
class of cylinder vector fields $V$ on $\varGamma _{ X}$ of the form
\begin{equation}\label{vf}
V(\gamma )_{x}=\sum_{i=1}^{k}A_{i}(\gamma )\mypp v_{i}(x)\in T_{x}
X,\qquad x\in X,
\end{equation}
where $A_{i}\in \mathcal{FC}(\varGamma _{ X})$ and $v_{i}\in
\Vect_{0}( X)$ ($:=$ the space of compactly supported
$C^\infty$-smooth vector fields on $X$), \,$i=1,\dots,k$ \,($k\in
\mathbb{N}$). Any vector filed $v\in \Vect_{0}( X)$ generates a
constant vector field $V$ on $\varGamma _{ X}$ defined by
$V(\gamma)_{x}:=v(x)$. We shall preserve the notation $v$ for it.
Thus,
\begin{equation}\label{eq:grad-new}
\nabla_{\myn v}^{\varGamma }F(\gamma )=\sum_{x\in\gamma }
\nabla_{\myn x}F(\gamma )\myn\CD v(x),\qquad \gamma\in\varGamma_X.
\end{equation}

Recall (see Proposition \ref{pr:properPoisson}\myp(a)) that if
$\lambda (\varLambda )<\infty$ then $\gamma(\varLambda)<\infty$ for
$\pi _{\lambda }$-a.a. $\gamma \in \varGamma _{ X}$. This motivates
the definition of the class $\mathcal{FC}_{\lambda }(\varGamma _{
X})$ of functions on $\varGamma _{ X}$ of the form
(\ref{local-funct}), where $\phi_{1},\dots ,\phi_{k}$ are $C^{\infty
}$-functions with $\lambda (\supp\phi_{i})<\infty $, $i=1,\dots ,k$.
Any function $F\in \mathcal{FC}_{\lambda }(\varGamma _{ X})$ is
local in the sense that there exists a set $B\in\mathcal{B}(X)$
(depending on $F$) such that $\lambda (B)<\infty $ and $F(\gamma
)=F(\gamma_B)$ for all $\gamma \in \varGamma _{ X}$. As in the case
of functions from $\mathcal{FC}(\varGamma _{ X})$, for a fixed
$\gamma$ there are only finitely many non-zero derivatives
$\nabla_{\mynn x}F(\gamma)$.

The approach based on ``lifting'' the differential structure from
the underlying space $X$ to the configuration space $\varGamma_X$ as
described above can also be applied to the spaces
$\frakX=\bigsqcup_{n=0}^\infty X^n$ and $\varGamma_\frakX$. First of
all, the space $\frakX$ is endowed with the natural differential
structure inherited from the constituent spaces $X^n$. Namely, the
tangent space of $\frakX$ at point $\bar{x}\in\frakX$ is defined
piecewise as $T_{\bar{x}}\frakX: = T_{\bar{x}}X^n$ for $\bar{x}\in
X^n$ ($n\in\ZZ_+$), with the scalar product in $T_{\bar{x}}\frakX$
induced from the tangent spaces $T_{\bar{x}}X^n$ and again denoted
by the dot~$\CD$\mypp; furthermore, for a function $f :
\frakX\to\RR$ its gradient $\nabla\myn f$ acts on each space $X^n$
as $\nabla\myn f(\bar{x}) = (\nabla_{\myn
x_1}f(\bar{x}),\dots,\nabla_{\myn x_n}f(\bar{x}))\in
T_{\bar{x}}X^n$, where $\nabla_{\myn x_i}$ is the ``partial''
gradient with respect to the component $x_i\in \bar{x}\in X^n$. A
vector field on $\frakX$ is a map $\mathfrak{X}\ni \bar{x}\mapsto
V(\bar{x})\in T_{\bar{x}}\mathfrak{X}$; in other words, the
restriction of $V$ to $X^n$ is a vector field on $X^{n}$
($n\in\ZZ_+$). The derivative of a function $f:\frakX\to\RR$ along a
vector field $V$ on $\frakX$ is then defined by $\nabla_{\myn V}
f(\bar{x}): = \nabla\myn f(\bar{x}) \CD V (\bar{x})$
\,($\bar{x}\in\frakX$).

The functional class $C^\infty(\frakX)$ is defined, as usual, as the
set of $C^\infty$-functions $f:\frakX\rightarrow \RR$; similarly,
$C^\infty_0(\frakX)$ is the subclass of $C^\infty(\frakX)$
consisting of functions with compact support. Since
differentiability is a local property, $C^\infty(\frakX)$ admits a
component-wise description: $f\in C^\infty(\frakX)$ if and only if
for each $n\in\ZZ_+$ the restriction of $f$ to $X^n$ is in
$C^\infty(X^n)$. However, this is not true for the class
$C^k_0(\frakX)$ which, according to Remark \ref{rm:compact},
involves a stronger condition that $f(\bar{x})\equiv \bar{x}$
\,($\bar{x}\in X^n$) for all large enough~$n$.

Now, lifting this differentiable structure from the space $\frakX$
to the configuration space $\varGamma_\frakX$ can be done by
repeating the same constructions as before with only obvious
modifications, so we do not dwell on details. This way, we introduce
the tangent space $T_{\bar{\gamma}}\varGamma
_{\frakX}=\bigoplus_{\bar{x}\in \bar{\gamma}}T_{\bar{x}}{\frakX}$,
vector fields $V$ over $\varGamma_{\frakX}$, and differentiable
functions $\varPhi:\varGamma_{\frakX}\to\RR$. Similarly to
(\ref{local-funct}) and (\ref{vf}) one can define the spaces
$\mathcal{FC}(\varGamma_{\frakX})$, $\mathcal{FC}_{\lambda ^{\star
}}(\varGamma _{{\frakX}})$ and $\mathcal{FV}(\varGamma _{ \frakX})$
of $C^\infty$-smooth local functions and vector fields on $\frakX$,
and we shall use these notations without further explanation.

\subsection{$\Diff_{0}$-quasi-invariance}
\label{sec:QI-mu}

In this section, we discuss the property of quasi-invariance of the
measure $\mucl$ with respect to diffeomorphisms of $X$. Let us start
by describing how diffeomorphisms of $X$ act on configuration
spaces. For a measurable mapping $\varphi: X\to X$, its
\emph{support} $\supp\varphi$ is defined as the smallest closed set
containing all $x\in X$ such that $\varphi(x)\ne x$. Let $\Diff_{0}(
X)$ be the group of diffeomorphisms of $ X$ with \emph{compact
support}. For any $\varphi \in \Diff_{0}( X)$, we define the
``diagonal'' diffeomorphism $\bar{\varphi}: \frakX\rightarrow
\frakX$ acting on each space $X^n$ ($n\in\ZZ_+$) as follows:
\begin{equation*}
X^n\ni \bar{x}=(x_{1},\dots ,x_{n})\mapsto
\bar{\varphi}(\bar{x}):=(\varphi (x_{1}),\dots ,\varphi (x_{n}))\in
X^n.
\end{equation*}

\begin{remark}\label{supp}
\normalfont Although $K:=\supp\varphi$ is compact in $X$, note that
the support of the diffeomorphism $\bar {\varphi }$ (again defined
as the closure of the set $\{\bar{x}\in \frakX: \varphi(\bar{x})\ne
\bar{x}\}$) is given by $\supp\bar{\varphi}=\frakX_{K}$
(see~(\ref{eq:K})) and hence is \emph{not} compact in the topology
of $\frakX$ (see Remark~\ref{rm:compact}). However,
$\lambda^{\star}(\frakX_{K})<\infty$ (by Proposition \ref{prop1}),
which is sufficient for our purposes.
\end{remark}

The mappings $\varphi $ and $\bar{\varphi}$ can be lifted to
measurable ``diagonal'' transformations (denoted by the same
letters) of the configuration spaces $\varGamma _{ X}$ and
$\varGamma_{\frakX}$, respectively:
\begin{equation}\label{di*}
\begin{aligned}
\varGamma_{X}\ni\gamma\mapsto \varphi(\gamma): ={}&\{\varphi(x),\
x\in\gamma\}\in \varGamma_{X},\\[.2pc]
\varGamma_{\frakX}\ni\bar{\gamma}\mapsto
\bar{\varphi}(\bar{\gamma}):={}&\{\bar{\varphi}(\bar{x}),\
\bar{x}\in\bar{\gamma}\}\in \varGamma_{\frakX}\myp.
\end{aligned}
\end{equation}

Let ${\mathcal{I}}:L^{2}(\varGamma _{ X},\mucl)\rightarrow
L^{2}(\varGamma _{\frakX},\pi _{\lambda^{\star}})$ be the isometry
defined by the projection $\mathfrak{p}$,
\begin{equation}\label{eq:I}
({\mathcal{I}}F) (\bar{\gamma
}):=F(\mathfrak{p}(\bar{\gamma})),\qquad
\bar{\gamma}\in\varGamma_{\frakX},
\end{equation}
and let ${\mathcal{I}}^{*}:L^{2}(\varGamma _{\frakX},\pi
_{\lambda^{\star}})\rightarrow L^{2}(\varGamma _{ X},\mucl)$ be the
adjoint operator.

\begin{remark}\label{rem:I}
\normalfont The definition implies that $\mathcal{I}^*\mathcal{I}$
is the identity operator in $L^{2}(\varGamma _{X},\mucl)$. However,
the operator $\mathcal{I}\mathcal{I}^*$ acting in the space
$L^2(\varGamma_\frakX,\pi_{\lambda^\star})$ is a non-trivial
orthogonal projection, which plays the role of an infinite particle
symmetrization operator. Unfortunately, general explicit form of the
operators $\mathcal{I}^*$ and $\mathcal{I}\mathcal{I}^*$ is not
known, and may be hard to obtain.
\end{remark}

By the next lemma, the action of $\Diff_{0}( X)$ commutes with the
operators $\mathfrak{p}$ and~${\mathcal{I}}$.

\begin{lemma}\label{lm:comm1}
For any $\varphi\in\Diff_0( X)$, we have $\varphi\circ
\mathfrak{p}=\mathfrak{p}\circ \bar{\varphi}$ and furthermore,
${\mathcal{I}}(F\circ \varphi )=({\mathcal{I}}F)\circ \bar{\varphi}$
\,for any\/ $F\in L^{2}(\varGamma _{ X},\mucl)$.
\end{lemma}

\proof The first statement  follows from the definition
(\ref{eq:proj}) of the mapping $\mathfrak{p}$ and the diagonal form
of $\bar{\varphi}$ (see~(\ref{di*})). The second statement then
readily follows by the definition (\ref{eq:I}) of the
operator~$\mathcal{I}$.
\endproof

Let us now consider the configuration space $\varGamma _{\frakX}$
equipped with the Poisson measure $\pi _{\lambda^{\star}}$
introduced in Section~\ref{sec:3.2}. As already mentioned, we assume
that the intensity measure $\lambda^{\star}$ is a.c.\ with respect
to the Lebesgue measure on $\mathfrak{X}$ and, moreover,
\begin{equation}\label{QI}
s(\bar{x}):=\frac{\lambda^{\star}(\rd\bar{x})}{\rd\bar{x}}>0\qquad
\text{for \,a.a.}\ \,\bar{x}\in\mathfrak{X}.
\end{equation}
This implies that the measure $\lambda^{\star}$ is quasi-invariant
with respect to the action of diagonal transformations
$\bar{\varphi}:\frakX\to\frakX$ \,($\varphi \in \Diff_{0}( X)$) and
the corresponding Radon--Nikodym derivative is given by
\begin{equation}\label{density'}
\rho _{\lambda^{\star}}^{\bar{\varphi
}}(\bar{x})=\frac{s(\bar{\varphi
}^{-1}(\bar{x}))}{s(\bar{x})}\,J_{\bar{\varphi}}(\bar{x})^{-1}\qquad
\text{for \ $\lambda^{\star}$-a.a. }\,\bar{x},
\end{equation}
where $J_{\bar{\varphi}}$ is the Jacobian determinant of
$\bar{\varphi }$ (we set $\rho _{\lambda^{\star}}^{\bar{\varphi}}
(\bar{x})=1$ if $s(\bar{x})=0$ or
$s(\bar{\varphi}^{-1}(\bar{x}))=0$).

\begin{proposition}\label{Poisson-qi}
The Poisson measure $\pi _{\lambda^{\star}}\!$ is quasi-invariant
with respect to the action of\/ diagonal diffeomorphisms\/
$\bar{\varphi}:\varGamma_{\frakX}\to\varGamma_{\frakX}$
\,\textup{(}$\varphi \in \Diff_{0}(X)$\textup{)}. The corresponding
Radon--Nikodym density $R_{\pi _{\lambda^{\star}}}^{\bar{\varphi
}}\!:=\rd(\bar\varphi^*\pi _{\lambda^{\star}})/\rd \pi
_{\lambda^{\star}}$ is given by
\begin{equation}\label{RND}
R_{\pi _{\lambda^{\star}}}^{\bar{\varphi }}(\bar{\gamma })= \exp
\left\{\int_{\frakX}\bigl(1-\rho _{\lambda^{\star}}^{\bar{\varphi
}}(\bar{x})\bigr)\,\lambda^{\star}(\rd\bar{x})\right\}\cdot\prod_{\bar{x}\in
\bar{\gamma }}\rho _{\lambda^{\star}}^{\bar{\varphi
}}(\bar{x}),\qquad \bar{\gamma}\in\varGamma_{\frakX},
\end{equation}
where $\rho _{\lambda^{\star}}^{\bar{\varphi}}$ is defined in\/
\textup{(\ref{density'})}.
\end{proposition}

\proof The result follows from Remark \ref{supp} and Proposition
\ref{q-i-Poisson} in the Appendix below (applied to the space
$\mathfrak{X}$ with measure $\lambda^{\star}$ and mapping
$\bar{\varphi}$).
\endproof

\begin{remark}\normalfont
The function $R_{\pi_{\lambda^{\star}}}^{\bar{\varphi}}$ is local in
the sense that, for $\pi_{\lambda^{\star}}$-a.a.\
$\bar{\gamma}\in\varGamma_{\frakX}$, we have
$R_{\pi_{\lambda^{\star}}}^{\bar{\varphi}}(\bar{\gamma})
=R_{\pi_{\lambda^{\star}}}^{\bar{\varphi}}(\bar{\gamma }\cap
\frakX_{K})$, where $K:=\supp\varphi $.
\end{remark}

\begin{remark}
[Explicit form of $R_{\pi_{\lambda^{\star}}}^{\bar{\varphi}}$]
\label{Explicit} \normalfont Let the measure $\eta(\rd\bar{y})$ be
a.c.\ with respect to Lebesgue measure $\rd\bar{y}$ on $\frakX$,
with density $h(\bar{y})$ (see~(\ref{eq:m})). According to
(\ref{density'}),
\begin{equation*}
\rho_{\lambda^{\star}}^{\bar{\varphi }}(\bar{y})=\frac{\int_{ X}
h(\varphi ^{-1}(y_{1})-x,\dots ,\varphi ^{-1}(y_{n})-x)\,\lambda
(\rd x)}{\int_{ X}h(y_{1}-x,\dots ,y_{n}-x)\,\lambda (\rd
x)}\prod_{i=1}^n J_{\varphi}(y_{i})^{-1},\qquad \bar{y}\in X^n,
\end{equation*}
where $J_{\varphi}(\bar{y})=\det(\partial \varphi_i/\partial y_j)$
is the Jacobian determinant of $\varphi $ (note that
$J_{\bar{\varphi}}(\bar{y})=\prod_{i=1}^n J_{\varphi}(y_{i})$ for
$\bar{y}\in X^n$). Then $R_{\pi _{\lambda^{\star}}}^{\bar{\varphi
}}(\bar{\gamma })$ can be calculated using formula (\ref{RND}). In
particular, if clusters have i.i.d.\ points, so that
$h(\bar{y})=\prod_{i=1}^n h_0(y_{i})$, then
\begin{equation*}
\rho _{\lambda^{\star}}^{\bar{\varphi }}(\bar{y})=\frac{\int_{ X}
\prod_{i=1}^n J_{\varphi}(y_{i})^{-1}\myp h_0(\varphi
^{-1}(y_{i})-x)\,\lambda (\rd x)}{\int_{ X} \prod_{i=1}^n
h_0(y_{i}-x)\,\lambda (\rd x)}\mypp, \qquad
\bar{y}=(y_1,\dots,y_n)\in X^n,
\end{equation*}
and
\begin{equation*}
R_{\pi _{\lambda^{\star}}}^{\bar{\varphi }}(\bar{\gamma})=
C\prod_{\bar{y}\in \bar{\gamma }}\frac{\int_{ X} \prod_{y\in
\bar{y}}J_{\varphi}(y)^{-1}\myp h_0(\varphi ^{-1}(y)-x)\,\lambda
(\rd x)}{\int_{ X} \prod_{y\in \bar{y}}h_0(y-x)\,\lambda (\rd
x)}\mypp,\qquad \bar{\gamma}\in\varGamma_{\mathfrak{X}},
\end{equation*}
where\/ $C:=\exp\left\{\int_{\mathfrak{X}}(1-\rho
_{\lambda^{\star}}^{\bar{\varphi }}(\bar{y}))\,
\lambda^{\star}(\rd\bar{y})\right\}$ is a normalizing constant.
\end{remark}

Now we can prove the main result of this section.

\begin{theorem}\label{q-inv}
Under condition \textup{(\ref{QI})}, the Poisson cluster measure
$\mucl$ on $\varGamma_{X}$ is quasi-invariant with respect to the
action of\/ $\Diff_{0}({ X})$ on $\varGamma_{X}$. The Radon--Nikodym
density $R_{\mucl}^{\varphi}:=\rd(\varphi^*\mucl)/\rd\mucl$ is given
by $R_{\mucl}^{\varphi}={\mathcal{I}}^{*}
R_{\pi_{\lambda^{\star}}}^{\bar{\varphi}}$, where the density
$R_{\pi_{\lambda^{\star\!}}}^{\bar{\varphi}}=\rd(\bar\varphi^*\pi
_{\lambda^{\star}})/\rd \pi _{\lambda^{\star}}$ is defined in
\textup{(\ref{RND})}.
\end{theorem}

\proof According to Theorem \ref{th:mucl} (see (\ref{eq:mu*})) and
Lemma \ref{lm:comm1},
\begin{equation*}
\begin{aligned}
\varphi ^{*}\mucl&= (\mathfrak{p}^{*} \pi_{\lambda^{\star\!}})\circ
\varphi ^{-1}=
\pi_{\lambda^{\star\!}}\circ(\varphi\circ\mathfrak{p})^{-1}\\
&= \pi_{\lambda^{\star\!}}\circ(\mathfrak{p}\circ\bar\varphi)^{-1}=
(\bar\varphi^{*} \pi_{\lambda^{\star\!}})\circ \mathfrak{p}
^{-1}=\mathfrak{p}^*(\bar\varphi^{*} \pi_{\lambda^{\star\!}}).
\end{aligned}
\end{equation*}
Hence, by the change of variables $\gamma=\mathfrak{p}(\bar\gamma)$,
for any non-negative measurable function $F$ on $\varGamma_X$ we
obtain
\begin{gather*}
\int_{\varGamma_{ X}}F(\gamma )\,(\varphi^{*} \mucl)(\rd \gamma )
=\int_{\varGamma _{X}} F(\gamma)\,\mathfrak{p}^*(\bar\varphi^{*}
\pi_{\lambda^{\star\!}})(\rd\gamma) =\int_{\varGamma
_{\frakX}}{\mathcal{I}}F(\bar{\gamma
})\,(\bar{\varphi}^{*}\,\pi _{\lambda^{\star}})(\rd\bar{\gamma }) \\
 =\int_{\varGamma _{\frakX}}{\mathcal{I}}F(\bar{\gamma })
\,R_{\pi_{\lambda^{\star\!}}}^{\bar{\varphi}}(\bar{\gamma })\;\pi
_{\lambda^{\star}}(\rd\bar{\gamma }) =\int_{\varGamma _{ X}}F(\gamma
)\,( {\mathcal{I}}^{*}
R_{\pi_{\lambda^{\star\!}}}^{\bar{\varphi}})(\gamma)\,\mucl(\rd\gamma),
\end{gather*}
where we have also used formula (\ref{eq:I}) and
Proposition~\ref{Poisson-qi}. Thus, the measure $\varphi^{*}\mucl$
is a.c.\ with respect to the measure $\mucl$, with the
Radon--Nikodym density $R_{\mucl}^{\varphi}={\mathcal{I}}^{*}R_{\pi
_{\lambda^{\star}}}^{\bar{\varphi }}$, and the theorem is proved.
\endproof

\begin{remark}\label{rem:R}\normalfont
We do not know an explicit form of the density $R_{\mucl}^{\varphi}$
(cf.\ Remark~\ref{rem:I}).
\end{remark}

\begin{remark}\normalfont
The Poisson cluster measure $\mucl$ on the configuration space
$\varGamma _{X}$ can be used to construct the canonical unitary
representation $U$ of the diffeomorphism group $\Diff_{0}(X)$ by
operators in $L^{2}(\varGamma_{X},\mucl)$, given by the formula
\begin{equation*}
U_{\varphi}F(\gamma )=\sqrt{R_{\mucl}^{\varphi }(\gamma
)}\,F(\varphi ^{-1}(\gamma )),\qquad F\in L^{2}(\varGamma
_{X},\mucl).
\end{equation*}
Such representations, which can be defined for arbitrary
quasi-invariant measures on $\varGamma _{ X}$, play a significant
role in the representation theory of the diffeomorphism group
$\Diff_{0}( X)$ \cite{Ism,VGG} and quantum field theory
\cite{GGPS,Goldin}. An important question is whether the
representation $U$ is irreducible. According to \cite{VGG}, this is
equivalent to the $\Diff_{0}( X)$-ergodicity of the measure $\mucl$,
which in our case is equivalent to the ergodicity of the measure
$\pi_{\lambda^\star}$ with respect to the group of transformations
$\bar{\varphi}$, where $\varphi \in \Diff_{0}( X)$. The latter is an
open question.
\end{remark}

\subsection{Integration-by-parts formula}\label{sec:4.2}

The main objective of this section is to establish an
integration-by-parts (IBP) formula for the Poisson cluster measure
$\mucl$, in the spirit of the IBP formula for Poisson measures
proved in \cite{AKR1}. To this end, we shall use the projection
operator $\mathfrak{p}$ and the properties of the auxiliary Poisson
measure $\pi_{\lambda^{\star}}$. Since our framework is somewhat
different from that in \cite{AKR1}, we give a proof of the IBP
formula for $\pi_{\lambda^{\star}}$.

First, recall that the classical IBP formula for a Borel measure
$\varpi$ on a Euclidean space $\mathbb{R}^{m}$ (see, e.g.,
\cite[Ch.~5]{Bo}) is expressed by the following identity that should
hold for any vector field $v\in \Vect_{0}(\mathbb{R}^{m})$ and all
functions $f,\myp g\in C_{0}^{\infty}(\mathbb{R}^{m})$:
\begin{equation}\label{ibp0}
\begin{aligned} \int_{\mathbb{R}^{m}}f(y)\mypp\nabla_{\mynn v}\mypp g(y)\,\varpi
(\rd y)={}&-\int_{\mathbb{R}^{m}}g(y)\mypp\nabla_{\mynn v}\myp
f(y)\,\varpi (\rd
y)\\[.2pc]
&-\int_{\mathbb{R}^{m}}f(y)\mypp g(y)\mypp\beta_{\varpi}^{\myp
v}(y)\,\varpi (\rd y),
\end{aligned}
\end{equation}
where $\nabla_{\mynn v}\myp \phi(y)$ is the derivative of $\phi$
along $v$ at point $y\in Y$ and $\beta_{\varpi}^{\myp v}\in
L_{\mathrm{loc}}^{1}(\mathbb{R}^{m},\varpi \mathbb{)}$ is a
measurable function called the \emph{logarithmic derivative }of
$\varpi $ along the vector field~$v$. It is easy to see that
$\beta_{\varpi}^{\myp v}$ can be represented in the form
\begin{equation*}
\beta_{\varpi}^{\myp v}(y)=\beta_{\varpi}(y)\myn\CD v(y)+\Div v(y),
\end{equation*}
where the corresponding mapping
$\beta_{\varpi}:\mathbb{R}^{m}\rightarrow \mathbb{R}^{m}$ is called
\emph{vector logarithmic derivative} of $\varpi $. Suppose that the
measure $\varpi $ is a.c.\ with respect to the Lebesgue measure $\rd
y$, with density $w$ such that $w^{1/2}\in
H_{\mathrm{loc}}^{1,2}(\mathbb{R}^{m})$ ($:=$ the local Sobolev
space of order $1$ in $L^2(\mathbb{R}^{m};\rd y)$, i.e., the space
of functions on $\mathbb{R}^{m}$ whose first-order partial
derivatives are locally square integrable). Then the measure
$\varpi$ satisfies the IBP formula (\ref{ibp0}) with the vector
logarithmic derivative $\beta _{\varpi}(y)= w(y)^{-1}\myp\nabla
w(y)$ (note that $w(y)\ne 0$ for $\varpi $-a.a.\ $y\in
\mathbb{R}^{m}$).

Assume that the density
$s(\bar{x})=\lambda^{\star}(\rd\bar{x})/\rd\bar{x}$
\,($\bar{x}\in\frakX$) satisfies the condition $s^{1/2}\in
H_{\mathrm{loc}}^{1,2}(\frakX)$ ($:=$ the local Sobolev space of
order $1$ in $L^2(\frakX; \rd\bar{x})$). By formula (\ref{density})
and decompositions (\ref{eq:star}) and (\ref{eq:s-sn}), the latter
condition is equivalent to the set of analogous conditions for the
restrictions of $s(\bar{x})$ to the spaces $X^n$. That is, assuming
without loss of generality that $p_n\ne 0$, for each
$s_n(\bar{x})=\lambda_n^{\star}(\rd\bar{x})/\rd\bar{x}$ ($\bar{x}\in
X^n$) we have $s_{n}^{1/2}\in H_{\mathrm{loc}}^{1,2}( X^{n})$. By
the general result alluded to above, this ensures that the IBP
formula holds for each measure $\lambda^\star_{n}$\myp, with the
vector logarithmic derivative $\beta_{\lambda^\star_{n}}(\bar{x})=
(\beta_{1}(\bar{x}),\dots, \beta_{n}(\bar{x}))$ \,($\bar{x}\in
X^n$), where
\begin{equation}\label{5.8}
\beta_{i}(\bar{x}):=\frac{\nabla_{\myn
i\,}s_{n}(\bar{x})}{s_{n}(\bar{x})} =\frac{\int_{ X}\nabla _{\myn
i\,}h_{n}(x_{1}-x,\dots ,x_{n}-x)\,\lambda(\rd x)}{\int_{
X}h_{n}(x_{1}-x,\dots,x_{n}-x)\, \lambda(\rd x)}
\end{equation}
if $s_{n}(\bar{x})\neq 0$ and $\beta _{i}(\bar{x}):=0$ if
$s_{n}(\bar{x})=0$.

For any $v\in \Vect_{0}( X)$, let us define the vector field
$\bar{v}$ on $\frakX$ by setting
\begin{equation}\label{eq:v-bar}
\bar{v}(\bar{x}):=(v(x_{1}),\dots,v(x_n)),\quad
\bar{x}=(x_{1},\dots,x_{n})\in  X^n\quad (n\in\mathbb{Z}_+).
\end{equation}
The logarithmic derivative of the measure $\lambda^\star_{n}$ along
the vector field $\bar{v}$ is given by
\begin{equation} \label{logder-sigma}
\beta_{\lambda^\star_{n}}^{\myp\bar{v}}(\bar{x})=\sum_{x_i\in\bar{x}}\bigl(
\beta _{i}(\bar{x})\myn\CD v(x_{i})+\Div v(x_{i})\bigr),\qquad
\bar{x}\in X^n .
\end{equation}

\begin{proposition}
The measure $\lambda^{\star}$ satisfies the following IBP formula:
\begin{equation}\label{ibp1}
\begin{aligned}
\int_{\frakX}f(\bar{x})\mypp\nabla_{\myn \bar{v}}\mypp
g(\bar{x})\,\lambda^{\star}(\rd\bar{x})
=&-\int_{\frakX}g(\bar{x})\mypp\nabla_{\myn \bar{v}}\myp f(\bar{x})
\,\lambda^{\star}(\rd\bar{x})\\[.2pc]
&-\int_{\frakX}f(\bar{x})\mypp
g(\bar{x})\mypp\beta_{\lambda^{\star}}^{\myp\bar{v}}(\bar{x})\,\lambda^{\star}(\rd
\bar{x}),
\end{aligned}
\end{equation}
where $f,g\in C_{0}^{\infty}(\frakX)$ and
$\beta_{\lambda^{\star}}^{\myp\bar{v}}(\bar{x})=
\beta_{\lambda^\star_{n}}^{\myp\bar{v}} (\bar{x})$ if\/ $\bar{x}\in
 X^{n}$ \textup{(}$n\in\mathbb{Z}_+$\textup{)}.
\end{proposition}

\proof The result easily follows from the decomposition
(\ref{eq:star}) of the measure $\lambda^{\star}$ and the IBP formula
for each measure $\lambda^{\star}_n$ such that $p_n\ne 0$
\,($n\in\mathbb{Z}_+$).
\endproof

\begin{remark}\normalfont
Formula (\ref{ibp1}) can be rewritten in the form
\begin{equation*}
\begin{aligned}
\int_{\frakX}f(\bar{x})\sum_{x\in \mathfrak{p}(\bar{x})}
\bigl(\nabla_{\mynn x}\mypp g(\bar{x})\myn\CD
v(x)\bigr)\,\lambda^{\star}(\rd\bar{x})=&-\int_{\frakX}g(\bar{x})\sum_{x\in
\mathfrak{p}(\bar{x})}\bigl(\nabla_{\mynn x}\myp f(\bar{x})\myn\CD
v(x)\bigr)\,\lambda^{\star}(\rd\bar{x})\\ &-\int_{\frakX}
f(\bar{x})\mypp
g(\bar{x})\mypp\beta_{\lambda^{\star}}^{\myp\bar{v}}(\bar{x})
\,\lambda^{\star}(\rd\bar{x}).
\end{aligned}
\end{equation*}
\end{remark}

Recall that the functional classes $\mathcal{FC}(\varGamma_{X})$,
$\mathcal{FC}(\varGamma_{{\mathfrak{X}}})$, and
$\mathcal{FC}_{\lambda ^{\star}}(\varGamma_{{\mathfrak{X}}})$ of
local functions on the configuration spaces $\varGamma_{X}$ and
$\varGamma_{{\mathfrak{X}}}$ are defined in Section~\ref{app2}.

\begin{theorem}\label{IBP-}
For each\/ $v\in \Vect_{0}( X)$ and any\/ $F,G\in
\mathcal{FC}(\varGamma _{ X})$, the following IBP formula holds:
\begin{equation}\label{IBP0-}
\begin{aligned}
\int_{\varGamma_{ X}}F(\gamma )\mypp\nabla_{\myn v}^{\varGamma
}G(\gamma)\,\mucl(\rd\gamma )=&-\int_{\varGamma_{ X}}G(\gamma)
\mypp\nabla_{\myn v}^{\varGamma}F(\gamma)\,\mucl(\rd\gamma )\\
&-\int_{\varGamma _{ X}}F(\gamma)\mypp G(\gamma)\myp B_{\mucl}^{\myp
v}(\gamma)\,\mucl(\rd\gamma ),
\end{aligned}
\end{equation}
where $\nabla_{\myn v}^{\varGamma }$ is the $\varGamma$-gradient
along the vector field $v$ defined by \textup{(\ref{eq:grad-new})},
$B_{\mucl}^{\myp v}(\gamma ):={\mathcal{I}}^{*}\langle
\beta_{\lambda^{\star}}^{\myp\bar{v}},\bar{\gamma }\rangle$, and
$\beta_{\lambda^{\star}}^{\myp\bar{v}}$ is the logarithmic
derivative of\/ $\lambda^{\star}$ along the corresponding vector
field\/ $\bar{v}$ \textup{(}see \textup{(\ref{eq:v-bar})}\textup{)}.
\end{theorem}

\proof  Denote
$$
Q(\gamma):=F(\gamma) \mypp\nabla_{\myn v}^{\varGamma }G(\gamma)=
F(\gamma )\sum_{x\in \gamma }\nabla _{\!x}G(\gamma )\myn\CD v(x),
$$
then
\begin{equation}\label{eq:IPhi}
({\mathcal{I}}\mypp Q) (\bar{\gamma })=({\mathcal{I}}\myp
F)(\bar{\gamma })\sum_{x\in \mathfrak{p}(\bar{\gamma })} \nabla
_{\mynn x}\myp G(\mathfrak{p}(\bar{\gamma}))\myn\CD v(x).
\end{equation}
Note that ${\mathcal{I}}\myp Q \in
\mathcal{FC}_{\lambda^{\star}}(\varGamma _{\mathfrak{X}})$, so we
can use (\ref{3.1}) in order to integrate $\mathcal{I}\myp Q$ with
respect to $\pi_{\lambda^{\star}}$. Using Theorem \ref{th:mucl}
(see~(\ref{eq:mu*})) and formula (\ref{eq:IPhi}), we obtain
\begin{align}
\notag &\int_{\varGamma _{X}}F(\gamma)\mypp\nabla_{\myn
v}^{\varGamma} G(\gamma)\,\mucl(\rd\gamma)
=\int_{\varGamma_{\frakX}} ({\mathcal{I}}\myp F)(\bar{\gamma
})\sum_{x\in\mathfrak{p}(\bar{\gamma})} \nabla_{\mynn x}
G(\mathfrak{p}(\bar{\gamma}))\myn\CD v(x)\,
\pi _{\lambda^{\star}}(\rd\bar{\gamma }) \\
\notag
&\hspace{1.2pc}=\re^{-\lambda^{\star}(\frakX_{K})}\sum_{m=0}^{\infty}
\frac{1}{m!}\int_{(\frakX_{K})^{m}}
F(\{\mathfrak{p}(\bar{x}_{1}),\dots,\mathfrak{p}(\bar{x}_{m})\})\\
\notag &\hspace{3.0pc} \times
\sum_{i=1}^{m}\sum_{x\in\mathfrak{p}(\bar{x}_{i})} \nabla_{\mynn
x}G(\{\mathfrak{p}(\bar{x}_{1}),\dots,
\mathfrak{p}(\bar{x}_{m})\})\myn\CD v(x) \,\bigotimes_{i=1}^m
\lambda^{\star} (\rd\bar{x}_{i})\\
\label{eq:r1} &\hspace{1pc}\begin{aligned}[h]
{}&=\re^{-\lambda^{\star}(\frakX_{K})}\sum_{m=0}^\infty \frac{1}{m!}
\sum_{i=1}^{m}\int_{(\frakX_K)^{m-1}}\biggl(\int_{\frakX_{K}}
F(\{\mathfrak{p}(\bar{x}_{1}),\dots,\mathfrak{p}(\bar{x}_{m})\})\\
&\hspace{1.8pc} \times
\sum_{x\in\mathfrak{p}(\bar{x}_{i})}\nabla_{\mynn
x}G(\{\mathfrak{p}(\bar{x}
_{1}),\dots,\mathfrak{p}(\bar{x}_{m})\})\myn\CD
v(x)\,\lambda^{\star} (\rd\bar{x}_{i})\biggr)\bigotimes_{j\ne i}
\lambda^{\star}(\rd\bar{x}_{j}).
\end{aligned}
\end{align}
By the IBP formula for $\lambda^{\star}$, the inner integral in
(\ref{eq:r1}) can be rewritten as
\begin{align*}
&-\int_{\frakX_{K}}G(\{\mathfrak{p}(\bar{x}_{1}),\dots,\mathfrak{p}(\bar{x}_{m})\})
\,\Biggl(\,\sum_{x\in \mathfrak{p}(\bar{x}_{i})}\nabla_{\myn x}
F(\{\mathfrak{p}(\bar{x}_{1}),\dots,\mathfrak{p}(\bar{x}_{m})\})\myn\CD
v(x)\\
&\hspace{14.5pc}+F(\{\mathfrak{p}(\bar{x}_{1}),\dots,\mathfrak{p}(\bar{x}_{m})\})\,\beta
_{\lambda^{\star}}^{\myp\bar{v}}(\bar{x}_{i})\Biggr)\,\lambda^{\star}(\rd\bar{x}_{i}).
\end{align*}
Hence, the right-hand side of (\ref{eq:r1}) is reduced to
\begin{align*}
&-\re^{-\lambda^{\star}(\frakX_{K})}\sum_{m=0}^{\infty}
\frac{1}{m!}\int_{(\frakX_{K})^{m}}G(\{\mathfrak{p}(\bar{x}_{1}),\dots,\mathfrak{p}(\bar{x}_{m})\})\\
&\hspace{1.8pc}\times\Biggl(\sum_{\,x\in
\mathfrak{p}(\{\bar{x}_{1},\dots,\bar{x}_m\})}\nabla_{\myn x}
F(\{\mathfrak{p}(\bar{x}_{1}),\dots,\mathfrak{p}(\bar{x}_{m})\})\myn\CD v(x) \\
&\hspace{6.5pc}+F(\{\mathfrak{p}(\bar{x}_{1}),\dots,\mathfrak{p}(\bar{x}_{m})\})\,B_{\pi
_{\lambda^{\star}}}^{\myp
\bar{v}}(\{\bar{x}_{1},\dots,\bar{x}_{m}\})\Biggr)\bigotimes_{i=1}^m
\lambda^{\star}(\rd\bar{x}_{i})\\
&=-\int_{\varGamma _{\frakX}}G(\mathfrak{p}(\bar{\gamma
}))\left(\sum_{x\in \mathfrak{p}(\bar{\gamma })}\!\nabla_{\!x}
F(\mathfrak{p}(\bar{\gamma }))\myn\CD
v(x)+F(\mathfrak{p}(\bar{\gamma }))\mypp
B_{\pi_{\lambda^{\star}}}^{\myp \bar{v}}(\bar{\gamma })\right)\pi
_{\lambda^{\star}}(\rd\bar{\gamma }) \\
&=-\int_{\varGamma _{X}}G(\gamma )\mypp\nabla_{\myn
v}^{\varGamma}\mynn F(\gamma )\,\mucl(\rd\gamma )-\int_{\varGamma _{
X}}F(\gamma )\myp G(\gamma )\myp B_{\mucl}^{\myp v}(\gamma
)\,\mucl(\rd\gamma ),
\end{align*}
where
\begin{equation}\label{logder-Poisson}
B_{\pi _{\lambda^{\star}}}^{\myp\bar{v}}(\bar{\gamma
}):=\sum_{\bar{x}\in \bar{\gamma
}}\;\beta_{\lambda^{\star}}^{\myp\bar{v}}(\bar{x})=\langle
\beta_{\lambda^{\star}}^{\myp\bar{v}},\bar{\gamma }\rangle,\qquad
\bar{\gamma}\in\varGamma_\frakX,
\end{equation}
and $B_{\mucl}^{v}:={\mathcal{I}}^{*}B_{\pi
_{\lambda^{\star}}}^{\myp\bar{v}}$. Note that
$B_{\pi_{\lambda^{\star}}}^{\myp\bar{v}}$ is well defined since
$\lambda^{\star}(\supp\bar{v})<\infty $, so there are only finitely
many non-zero terms in the sum (\ref{logder-Poisson}). Moreover,
finiteness of the first and second moments of $\pi
_{\lambda^{\star}}$ implies that
$B_{\pi_{\lambda^{\star}}}^{\myp\bar{v}}\!\in L^{2}(\varGamma
_{\frakX},\pi _{\lambda^{\star}})$.
\endproof

\begin{remark}\normalfont
The logarithmic derivative $B_{\pi_{\lambda^{\star}}}^{\myp
\bar{v}}$ can be written in the form (cf.~(\ref{5.8}))
\begin{align*} B_{\pi _{\lambda^{\star}}}^{\myp\bar{v}}(\bar{\gamma})
&=\sum_{\bar{x}\in \bar{\gamma }}\sum_{x_i\in\bar{x}}\bigl(\beta_{i}
(\bar{x})\myn\CD v(x_{i})+\Div v(x_{i})\bigr)\\
&=\sum_{\bar{x}\in \bar{\gamma }}\bigl(\beta _{\lambda^{\star}}
(\bar{x})\myn\CD \bar{v}(\bar{x})+\Div \bar{v}(\bar{x})\bigr),
\qquad \bar{\gamma}\in\varGamma_\frakX.
\end{align*}
\end{remark}

Formula (\ref{IBP0-}) can be extended to more general vector fields
on $\varGamma _{ X}$. For any vector field $V\in
\mathcal{FV}(\varGamma_{X})$ of the form (\ref{vf}), we set
\begin{equation*}
B_{\mucl}^{V}(\gamma ):=\sum_{i=1}^{k}\biggl(A_{i}(\gamma )B_{\mu
}^{v_{i}}(\gamma )+\sum_{x\in \gamma }\nabla _{\mynn x}A_{i}(\gamma
)\myn\CD v_{i}(x)\biggr), \qquad \gamma\in\varGamma_X.
\end{equation*}

\begin{theorem}\label{IBP1}
For any\/ $V\in \mathcal{FV}(\varGamma_{X})$ and all $F,G\in
\mathcal{FC}(\varGamma_{X})$, we have
\begin{equation}\label{IBP2}
\begin{aligned}
\int_{\varGamma _{ X}}F(\gamma)\mypp \nabla_{\myn V}^{\varGamma
}\mypp G(\gamma )\,\mucl(\rd\gamma ) =&-\int_{\varGamma _{
X}}G(\gamma )\myp\nabla_{\myn V}^{\varGamma }\myp F(\gamma
)\,\mucl(\rd\gamma
)\\
&-\int_{\varGamma _{ X}}F(\gamma )\mypp G(\gamma )\mypp
B_{\mucl}^{\myp V}(\gamma) \,\mucl(\rd\gamma ).
\end{aligned}
\end{equation}
\end{theorem}

\proof The result readily follows from Theorem \ref{IBP-} and
linearity of the right-hand side of (\ref{logder-sigma}) with
respect to $v$.
\endproof

\begin{remark}\label{rem:B}
\normalfont An explicit form of $B_{\mu _{{\mathrm{cl}}}}^{V}$ is
not known (cf.\ Remarks \ref{rem:I} and \ref{rem:R}).
\end{remark}

\begin{remark}\normalfont
The logarithmic derivative $B_{\mucl}^{V}$ can be represented in the
form $B_{\mucl}^{V}={\mathcal{I}}^{*}
B_{\pi_{\lambda^{\star}}}^{{\mathcal{I}}V}$, where $B_{\pi
_{\lambda^{\star}}}^{{\mathcal{I}}V}$ is the logarithmic derivative
of $\pi _{\lambda^{\star}}$ along the vector field
${\mathcal{I}}\myp V(\bar{\gamma}) :=V(\mathfrak{p}(\bar{\gamma
}))$. Note that the equality
\begin{align*}
T_{\bar{\gamma }}\varGamma _{\frakX}& =\bigoplus_{\bar{x}\in
\bar{\gamma }}T_{\bar{x}}\frakX=\bigoplus_{\bar{x}\in \bar{\gamma }}
\bigoplus_{x_i\in \bar{x}}T_{x_i} X =\bigoplus_{x\in
\mathfrak{p}(\bar{\gamma })}T_{x} X=T_{\mathfrak{p}(\bar{\gamma
})}\varGamma _{ X}
\end{align*}
implies that $V(\mathfrak{p}(\bar{\gamma}))\in T_{\bar{\gamma
}}\varGamma_{\frakX}$, and thus ${\mathcal{I}}\myp V(\bar{\gamma })$
is a vector field on $\varGamma _{\mathfrak{X}}$.
\end{remark}

\section{Dirichlet forms and equilibrium stochastic
dynamics}\label{sec:Dirichlet}

In this section, we construct a Dirichlet form $\mathcal{E}_{\mucl}$
associated with the Poisson cluster measure $\mucl$ and prove the
existence of the corresponding equilibrium stochastic dynamics on
the configuration space. We also show that the Dirichlet form
$\mathcal{E}_{\mucl}$ is irreducible. We assume throughout that the
measure $\lambda^\star$ satisfies all the conditions set out at the
beginning of Section \ref{sec:QI-IBP} and in Section~\ref{sec:4.2}.

\subsection{The Dirichlet form associated with $\mucl$}\label{sec:Dir-mu}

Let us introduce the pre-Dirichlet form $\mathcal{E}_{\mucl}$
associated with the Poisson cluster measure $\mucl$, defined on
$\mathcal{FC}(\varGamma _{ X})\subset L^{2}(\varGamma_{X},\mucl)$ by
\begin{equation}\label{eq:E-mu}
\mathcal{E}_{\mucl}(F,G):=\int_{\varGamma _{ X}}\langle
\nabla^{\varGamma\myn}F(\gamma ),\nabla^{\varGamma\myn} G (\gamma
)\rangle _{\gamma }\:\mucl(\rd\gamma ),\qquad
F,G\in\mathcal{FC}(\varGamma _{X}),
\end{equation}
where $\nabla^{\varGamma }$ is the $\varGamma$-gradient on the
configuration space $\varGamma _{X}$ (see~(\ref{eq:G-gradient})).
The next proposition shows that the form $\mathcal{E}_{\mucl}$ is
well defined.

\begin{proposition}
For any $F,G\in \mathcal{FC}(\varGamma _{ X})$, we have
$\mathcal{E}_{\mucl}(F,G)<\infty $.
\end{proposition}

\proof The statement follows from the existence of the first moments
of $\mucl$. Indeed, let $F,G\in \mathcal{FC}(\varGamma _{ X})$ have
representations
\begin{align*}
F(\gamma )=f(\langle \phi_{1},\gamma \rangle ,\dots ,\langle
\phi_{k},\gamma \rangle ),\qquad G(\gamma )=g(\langle
\psi_{1},\gamma \rangle ,\dots ,\langle \psi _{\ell},\gamma \rangle)
\end{align*}
(see~(\ref{local-funct})), then a direct calculation shows that
\begin{equation*}
\langle \nabla^{\varGamma\myn}F(\gamma ),\nabla^{\varGamma}\myn
G(\gamma )\rangle _{\gamma }=\sum_{x\in \gamma }\nabla _{\mynn
x}F(\gamma )\CD \nabla _{\myn x}\myp G(\gamma )=\sum_{i,j}
Q_{ij}(\gamma )\myp\langle q_{ij},\gamma\rangle,
\end{equation*}
where $q_{ij}(x):=\nabla \phi_{i}(x)\myn\CD \nabla \psi_{j}(x)\in
C_{0}( X)$ and
\begin{equation*} Q_{ij}(\gamma
):=\nabla_{\myn i}f(\langle \phi_{1},\gamma\rangle,\dots,\langle
\phi_{k},\gamma\rangle)\,\nabla_{\mynn j}\mypp g(\langle
\psi_{1},\gamma\rangle,\dots,\langle \psi_{\ell},\gamma\rangle)\in
\mathcal{FC}(\varGamma_{X}).
\end{equation*}
Denoting for brevity $q(x):=q_{ij}(x)$ and setting
$\tilde{q}(\bar{x}):=\sum_{x\in \bar{x}} q(x)$, by Theorem
\ref{th:mucl} we have
\begin{align*}
\int_{\varGamma _{ X}} \langle
q,\gamma\rangle\,\mucl(\rd\gamma)&=\int_{\varGamma _{\frakX}}\langle
q, \mathfrak{p}(\bar{\gamma})\rangle\,\pi_{\lambda^{\star}}(\rd\bar{\gamma })\\
&=\int_{\varGamma _{\frakX}}\langle\tilde{q},\bar{\gamma
}\rangle\,\pi _{\lambda^{\star}}(\rd\bar{\gamma
})=\int_{\frakX}\tilde{q}(\bar{y})\,\lambda^{\star}(\rd\bar{y})<\infty,
\end{align*}
because $\lambda^{\star}(\supp\tilde{q}\myp)
=\lambda^{\star}(\frakX_{\supp q})<\infty$ by Proposition
\ref{prop1}. Therefore, $\langle q, \gamma\rangle\in L^{1}(\varGamma
_{ X},\mucl)$ and the required result follows.
\endproof

Let us also consider the pre-Dirichlet form
$\mathcal{E}_{\pi_{\lambda^{\star}}}$ associated with the Poisson
measure $\pi_{\lambda^{\star}}$, defined on the space
$\mathcal{FC}(\varGamma_{\frakX})\subset
L^{2}(\varGamma_{\frakX},\pi_{\lambda^{\star}})$ by
\begin{equation*}
\mathcal{E}_{\pi
_{\lambda^{\star}}}(\varPhi,\varPsi):=\int_{\varGamma_{\frakX}}
\langle\nabla^{\varGamma\myn}\varPhi(\bar{\gamma
}),\nabla^{\varGamma} \varPsi(\bar{\gamma})\rangle_{\bar{\gamma}}
\:\pi_{\lambda^{\star}}(\rd\bar{\gamma}), \qquad
\varPhi,\varPsi\in\mathcal{FC}(\varGamma_{\frakX})
\end{equation*}
(here $\nabla^{\varGamma}$ is the $\varGamma$-gradient on the
configuration space $\varGamma_{\frakX}$,
cf.~(\ref{eq:G-gradient})). Pre-Dirichlet forms of such type
associated with general Poisson measures were introduced and studied
in \cite{AKR1}. Finiteness of the first moments of the Poisson
measure $\pi _{\lambda^{\star}}$ implies that $\mathcal{E}_{\pi
_{\lambda^{\star}}}$ is well defined. It follows from the IBP
formula for $\pi _{\lambda^{\star}}$ that
\begin{equation}\label{gen-poisson}
\mathcal{E}_{\pi
_{\lambda^{\star}}}(\varPhi,\varPsi)=\int_{\varGamma
_{\frakX}}H_{\pi _{\lambda^{\star}}}\varPhi(\bar{\gamma })\,
\varPsi(\bar{\gamma })\,\pi _{\lambda^{\star}}(\rd\bar{\gamma
}),\qquad \varPhi,\varPsi\in \mathcal{FC}(\varGamma _{\frakX}),
\end{equation}
where $H_{\pi _{\lambda^{\star}}}$ is a symmetric non-negative
operator in $L^{2}(\varGamma_{\frakX},\pi_{\lambda^{\star}})$
(called the Dirichlet operator of the Poisson measure $\pi
_{\lambda^{\star}}$, see \cite{AKR1}) defined on the domain
$\mathcal{FC}(\varGamma _{\frakX})$ by
\begin{equation}\label{gen-poisson1}
(H_{\pi_{\lambda^{\star}}}\varPhi)(\bar{\gamma }):=-\sum_{\bar{x}\in
\bar{\gamma }}\bigl(\Delta _{\bar{x}}\mypp\varPhi(\bar{\gamma
})+\nabla_{\myn\bar{x}}\mypp\varPhi(\bar{\gamma})\myn\CD
\beta_{\lambda^{\star}}(\myp\bar{x})\bigr)\qquad (\bar{\gamma}\in
\varGamma_{\frakX}).
\end{equation}
Since function $\varPhi\in \mathcal{FC}(\varGamma _{\frakX})$ is
local  (see Section~\ref{app2}), there are only finitely many
non-zero terms in the sum (\ref{gen-poisson1}).

\begin{remark}\normalfont
Note that the operator $H_{\pi _{\lambda^{\star}}}$ is well defined
by formula (\ref{gen-poisson1}) on the bigger space
$\mathcal{FC}_{\lambda^{\star}}(\varGamma _{\frakX})$. Similar
arguments as before show that the pre-Dirichlet form
$\mathcal{E}_{\pi _{\lambda^{\star}}}(\varPhi,\varPsi)$ is well
defined on $\mathcal{FC}_{\lambda^{\star}}(\varGamma _{\frakX})$ and
formula (\ref{gen-poisson}) holds for any $\varPhi,\varPsi\in
\mathcal{FC}_{\lambda^{\star}}(\varGamma _{\frakX})$.
\end{remark}

Consider a symmetric operator in $L^{2}(\varGamma _{X},\mucl)$
defined on $\mathcal{FC}(\varGamma _{ X})$ by the formula
\begin{equation}\label{eq:H-mu}
H_{\mucl}:={\mathcal{I}}^{*}H_{\pi _{\lambda^{\star}}}{\mathcal{I}}.
\end{equation}
Note that the domain  $\mathcal{FC}(\varGamma _{ X})$ is dense in
$L^{2}(\varGamma _{X},\mucl)$.
\begin{theorem}\label{th:Hcl}
For any $F,G\in \mathcal{FC}(\varGamma _{ X})$, the form
\textup{(\ref{eq:E-mu})} satisfies the equality
\begin{equation}\label{generator}
\mathcal{E}_{\mucl}(F,G)=\int_{\varGamma _{ X}}H_{\mucl}F(\gamma)
\,G(\gamma )\,\mucl(\rd\gamma ).
\end{equation}
In particular, this implies that $H_{\mucl}$ is a non-negative
operator on $\mathcal{FC}(\varGamma _{ X})$.
\end{theorem}

\proof Let us fix $F,G\in \mathcal{FC}(\varGamma _{ X})$ and set
$Q(\gamma ):=\langle \nabla^{\varGamma }F(\gamma ),\nabla^{\varGamma
}G(\gamma )\rangle _{\gamma }$. From the definition (\ref{eq:I}) of
the operator ${\mathcal{I}}$, it readily follows that
\begin{equation}\label{eq:IF}
({\mathcal{I}}Q) (\bar{\gamma }) =\sum_{x\in
\mathfrak{p}(\bar{\gamma})} \nabla _{\mynn
x}\myp{\mathcal{I}}F(\bar{\gamma })\CD \nabla _{\mynn x}
\myp{\mathcal{I}}\myp G(\bar{\gamma }) =\sum_{\bar{x}\in \bar{\gamma
}} \nabla _{\mynn \bar{x}}\myp{\mathcal{I}}F(\bar{\gamma})\CD
\nabla_{\mynn \bar{x}}\myp{\mathcal{I}}\myp G(\bar{\gamma }),
\end{equation}
where $\nabla _{\mynn \bar{x}}:=(\nabla _{\mynn x_{1}},\dots,\nabla
_{\mynn x_{n}})$ when $\bar{x}=(x_{1},\dots,x_{n})\in X^n$
\,($n\in\NN$). Thus, by Theorem \ref{th:mucl} and formulas
(\ref{eq:I}) and (\ref{eq:IF}) we obtain
\begin{align}
\mathcal{E}_{\mucl}(F,G)&=\int_{\varGamma _{ X}}Q(\gamma
)\,\mucl(\rd\gamma )=\int_{\varGamma
_{\frakX}}({\mathcal{I}}Q)(\bar{\gamma })\,\pi
_{\lambda^{\star}}(\rd\bar{\gamma })  \notag \\
&=\int_{\varGamma _{\frakX}}\sum_{\bar{x}\in \bar{\gamma }}\nabla
_{\mynn \bar{x}}\myp{\mathcal{I}}F(\bar{\gamma })\CD \nabla _{\mynn
\bar{x}}\myp{\mathcal{I}}\myp G(\bar{\gamma })\,\pi
_{\lambda^{\star}}(\rd\bar{\gamma })=\mathcal{E}_{\pi
_{\lambda^{\star}}}({\mathcal{I}}F,{\mathcal{I}}\myp G)
\label{df-corr}
\end{align}
(note that ${\mathcal{I}}F,{\mathcal{I}}\myp G\in
\mathcal{FC}_{\lambda^{\star}}(\varGamma_{\frakX})\subset
\mathcal{D}(\mathcal{E}_{\pi _{\lambda^{\star}}})$). Finally,
combining (\ref{df-corr}) with formula (\ref{gen-poisson}) we get
(\ref{generator}).
\endproof

\begin{remark}\normalfont
The operator $H_{\mucl}$ defined in (\ref{eq:H-mu}) can be
represented in the following form separating its diffusive and drift
parts:
\begin{equation}\label{generator1}
(H_{\mucl}F)(\gamma )=-\sum_{x\in \gamma }\Delta _{x}F(\gamma)-
({{\mathcal{I}}}^{*}\varPsi_F) (\gamma ), \qquad F\in
\mathcal{FC}(\varGamma _{ X}),
\end{equation}
where \myp$\varPsi_F(\bar {\gamma }):=\sum _{\bar {x}\in \bar
{\gamma }} \nabla _{\myn \bar {x}}\mypp{{\mathcal{I}}}F(\bar {\gamma
})\myn\CD \beta _{\lambda ^{\star }}(\myp\bar {x})$
\,($\bar{\gamma}\in\varGamma_{\frakX}$).
\end{remark}

\begin{remark}\normalfont
Formulas (\ref{generator}) and (\ref{generator1}) can also be
obtained directly from the IBP formula (\ref{IBP2}).
\end{remark}

\subsection{The associated equilibrium stochastic dynamics}\label{sec:equil}

Formula (\ref{generator}) implies that the form
$\mathcal{E}_{\mucl}$ is closable on $L^{2}(\varGamma _{ X},\mucl)$,
and we preserve the same notation for its closure. Its domain
$\mathcal{D}(\mathcal{E}_{\mucl})$ is obtained as a completion of
$\mathcal{FC}(\varGamma _{ X})$ with respect to the norm
\begin{equation*}
\left\| F\right\| _{\mathcal{E}_{\mucl}\!}:=\left(
\mathcal{E}_{\mucl}(F,F)+\int_{\varGamma _{ X}}F^{2}\,\rd
\mucl\right)^{1/2}.
\end{equation*}
In the canonical way, the Dirichlet form
$(\mathcal{E}_{\mucl},\mathcal{D}(\mathcal{E}_{\mucl}))$ defines a
non-negative self-adjoint operator in $L^2(\varGamma_X, \mucl)$
(i.e., the Friedrichs extension of
$H_{\mucl}={\mathcal{I}}^{*}H_{\pi _{\lambda^{\star}}}{\mathcal{I}}$
from the domain $\mathcal{FC}(\varGamma _{ X})$), for which we keep
the same notation $H_{\mucl}$. In turn, this operator generates the
semigroup $\exp(-tH_{\mucl})$ in $L^2(\varGamma_X, \mucl)$.

According to a general result (see \cite[\S\,4]{MR}), it follows
that \,$\mathcal{E}_{\mucl}$ is a quasi-regular local Dirichlet form
on a bigger space $L^{2}(\overset{\,\myp..}{\varGamma }_{
X},\mucl)$, where $\overset{\,..}{\varGamma }_{ X}$ is the space of
all locally finite configurations $\gamma$ with possible multiple
points (note that $\overset{\,..}{\varGamma }_{ X}$ can be
identified in the standard way with the space of
$\mathbb{Z}_{+}$-valued Radon measures on $X$, cf.\
\cite{AKR1,MR,RS}). Then, by the general theory of Dirichlet forms
(see \cite{MR0}), we obtain the following result.
%[Ch.\,\myp4, 5]

\begin{theorem}\label{th:7.2}
There exists a conservative diffusion process
$\mathbf{X}=(\mathbf{X}_t,\,t\ge0)$ on
$\overset{\,\myp..}{\varGamma}_{X}$, properly associated with the
Dirichlet form $\mathcal{E}_{\mucl}$\textup{;} that is, for any
function $F\in L^{2}(\overset{\,\myp..}{\varGamma }_{ X},\mucl)$ and
all\/ $t\ge0$, the mapping
\begin{equation*}
\overset{\,\myp..}{\varGamma}_{X}\ni \gamma \mapsto p_{t}F(\gamma)
:=\int_{\varOmega} F(\mathbf{X}_{t})\,\rd P_{\gamma}
\end{equation*}
is an\/ $\mathcal{E}_{\mucl}$-quasi-continuous version of\/
$\exp(-tH_{\mucl}) F$. Here $\varOmega$ is the canonical sample
space \textup{(}of $\overset{\,\myp..}{\varGamma}_X$-valued
continuous functions on $\mathbb{R}_+$\textup{)} and
$(P_\gamma,\,\gamma\in\overset{\,\myp..}{\varGamma}_X)$ is the
family of probability distributions of the process $\mathbf{X}$
conditioned on the initial value $\gamma=\mathbf{X}_0$. The process
$\mathbf{X}$ is unique up to $\mucl$-equivalence. In particular,
$\mathbf{X}$ is $\mucl$-symmetric \textup{(}i.e., $\int F\myp
p_{t}G\,\rd\mucl = \int G\myp p_{t} F\,\rd\mucl$ for all measurable
functions $F,G:\overset{\,\myp..}{\varGamma}_{
X}\to\mathbb{R}_{+}$\textup{)} and $\mucl$ is its invariant measure.
\end{theorem}

\begin{remark}\normalfont
It can be proved that in the case of Poisson and Gibbs measures,
under certain technical conditions the diffusion process
$\mathbf{X}$ actually lives on the proper configuration space
$\varGamma_X$ (see \cite{RS}). It is plausible that a similar result
should be valid for the Poisson cluster measure, but this is an open
problem.
\end{remark}

\begin{remark}\normalfont
Formula (\ref{gen-poisson}) implies that the ``pre-projection'' form
$\mathcal{E}_{\pi _{\lambda^{\star}}}$ is closable. According to the
general theory of Dirichlet forms \cite{MR0,MR}, its closure is a
quasi-regular local Dirichlet form on $\overset{\,..}{\varGamma
}_{\frakX}$ and as such generates a diffusion process
$\bar{\mathbf{X}}$ on $\overset{\,..}{\varGamma}_{\frakX}$. This
process coincides with the independent infinite particle process,
which amounts to independent distorted Brownian motions in $\frakX$
with drift given by the vector logarithmic derivative of $\lambda $
(see~\cite{AKR1}). However, it is not clear in what sense the
process $\mathbf{X}$ constructed in Theorem \ref{th:7.2} can be
obtained directly via the projection of $\bar{\mathbf{X}}$ from
$\overset{\,..}{\varGamma }_{\frakX}$ onto $\overset{\,..}{\varGamma
}_{ X}$.
\end{remark}

\subsection{Irreducibility of the Dirichlet form
$\mathcal{E}_{\mucl}$}\label{sec:irreduc}

Let us recall that a Dirichlet form $\mathcal{E}$ is called
\emph{irreducible} if the condition $\mathcal{E}(F,F)\allowbreak=0$
implies that $F=\const$.

\begin{theorem}\label{th:irr}
The Dirichlet form
$(\mathcal{E}_{\mucl},\mathcal{D}(\mathcal{E}_{\mucl}))$ is
irreducible.
\end{theorem}

\proof For any $F\in\mathcal{D}(\mathcal{E}_{\mucl})$, we have
\begin{align*}
\left\| F\right\| _{\mathcal{E}_{\mucl}}^{2}
&=\mathcal{E}_{\mucl}(F,F)+\int_{\varGamma _{ X}}F^{2}\,\rd\mucl \\
&=\mathcal{E}_{\pi
_{\lambda^{\star}}}({\mathcal{I}}F,{\mathcal{I}}F)+\int_{\varGamma
_{ X}} ({\mathcal{I}}F)^{2}\,\rd\pi_{\lambda^{\star}}=
\|{\mathcal{I}}F\|_{\mathcal{E}_{\pi _{\lambda^{\star}}}}^{2},
\end{align*}
which implies that
${\mathcal{I}}\mathcal{D}(\mathcal{E}_{\mucl})\subset
\mathcal{D}(\mathcal{E}_{\pi _{\lambda^{\star}}})$. It is obvious
that if ${\mathcal{I}}F=\const$ ($\pi _{\lambda^{\star}}$-a.s.) then
$F=\const$ ($\mucl$-a.s.). Therefore, according to formula
(\ref{df-corr}), it suffices to prove that the Dirichlet form
$(\mathcal{E}_{\pi _{\lambda^{\star}}},\mathcal{D}(\mathcal{E}_{\pi
_{\lambda^{\star}}}))$ is irreducible, which is established in Lemma
\ref{lm:irred} below.
\endproof

We first need the following general result (see
\cite[Lemma~3.3]{ADL}).

\begin{lemma}\label{kernel-of-sum}
Let $A$ and $B$ be self-adjoint, non-negative operators in separable
Hilbert spaces $\mathcal{H}$ and $\mathcal{K}$, respectively. Then
$\Ker(A\boxplus B)=\Ker A\otimes \Ker B$, where $A\boxplus B$ is the
closure of the operator $A\otimes I+I\otimes B$ from the algebraic
tensor product of the domains of $A$ and $B$.
\end{lemma}

\proof $\Ker A$ and $\Ker B$ are closed subspaces of $\mathcal{H}$
and $\mathcal{K}$, respectively, and so their tensor product $\Ker
A\otimes \Ker B$ is a closed subspace of the space
$\mathcal{H}\otimes \mathcal{K}$. The inclusion $\Ker A\otimes \Ker
B\subset \Ker (A\boxplus B)$ is trivial. Let $f\in \Ker (A\boxplus
B)$. Using the theory of operators admitting separation of variables
(see, e.g., \cite[Ch.~6]{Ber}), we have
\begin{align}
0=(A\boxplus Bf,f)&
=\int_{\mathbb{R}_{+}^{2}}(x_{1}+x_{2})\,\rd(E(x_{1},x_{2})f,f)
\notag \\
&=\int_{\mathbb{R}_{+}^{2}}x_{1}\,\rd(E(x_{1},x_{2})f,f)+\int_{\mathbb{R}
_{+}^{2}}x_{2}\,\rd(E(x_{1},x_{2})f,f)  \notag \\
& =(A\otimes If,f)+(I\otimes Bf,f),  \label{dderuz}
\end{align}
where $E$ is a joint resolution of the identity of the commuting
operators $A\otimes I$ and $I\otimes B$. Since both operators
$A\otimes I$ and $I\otimes B$ are non-negative, we conclude from
(\ref{dderuz}) that
\begin{equation*}
f\in \Ker (A\otimes I)\cap \Ker (I\otimes B)=\Ker A\otimes \Ker B,
\end{equation*}
which completes the proof of the lemma.
\endproof

\begin{lemma}\label{lm:irred}
The Dirichlet form $(\mathcal{E}_{\pi
_{\lambda^{\star}}},\mathcal{D}(\mathcal{E}_{\pi_{\lambda^\star}}))$
is irreducible.
\end{lemma}

\begin{remark}\normalfont
Irreducibility of Dirichlet forms
associated with Poisson measures on configuration spaces of
connected Riemannian manifolds was shown in \cite{AKR1}. However,
the space $\frakX$ consists of countably many disjoint connected
components $ X^{n}$, so we need to adapt the result of \cite{AKR1}
to this situation.
\end{remark}

\begin{proof}[Proof of Lemma \textup{\ref{lm:irred}}]
Let us recall that, according to the general theory (see, e.g.,
\cite{AKR97}), irreducibility of a Dirichlet form is equivalent to
the condition that the kernel of its generator consists of constants
(\emph{uniqueness of the ground state}). Thus, it suffices to prove
that $\Ker H_{\pi _{\lambda^{\star}}}=\{\const\}$.

Let us consider the ``residual'' spaces
$\tilde{\mathfrak{X}}_{n}:=\bigsqcup_{k=n}^{\infty} X^{k}$,
\,$n\in\mathbb{Z}_+$\myp, endowed with the measures
$\tilde\lambda^\star_{n}:=\sum_{k=n}^{\infty}
p_{k}\myp\lambda^\star_{k}$\myp. Hence, $\mathfrak{X}= X^0 \sqcup
X^{1} \sqcup \cdots\sqcup X^{n}\sqcup \tilde{\mathfrak{X}}_{n+1}$,
which implies that $\varGamma _{\frakX}=\varGamma _{X^0}\times
\varGamma _{X^1}\times \dots\times \varGamma _{ X^{n}}\times
\varGamma _{\tilde{\mathfrak{X}}_{n+1}}$ and, according to
Proposition \ref{pr:product}, $\pi _{\lambda^{\star}}= \pi_0\otimes
\pi_1\otimes \dots\otimes \pi _{n}\otimes \tilde\pi _{n+1}$\myp,
where we use a shorthand notation $\pi_n:= \pi
_{p_{n}\lambda^\star_{n}}$, \,$\tilde\pi_n:=\pi
_{\tilde\lambda^\star_{n}}$. Therefore, there is an isomorphism of
Hilbert spaces
$$
L^{2}(\varGamma _{\frakX},\pi _{\lambda^{\star}})\cong
L^{2}(\varGamma _{ X},\pi_{1})\otimes \cdots \otimes L^{2}(\varGamma
_{ X^{n}},\pi _{n})\otimes L^{2}(\varGamma _{\frakX_{n+1}},\tilde
\pi _{n+1}).
$$
Consequently, the Dirichlet operator $H_{\pi _{\lambda^{\star}}}$
can be decomposed as
\begin{equation}\label{eq:box}
H_{\pi _{\lambda^{\star}}}= H_{\pi _{1}}\mynn\boxplus \dots\boxplus
H_{\pi_{n}}\myn\boxplus H_{\tilde\pi_{n+1}}.
\end{equation}
Since all operators on the right-hand side of (\ref{eq:box}) are
self-adjoint and non-negative, it follows by Lemma
\ref{kernel-of-sum} that
\begin{equation}\label{eq:Ker}
\Ker H_{\pi _{\lambda^{\star}}\!}= \Ker H_{\pi_{1}}\mynn\otimes
\dots\otimes \Ker H_{\pi_{n}}\myn\otimes \Ker H_{\tilde\pi_{n+1}}.
\end{equation}
The Dirichlet forms of all measures $\pi_{k}$ are irreducible (as
Dirichlet forms of Poisson measures on connected manifolds), hence
$\Ker H_{\pi_{k}}=\mathbb{R}$ and (\ref{eq:Ker}) implies that $\Ker
H_{\pi _{\lambda^{\star}}\!}=\Ker H_{\tilde\pi _{n+1}}$. Since $n$
is arbitrary, it follows that every function $F\in \Ker
H_{\pi_{\lambda^{\star}}}$ does not depend on any finite number of
variables, and thus $F=\const$ ($\pi _{\lambda^{\star}}$-a.s.).
\end{proof}

\begin{remark}\normalfont
The result of Lemma \ref{lm:irred} (and the idea of its proof) can
be viewed as a functional-analytic analogue of Kolmogorov's
zero\myp--\myp{}one law (see, e.g., \cite[Ch.~3]{Kal}), stating that
for a sequence of independent random variables $(X_n)$, the
corresponding tail $\sigma$-algebra
$\mathcal{F}_\infty:=\bigcap_{n}\mathcal{F}_{\ge n}$ is trivial
(where $\mathcal{F}_{\ge n}:=\sigma\{X_k: k\ge n\}$), and in
particular, all $\mathcal{F}_\infty$-measurable random variables are
a.s.\ constants.
\end{remark}

\begin{remark}\normalfont
According to the general theory of Dirichlet forms (see, e.g.,
\cite{AKR97}), the irreducibility of\/ $\mathcal{E}_{\mucl}$ is
equivalent to each of the following properties:
\begin{enumerate}
\item[\rm (i)]
\textit{The semigroup}\/ $\re^{-tH_{\mucl}}$ \textit{is
$L^{2}$-ergodic, that is, as $t\rightarrow \infty $},
\begin{equation*}
\int_{\varGamma _{ X}}\left( \re^{-tH_{\mucl}}F(\gamma
)-\int_{\varGamma _{ X}}F(\gamma )\,\mucl(\rd\gamma )\right)
^{2}\mucl(\rd\gamma )\rightarrow 0\myp.
\end{equation*}

\item[\rm (ii)]
\textit{If\/ $F\in \mathcal{D}(H_{\mucl})$ and $H_{\mucl}F=0$ then
$F=\const$.}
\end{enumerate}
\end{remark}

\section{Appendix}

\subsection{Proof of Theorem \ref{th:properClusterPoisson}}\label{app1}

Note that the droplet cluster
$D_B(\gamma_0^{\myp\prime})=\bigcup_{y\in\gamma_0^{\myp\prime}}
(B-y)$ (see~(\ref{eq:D})) can be decomposed into disjoint components
according to the number of constituent ``layers'' (including
infinitely many):
$$
D_B(\gamma_0^{\myp\prime})=\bigcup_{1\le \myp\ell\myp\le \infty}
D_B^{\myp\ell}(\gamma_0^{\myp\prime}),
$$
where
$$
D_B^{\myp\ell}(\gamma_0^{\myp\prime}):=\{x\in X:
\,\gamma_0^{\myp\prime}(B-x)=\ell\},\qquad \ell\in\overline{\ZZ}_+.
$$

(a) Set $f_q:=-\ln q\cdot{\bf 1}_K\in\mathrm{M}_+(X)$ \,($0<q<1$),
then
\begin{align}\label{eq:expansion}
L_{\mucl}[f_q]=\int_{\varGamma
_{X}^\sharp}q^{\gamma(K)}\,{\mucl}(\rd\gamma) &=\sum_{n=0}^\infty
q^{n}\myp\mucl\{\gamma\in\varGamma_X^\sharp:\gamma(K)=n\}\\
\notag &\to
\mucl\{\gamma\in\varGamma_X^\sharp:\gamma(K)<\infty\}\qquad
(q\uparrow1).
\end{align}
Therefore, $\gamma(K)<\infty$ ($\mucl$-a.s.) if and only if
$\lim_{q\uparrow 1} \ln L_{\mucl}[f_q]=0$.

Clearly, condition (\ref{eq:condA1}) is necessary for local
finiteness of $\mucl$-a.a.\ configurations
$\gamma\in\varGamma_X^\sharp$. Furthermore, (\ref{eq:condA1})
implies that, for any compact set $K\subset X$ and any $x\in X$, we
have $\gamma_0^{\myp\prime}(K-x)<\infty$ ($\mu_0$-a.s.). Hence,
according to (\ref{eq:ClusterPoissonLT}),
\begin{align}
\notag -\ln L_{\mucl}[f_q]&= \int_{X}
\biggl(\int_{\varGamma^\sharp_{ X}}
\left(1-q^{\gamma_0^{\myp\prime}(K-x)}\right)
\mu_0(\rd \gamma_0^{\myp\prime})\biggr)\,\lambda(\rd x)\\
\notag &=\int_{\varGamma^\sharp_{ X}} \biggl(\int_{X}
\sum_{\ell\myp=\myp0}^{\infty} (1- q^{\ell}\myp)\,{\bf
1}_{D_K^{\myp\ell}(\gamma_0^{\myp\prime})}(x)\,
\lambda(\rd x)\biggr)\,\mu_0(\rd \gamma_0^{\myp\prime})\\
\label{eq:upper} &=\int_{\varGamma^\sharp_{ X}}
\sum_{\ell\myp=1}^\infty
(1-q^{\ell}\myp)\,\lambda\bigl(D_K^{\myp\ell}(\gamma_0^{\myp\prime})\bigr)\,\mu_0(\rd
\gamma_0^{\myp\prime}).
\end{align}
Note that, for $0<q<1$,
$$
0\le\sum_{\ell\myp=1}^{\infty} (1-
q^{\ell}\myp)\,\lambda\bigl(D_K^{\myp\ell}(\gamma_0^{\myp\prime})\bigr)\le
\sum_{\ell\myp=1}^{\infty}\lambda\bigl(D_K^{\myp\ell}(\gamma_0^{\myp\prime})\bigr)
=\lambda\bigl(D_K(\gamma_0^{\myp\prime})\bigr),
$$
so if condition (\ref{eq:condA2}) is satisfied then we can apply
Lebesgue's dominated convergence theorem and pass termwise to the
limit on the right-hand side of (\ref{eq:upper}) as $q\uparrow 1$,
which gives $\lim_{q\uparrow 1}\ln L_{\mucl}[f_q]=0$, as required.

Conversely, since
\begin{align*}
\sum_{\ell\myp=1}^{\infty} (1-
q^{\ell})\,\lambda\bigl(D_K^{\myp\ell}(\gamma_0^{\myp\prime})\bigr)&\ge
(1- q)\sum_{\ell\myp=1}^\infty
\lambda\bigl(D_K^{\myp\ell}(\gamma_0^{\myp\prime})\bigr)=(1-
q)\,\lambda\bigl(D_K(\gamma_0^{\myp\prime})\bigr)\ge0,
\end{align*}
from (\ref{eq:upper}) we must have
$$
(1-q)\int_{\varGamma^\sharp_{X}}
\lambda\bigl(D_K(\gamma_0^{\myp\prime})\bigr)\,\mu_0(\rd
\gamma_0^{\myp\prime})\to 0\qquad (q\uparrow 1),
$$
which implies (\ref{eq:condA2}).

(b) Let us first prove the ``only if'' part. Clearly, condition
(\ref{eq:condB1}) is necessary in order to avoid any in-cluster
ties. Furthermore, each fixed $x_0\in X$ cannot belong to more than
one cluster; in particular, for any $2\le \ell\le \infty$,
\begin{equation}\label{eq:D=0}
\lambda\left(D_{\{x_0\}}^{\myp\ell}(\gamma_0^{\myp\prime})\right)=0\qquad
(\mu_0\text{-a.s.})
\end{equation}
Let $f_q:=-\ln q\cdot {\bf 1}_{\{x_0\}}$ $(0<q<1)$. The expansion
(\ref{eq:expansion}) then implies that in order for $x_0$ to be
simple ($\mucl$-a.s.), $L_{\mucl}[f_q]$ must be a linear function of
$q$. But from (\ref{eq:upper}) and (\ref{eq:D=0}) we have
\begin{equation*}
L_{\mucl}[f_q]=\exp\Bigl\{-(1-q)\int_{\varGamma_ X^\sharp}
\lambda\bigl(D_{\{x_0\}}^{\myp\ell=1}(\gamma_0^{\myp\prime})\bigr)\mypp
\mu_0(\rd\gamma_0^{\myp\prime})\Bigr\},
\end{equation*}
and it follows that
$\lambda\bigl(D_{\{x_0\}}^{\myp\ell=1}(\gamma_0^{\myp\prime})\bigr)=0$
\,($\mu_0$-a.s.). Together with (\ref{eq:D=0}), this gives
$$
\lambda\bigl(D_{\{x_0\}}(\gamma_0^{\myp\prime})\bigr)=\sum_{1\le\myp
\ell\myp\le\infty}\lambda\bigl(D_{\{x_0\}}^{\myp\ell}
(\gamma_0^{\myp\prime})\bigr)=0\qquad (\mu_0\text{-a.s.}),
$$
and condition (\ref{eq:condB2}) follows.

To prove the ``if'' part, it suffices to show that, under conditions
(\ref{eq:condB1}) and (\ref{eq:condB2}), with probability one there
are no cross-ties between the clusters whose centres belong to a set
$\varLambda\subset X$, $\lambda(\varLambda)<\infty$. Conditionally
on the total number of cluster centres  in $\varLambda$ (which are
then i.i.d.\ and have the distribution
$\lambda(\cdot)/\lambda(\varLambda)$), the probability of a tie
between a given pair of (independent) clusters is given by
\begin{align*}
\frac{1}{\lambda(\varLambda)^2}\int_{\varGamma_
X^\sharp\times\varGamma_ X^\sharp}
\lambda^{\otimes\myp2}\bigl(B_\varLambda(\gamma_1,\gamma_2)\bigr)\,
\mu_0(\rd\gamma_1)\,\mu_0(\rd\gamma_2),
\end{align*}
where
$$
B_\varLambda(\gamma_1,\gamma_2):=\{(x_1,x_2)\in\varLambda^2:
\,x_1+y_1=x_2+y_2\ \text{\,for some }\, y_1\in\gamma_1,\
y_2\in\gamma_2\}.
$$
But
\begin{align*}
\lambda^{\otimes\myp2}\bigl(B_\varLambda(\gamma_1,\gamma_2)\bigr)
&=\int_\varLambda \lambda
\left(\textstyle\bigcup\nolimits_{y_1\in\gamma_1}
\!\bigcup\nolimits_{y_2\in\gamma_2}
\{x_1+y_1-y_2\}\right)\lambda(\rd
x_1)\\[.2pc]
&\le \sum_{y_1\in\gamma_1}\int_\varLambda \lambda
\left(\textstyle\bigcup\nolimits_{y_2\in\gamma_2}\mynn
\{x_1+y_1-y_2\}\right)\lambda(\rd
x_1)\\
&=\sum_{y_1\in\gamma_1}\int_\varLambda \lambda
\bigl(D_{\{x_1+y_1\}}(\gamma_2)\bigr)\,\lambda(\rd x_1) =0\qquad
(\mu_0\text{-a.s.}),
\end{align*}
since, by assumption (\ref{eq:condB2}),
$\lambda\bigl(D_{\{x_1+y_1\}}(\gamma_2)\bigr)=0$ \,($\mu_0$-a.s.)
and $\gamma_1$ is a countable set. Thus, the proof is complete.

\subsection{Quasi-invariance of Poisson measures} \label{app3}

The next general result is a direct consequence of Skorokhod's
theorem \cite{Sk} on the absolute continuity of Poisson measures
(see also \cite{AKR1}). Although essentially well known, we give its
simple proof adapted to our slightly more general setting, whereby
transformations $\varphi$ have support of finite measure rather than
compact.

Suppose that $\pi _{\lambda}$ is a Poisson measure on the
configuration space $\varGamma_{{X}}$ with intensity measure
$\lambda$. Let $\varphi : {X}\rightarrow {X}$ be a measurable
mapping; as explained earlier (see~(\ref{di*})), it can be lifted to
a (measurable) transformation of $\varGamma_{{X}}$:
\begin{equation}\label{eq:theta}
\varGamma_{{X}}\ni \gamma\mapsto \varphi (\gamma):=\{\varphi(x),\
x\in\gamma\}\in \varGamma_{X}.
\end{equation}

\begin{proposition}\label{q-i-Poisson}
Let $\varphi: X\to X$ be a measurable bijection such that
$\lambda(\supp\varphi)<\infty$. Assume that the measure $\lambda$ is
quasi-invariant with respect to $\varphi$, that is, the push-forward
measure $\varphi^{*}\lambda \equiv\lambda\circ \varphi^{-1}$ is
a.c.\ with respect to $\lambda $, with density
\begin{equation}\label{eq:rr}
\rho _{\lambda }^{\varphi }(x):=\frac{\varphi^{*}\lambda (\rd
x)}{\lambda (\rd x)}\mypp,\qquad x\in {X}.
\end{equation}
Then the measure $\pi_{\lambda}$ is quasi-invariant with respect to
the action \textup{(\ref{eq:theta})}, that is,
\begin{equation}\label{eq:pi-R}
\varphi^{*}\pi _{\lambda }(\rd\gamma )=R_{\pi _{\lambda }}^{\varphi
}(\gamma )\,\pi _{\lambda }(\rd\gamma ),\qquad
\gamma\in\varGamma_{X},
\end{equation}
where the density $R_{\pi _{\lambda }}^{\varphi}$ is given by
\begin{equation}\label{eq:R}
R_{\pi _{\lambda }}^{\varphi }(\gamma )=\exp
\left\{\int_{X}\bigl(1-\rho _{\lambda }^{\varphi }(x)\bigr)\,\lambda
(\rd x)\right\}\cdot\prod_{x\in \gamma }\rho_{\lambda }^{\varphi
}(x),\qquad \gamma\in\varGamma_{X},
\end{equation}
and moreover, $R_{\pi _{\lambda }}^{\varphi }\in
L^{2}(\varGamma_{X},\pi _{\lambda})$.
\end{proposition}

\proof Note that $\rho _{\lambda }^{\varphi }\equiv 1$ outside the
set $K:=\supp \varphi$. By Proposition
\ref{pr:properPoisson}\myp(a), the condition $\lambda(K)<\infty $
implies that, for $\pi _{\lambda }$-a.a. $\gamma\in\varGamma_{{X}}$,
there are only finitely many terms in the product $\prod_{x\in
\gamma }\rho _{\lambda }^{\varphi }(x)$ not equal to $1$, thus the
right-hand side of equation (\ref{eq:R}) is well defined. Using
formulas (\ref{eq:rr}), (\ref{eq:R}) and Proposition
\ref{pr:PoissonLT}, the Laplace functional of the measure
$\pi^{\varphi } _{\lambda }:=R_{\pi _{\lambda }}^{\varphi }\pi
_{\lambda }$ is obtained as follows:
\begin{align*}
L_{\pi^{\varphi }_{\lambda}}[f] &=\exp \left\{\int_{
{X}}\bigl(1-\rho _{\lambda }^{\varphi }(x)\bigr)\,\lambda (\rd
x)\right\}\cdot \int_{\varGamma_{{X}}}\re^{-\langle
f,\gamma\rangle}\prod_{x\in \gamma }\rho _{\lambda
}^{\varphi }(x)\,\pi _{\lambda }(\rd\gamma) \\[.2pc]
&=\exp \left\{ \int_{ {X}}\bigl(1-\rho _{\lambda }^{\varphi
}(x)\bigr)\,\lambda (\rd x)\right\} \cdot \exp\left\{-\int_{
{X}}\left(1-\re^{-f(x)+\ln\rho
_{\lambda}^{\varphi }(x)}\right) \lambda (\rd x)\right\}\\[.2pc]
&=\exp\left\{-\int_{ {X}}\bigl(1-\re^{-f(x)}\bigr)\,\rho _{\lambda
}^{\varphi }(x)\,\lambda (\rd x)\right\}\\[.2pc]
&=\exp\left\{-\int_{{X}}\bigl(1-\re^{-f(x)}\bigr)\,\varphi^*\lambda
(\rd x)\right\} =L_{\pi _{\varphi ^{*}\lambda }}[f],
\end{align*}
and so $\pi^{\varphi }_{\lambda}=\pi _{\varphi ^{*}\lambda }$. But,
according to the Mapping Theorem (see Proposition~\ref{pr:mapping}),
we have $\pi _{\varphi ^{*}\lambda }=\varphi ^{*}\pi _{\lambda }$,
and formula (\ref{eq:pi-R}) follows.

To check that $R_{\pi _{\lambda }}^{\varphi }\in
L^{2}(\varGamma_{{X}},\pi _{\lambda })$, let us compute its
$L^{2}$-norm:
\begin{align*}
\int_{\varGamma_{{X}}} | &R_{\pi _{\lambda }}^{\varphi
}(\gamma)|^{2}\,\pi _{\lambda }(\rd\gamma ) =\exp
\left\{\int_{{X}}\bigl(1-\rho _{\lambda
}^{\varphi}(x)\bigr)\,\lambda(\rd
x)\right\}\cdot\int_{\varGamma_{{X}}}\re^{\langle 2\ln \rho
_{\lambda }^{\varphi },\,\gamma\rangle}
\,\pi _{\lambda }(\rd\gamma) \\
&=\exp \left\{ \int_{ {X}}\bigl(1-\rho _{\lambda}^{\varphi
}(x)\bigr)\,\lambda (\rd x)\right\}\cdot \exp\left\{-\int_{
{X}}\left(1-\re^{2\ln\rho _{\lambda }^{\varphi
}(x)}\right) \lambda (\rd x)\right\} \\[.2pc]
&=\exp\left\{\int_{ {X}}\left(| \rho _{\lambda }^{\varphi }(x)|
^{2}-\rho _{\lambda }^{\varphi }(x)\right) \lambda (\rd
x)\right\}<\infty,
\end{align*}
because $| \rho _{\lambda }^{\varphi }(x)|^{2}-\rho _{\lambda
}^{\varphi }(x)=0$ outside the set $K=\supp \varphi$.
\endproof

\section*{Acknowledgements}
Part of this research was done during the authors' visits to the
Institute of Applied Mathematics of the University of Bonn supported
by SFB~611. Financial support through DFG Grant 436 RUS 113/722 is
gratefully acknowledged. The authors would like to thank Sergio
Albeverio, Yuri Kondratiev and Eugene Lytvynov for useful
discussions. Thanks are also due to the anonymous referee for the
careful reading of the manuscript and valuable comments.
%
%constructive criticism that has led to a significant improvement of
%the paper.

\end{document}